\documentclass[aap,reqno,preprint]{imsart}
\pdfoutput=1
\setattribute{journal}{name}{}
\usepackage{a4wide}

\usepackage{amsthm,amsmath}

\usepackage[numbers]{natbib}
\usepackage{appendix}

\numberwithin{equation}{section}
\allowdisplaybreaks[3]
\theoremstyle{plain} 
\newtheorem{theorem}{Theorem}[section]
\newtheorem{lemma}[theorem]{Lemma}
\newtheorem{proposition}[theorem]{Proposition}
\newtheorem{corollary}[theorem]{Corollary}

\theoremstyle{definition}
\newtheorem{definition}[theorem]{Definition}

\newtheorem{example}[theorem]{Example}

\newtheorem{remark}{Remark}
\newtheorem{problem}{Problem}

\usepackage[utf8]{inputenc}
\usepackage{amssymb, color, mathbbol}
\usepackage[shortlabels]{enumitem}
\usepackage[margin=10pt,font=small,labelfont = bf]{caption}
\usepackage{graphicx}
\usepackage{hyperref}
\usepackage{multirow}
\usepackage{verbatim}
\usepackage[rgb]{xcolor}
	\definecolor{burgundy}{rgb}{0.5, 0.0, 0.13}
\usepackage{makecell}
\usepackage{multicol}
\usepackage{tikz}
\usetikzlibrary{snakes}
\usetikzlibrary {shapes}
\usetikzlibrary {arrows}
\usetikzlibrary {positioning}
\usetikzlibrary {calc}
\usetikzlibrary{quotes}
\usetikzlibrary{fit}
\usetikzlibrary{patterns}
\usetikzlibrary{backgrounds}	

\usepackage{centernot}

\usepackage[noend]{algpseudocode}
\usepackage{algorithm}

\startlocaldefs

\newcommand{\beao}{\begin{eqnarray*}}
\newcommand{\eeao}{\end{eqnarray*}\noindent}

\newcommand{\beam}{\begin{eqnarray}}
\newcommand{\eeam}{\end{eqnarray}\noindent}

\newcommand{\barr}{\begin{array}}
\newcommand{\earr}{\end{array}}

\DeclareMathOperator{\C}{\mathcal{C}}
\DeclareMathOperator{\R}{\mathbb{R}}
\DeclareMathOperator{\E}{\mathcal{E}}

\renewcommand{\P}{\mathbb{P}}

\newcommand{\Rplus}{\R_>}

\newcommand{\ci}{\,{\perp\!\!\!\perp}\,}
\newcommand{\notci}{\centernot{\ci}}

\newcommand{\D}{\mathcal{D}}

\newcommand{\dse}{\perp_{\D}}

\newcommand{\sg}{{\perp_\sigma}}
\newcommand{\starse}{\perp_{*}}
\newcommand{\tildese}{\perp_{\D^\ast}}
\newcommand{\critsep}{\perp_{C^*}}
\newcommand{\cd}{\,|\,}
\renewcommand{\dag}{\mathcal{D}}
\newcommand{\an}{{\rm an}}

\newcommand{\pa}{{\rm pa}}
\newcommand{\ch}{{\rm ch}}

\newcommand{\calx}{\mathcal{X}}

\newcommand{\I}{\mathcal{I}}

\newcommand{\halmos}{\quad\hfill\mbox{$\Box$}}

\endlocaldefs
\begin{document}
\begin{frontmatter}

\title{Conditional Independence \\ in Max-linear Bayesian Networks}
\runtitle{Conditional Independence in Max-linear Bayesian Networks}


\author{\fnms{Carlos} \snm{Am\'endola}\ead[label=e1]{carlos.amendola@tum.de}} 
\address{Center for Mathematical Sciences\\Technical University of Munich\\ Boltzmanstrasse 3\\ 85748 Garching, Germany\\  \printead{e1}}
\and
\author{\fnms{Claudia} \snm{Kl\"uppelberg}\ead[label=e2]{cklu@tum.de}}
\address{Center for Mathematical Sciences\\Technical University of Munich\\ Boltzmanstrasse 3\\ 85748 Garching, Germany\\ \printead{e2}}
\affiliation{Technical University of Munich}
\and
\author{\fnms{Steffen} \snm{Lauritzen}\ead[label=e3]{lauritzen@math.ku.dk}}
\address{Department of Mathematical Sciences\\ University of Copenhagen \\ Universitetsparken 5 \\ 2100 Copenhagen, Denmark \\ \printead{e3}}
\affiliation{University of Copenhagen}
\and
\author{\fnms{Ngoc M.} \snm{Tran}\ead[label=e4]{ntran@math.utexas.edu}}
\address{Department of Mathematics\\ University of Texas at Austin \\
Speedway 2515 Stop C1200\\ Austin TX 78712, USA \\ \printead{e4}}
\affiliation{University of Texas at Austin}

\runauthor{Am\'endola, Kl\"uppelberg, Lauritzen, Tran}

\begin{abstract}
Motivated by extreme value theory, max-linear Bayesian networks have been recently introduced and studied as an alternative to linear structural equation models. However, for max-linear systems the classical independence results for Bayesian networks are far from  exhausting valid conditional independence statements. We use tropical linear algebra to derive a compact representation of the conditional distribution given a partial observation, and exploit this to obtain a complete description of all conditional independence relations. In the context-specific case, where conditional independence is queried relative to a specific value of the conditioning variables, we introduce the notion of a \emph{source DAG} to disclose the valid conditional independence relations. In the context-free case we characterize conditional independence through a modified separation concept, $\ast$-separation, combined with a tropical eigenvalue condition. We also introduce the notion of an \emph{impact graph} which describes how extreme events spread deterministically through the network and we give a complete characterization of such impact graphs. Our analysis opens up several interesting questions concerning conditional independence and tropical geometry. 
\end{abstract}



\end{frontmatter}

\section{Introduction}

Max-linear graphical models were introduced in \cite{GK1} to model causal dependence between extreme events. The underlying graphical structure of the model is a directed acyclic graph (DAG) and to emphasize this aspect, we shall here use the term \emph{max-linear Bayesian network}, to allow for generalizations and extensions (see Section~\ref{sec:extensions} at the end of this paper).

A \emph{max-linear Bayesian network} is specified by a random vector $X=(X_1,\ldots,X_d)$, a directed acyclic graph $\D=(V,E)$ with nodes $V=\{1,\ldots, d\}$, non-negative {\em edge weights\/} $c_{ij}\geq 0$ for $i,j \in V$, and independent positive random variables $Z_1, \dots, Z_d$. These, known as \emph{innovations}, have support $\Rplus:=(0,\infty)$ and have atom-free distributions. Then $X$ is specified by a \emph{recursive system of max-linear structural equations} as
\begin{equation}\label{def:ml}
X_i = \bigvee_{j \in {{\pa}}(i)} c_{ij} X_j \vee  Z_i, \quad i=1,\ldots,d.    
\end{equation}
Without loss of generality, we assume that the basic probability space is $\Omega=\Rplus^{V}$ equipped with the standard Borel $\sigma$-algebra, so all randomness in the system originates from the innovations $Z_1, \dots, Z_d$.  
The equation system \eqref{def:ml} has solution
\begin{equation}\label{def:mlin}
X_i = \bigvee_{j \in {\an(i)\cup\{i\}}} c^*_{ij} Z_j, \quad i=1,\ldots,d,   
\end{equation}
where $\an(i)$ denotes the set of nodes $j$ where there is a directed path from $j$ to $i$, and $c^*_{ij}$ is a maximum taken over all the products along such paths (see \cite{GK1}, Theorem~2.2). Any such path that realizes this maximum is called \emph{critical} (max-weighted under $C$). The \emph{max-linear coefficient matrix} $C^* = (c^\ast_{ij})$ is also known from tropical linear algebra as the \emph{Kleene star} of $C = (c_{ij})$, cf.\ \eqref{ml} below.

In \cite{KL2017}, it was observed that the conditional independence properties for max-linear Bayesian networks are very different from those in a standard Bayesian network. In particular, they are often not faithful to their underlying DAG $\D$. This means that the usual $d$-separation criterion (\citep{geiger:verma:pearl:90}) on the DAG typically will not identify all valid conditional independence relations, in contrast to the situation for most Bayesian networks based on discrete random variables or linear structural equations. 
 Example~\ref{ex:diamond0} below gives a simple example of this phenomenon. 

\begin{example}[Diamond]\label{ex:diamond0} Consider the DAG in Figure~\ref{fig:diamondsep0}. 
\begin{figure}[tb]
\begin{center}
\begin{tikzpicture}
\begin{scope}[->,every node/.style={circle,draw},line width=1pt, node distance=1.8cm]
\node (1) {$1$};
\node (2) [below left of=1, fill=red!75] {$2$};
\foreach \from/\to in {1/2}
\draw (\from) -- (\to);
\path[every node/.style={font=\sffamily\small}]
(1) -- (2) node [near start, left] {$c_{21}$};
\node (3) [below right of=1] {$3$}; 
\foreach \from/\to in {1/3}
\draw (\from) -- (\to);
\path[every node/.style={font=\small}]
(1) -- (3) node [near start, right] {$c_{31}$};
\node (4) [below right of=2] {$4$};
\foreach \from/\to in {2/4,3/4}
\draw (\from) -- (\to);
\path[every node/.style={font=\small}]
(2) -- (4) node [near end, left] {$c_{42}$};
\path[every node/.style={font=\small}]
(3) -- (4) node [near end, right] {$c_{43}$};
\end{scope}
\end{tikzpicture}
\end{center}
\caption{Diamond graph with the set $K=\{2\}$ being observed, as indicated by red color. If  $c_{42}c_{21} \geq c_{43}c_{31}$, it holds that $X_1\ci X_4 \cd X_2$.}\label{fig:diamondsep0}
\end{figure}
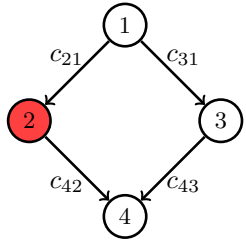
The path $1\to 2\to 4$ is critical if and only if $c_{42}c_{21} \geq c_{43}c_{31}$. If this is the case, the joint distribution of $(X_1,X_2,X_4)$ has the representation
\[ X_1=Z_1,\quad X_2=c_{21}X_1\vee Z_2, \]
and
\begin{align*}
X_4 &= c_{42}X_2 \vee Z_4 \vee c_{43}X_3 \\
&= c_{42}(Z_2 \vee c_{21}Z_1) \vee Z_4 \vee c_{43}(Z_3 \vee c_{31}Z_1) \\
&= c_{42}Z_2 \vee c_{42}c_{21}Z_1 \vee Z_4 \vee c_{43}Z_3 \vee c_{43}c_{31} Z_1 \\
&= c_{42}Z_2 \vee c_{42}c_{21}Z_1 \vee Z_4 \vee c_{43}Z_3 \quad \mbox{ since } c_{42}c_{21} \geq c_{43}c_{31} \\
&= c_{42}X_2 \vee Z_4 \vee c_{43}Z_3
\end{align*}
and hence we have $X_1\ci X_4\cd X_2$ which does \emph{not} follow from the $d$-separation criterion. Here, the fact that $1 \to 2 \to 4$ is \emph{critical} renders the path $1 \to 3 \to 4$ unimportant for the conditional independence $X_1\ci X_4\cd X_2$, even if $1 \to 3 \to 4$ were also critical (that is, even if $c_{42}c_{21} = c_{43}c_{31}$). 
\halmos
\end{example}

In Example~\ref{ex:diamond0}, the complicating issue was associated with paths being critical or not. However, this is not the only way standard $d$-separation fails. In Example~\ref{ex:cassiopeia0} below, the complications are associated with double colliders along a path.

\begin{example}[Cassiopeia]\label{ex:cassiopeia0} 
We shall show later (see Example~\ref{ex:cassiopeia2}) that a max-linear Bayesian network on the graph in Figure~\ref{fig:cassiopeia0}  
\begin{figure}[tb]
\begin{center}
\begin{tikzpicture}[->,every node/.style={circle,draw},line width=1pt, node distance=1.5cm]
  \node (1)  {$1$};
  \node (4) [below right of=1,fill=red!75]{$4$};
  \node (2) [above right of=4] {$2$};
  \node (5) [below right of=2,fill=red!75] {$5$};
  \node (3) [above right of=5] {$3$}; 
\foreach \from/\to in {1/4,2/4,2/5,3/5}
\draw (\from) -- (\to);   
 \end{tikzpicture}  
 \caption{The Cassiopeia graph with observed nodes $K=\{4,5\}$.  
 Here it holds that $X_1\ci X_3\cd  X_{\{4,5\}}$.
 }\label{fig:cassiopeia0}
\end{center}
\end{figure}
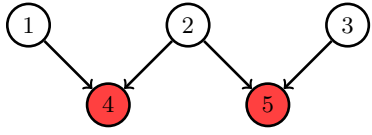
 will satisfy $X_1\ci X_3\cd X_{\{4,5\}}$ for all coefficient matrices $C$.  However, this conditional independence statement does \emph{not} follow from the $d$-separation criterion since the path from $1$ to $3$ is $d$-connecting relative to $\{4,5\}$. 
 
\halmos
\end{example}

 Example~\ref{ex:cassiopeia0} shows that not only are max-linear Bayesian networks often not faithful to $d$-separation, but $d$-separation is also not \emph{complete}  in the sense of \citep{geiger:verma:pearl:90} for conditional independence in these networks. That is, there are conditional independence statements which are valid for any choice of coefficients $C$, but cannot be derived from $d$-separation.
 
Also, in contrast to standard results for Bayesian networks, some conditional independence relations are highly \emph{context-specific,} i.e.\ depend drastically on the particular values of the conditioning variables,  as in Example~\ref{ex:tent0} below.  To control this, we introduce the notion of a \emph{source DAG} $\C(X_K=x_K)$ for a given \emph{context} $\{X_K=x_K\}$, see Definition~\ref{defn:context.graph} for details.

\begin{example}[Tent]\label{ex:tent0}
Consider the DAG $\D$ to the left in Figure~\ref{fig:tentA}
with all edge weights $c_{ij}=1$. Let $K=\{ 4,5 \}$ be the set of observed nodes; we seek all independence relations conditionally valid in the context $X_4=X_5=2$.
Writing out the model \eqref{def:ml} we find
\begin{align*}
X_1 &= Z_1,\quad X_2 = Z_2,\quad X_3 = Z_3\vee X_1 \vee X_2\\
X_4 &= Z_4\vee X_1 \vee X_2 = 2 \\
X_5 &= Z_5\vee X_1 \vee X_2 = 2.
\end{align*}
Since $Z_1,\dots,Z_5$ are a.s.\ different when the innovations have atom-free distributions,  it holds apart from a null-set that $X_1\vee X_2=Z_1\vee Z_2=2$. This introduces bounds on the innovations; we must have $Z_1,Z_2, Z_4,Z_5 \le 2$ and it also holds that $X_3 \ge 2$. Further, we then have
\begin{align*}
X_1 &= Z_1,\quad X_2 = Z_2,\quad  X_1\vee X_2=2, \quad X_3 = Z_3\vee2,\\
X_4 &= Z_4\vee 2 = 2 \\
X_5 &= Z_5\vee 2 = 2,
\end{align*}
whence we conclude that $X_3\ci (X_1, X_2)\cd X_4=X_5=2$, since now the dependence of $X_3$ on $X_1,X_2$ has disappeared. This independence statement is reflected in the lack of edges $1 \to 3$ and $2 \to 3$ in the source DAG $\C(X_4 = X_5 = 2)$, shown to the right in Figure~\ref{fig:tentA}.
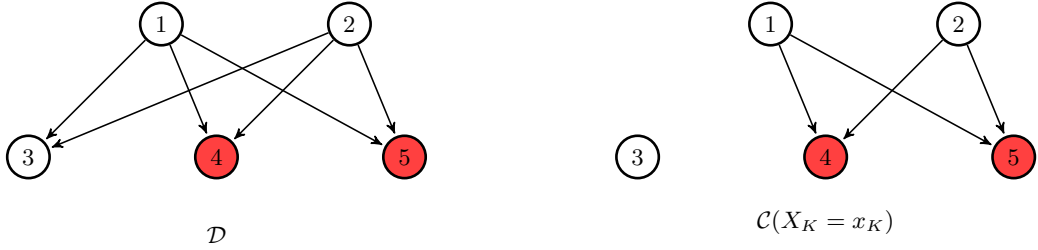
\begin{figure}
\begin{center}
\begin{tikzpicture}[->,>=stealth',shorten >=1pt,auto,node distance=2.5cm,semithick]
  \tikzstyle{every node}=[circle,line width =1pt]
  \node (1) [draw] {$1$};
  \node (2) [right of=1,draw] {$2$};
  \node (3) [below left of=1,draw] {$3$};  
  \node (4) [right of = 3,draw,fill=red!75] {$4$};
  \node (5) [right of = 4,draw,fill=red!75] {$5$};
  \path (1) edge (3);
  \path (1) edge (4);  
  \path (1) edge (5);    
  \path (2) edge (3);
  \path (2) edge (4);
  \path (2) edge (5);
  \node [below = 0.4 cm of 4] {$\D$};
  
  \node (a3) [right= 2.5cm of 5, draw] {$3$};
  \node (a1) [above right of=a3,draw] {$1$};  
  \node (a2) [right of=a1,draw] {$2$};
  \node (a4) [right of = a3,draw,fill=red!75] {$4$};
  \node (a5) [right of = a4,draw,fill=red!75] {$5$};
  \path (a1) edge (a4);  
  \path (a1) edge (a5);    
  \path (a2) edge (a4);
  \path (a2) edge (a5);
   \node [below = -0.5 cm of a4] {$\C(X_K = x_K)$};
\end{tikzpicture}
\end{center} \vspace{-1cm}
\caption{The left-hand figure displays what we shall name the tent DAG $\D$. For all coefficients equal to 1,  the source DAG $\C(X_K = x_K)$ when the observed nodes are $K=\{4,5\}$ with observed values $x_4=x_5=2$, is obtained from the left-hand figure by removing the edges $1 \to 3$ and $2 \to 3$, which become redundant in the context $\{X_4=X_5=2\}$, see Section~\ref{sec:source}.} \label{fig:tentA} 
\end{figure}
\halmos
\end{example}

In this paper we give a complete description of valid conditional independence statements for a given matrix $C$,  conditional independence statements that hold for all $C$ supported on a given DAG $\mathcal{D}$, as well as those that depend on the specific values of the conditioning variables.  
We achieve this by introducing three separation criteria. These are less restrictive than $d$-separation, as they focus on paths that are \emph{critical}  (see Example~\ref{ex:diamond0}), do not have \emph{multiple colliders} (see Example~\ref{ex:cassiopeia0}), and, for a given context, refer to the source DAG, obtained by removing edges that are \emph{redundant} in the context (see Example~\ref{ex:tent0}). 

Before we state and prove results for conditional independence, we investigate how extreme events at selected nodes spread through the network. 
We define an \emph{impact graph}  as a random graph on $V$, containing the edge $j\to i\iff X_i=c_{ij}^* Z_j$, i.e.\ if $X_i$ is realized (determined) by $Z_j$ (see Definition~\ref{defn:impact}). Since the distributions of the innovations are atom-free, it holds with probability one that any node $i$ has at most one parent in such a graph. 
We give a complete description of all impact graphs with positive probability in {\bf Theorem}~\ref{thm:impact}. As we shall see in Remark \ref{rmk:linear.pieces}, the impact graphs index partitions of the innovation space into regions of linearity for the max-linear map in \eqref{def:mlin}.

Impact graphs can be compatible with a context $\{X_K=x_K\}$ or not, and vice versa, a context $\{X_K=x_K\}$ can be possible under a certain impact graph or not. 
For instance, for the Cassiopeia graph in Example~\ref{ex:cassiopeia0}, the possible impact graphs are: the empty graph, all subgraphs with a single edge, and the four subgraphs with two edges displayed in Figure~\ref{fig:cass.impact}. 
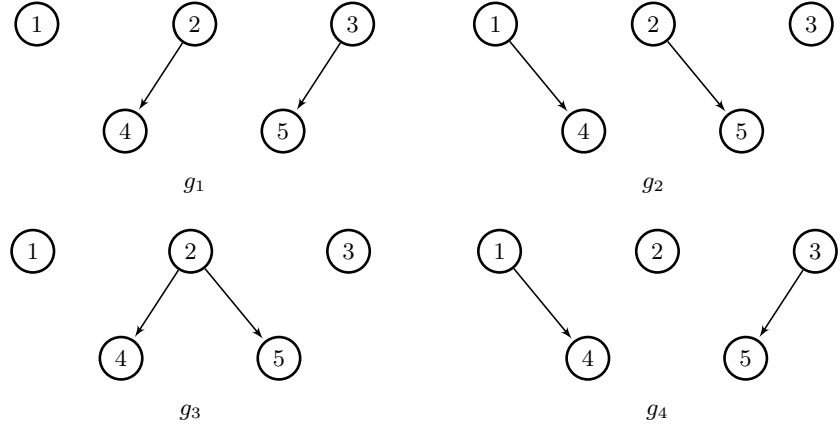
\begin{figure}
\begin{center}
\begin{tikzpicture}[->,>=stealth',shorten >=1pt,auto,node distance=1.5cm,semithick]
  \tikzstyle{every node}=[circle,line width =1pt]  
      \node(a) at (0,0)[draw] {$1$};
      \node(b) [right =  of a,draw] {$2$};
      \node(c) [right =  of b,draw] {$3$}; 
      \node(d) [below right =1cm and 0.75cm of a,draw]{$4$};
      \node(e) [below right = 1cm and 0.75cm of b,draw]{$5$};
 \node(f) [below= 1.5cm of b]{$g_1$};
    \begin{scope}[->, > = latex']
    \draw (b) -- (d);
	\draw (c) -- (e);
    \end{scope}

    \end{tikzpicture}\qquad \qquad \begin{tikzpicture}[->,>=stealth',shorten >=1pt,auto,node distance=1.5cm,semithick,draw=black]
  \tikzstyle{every node}=[circle,line width =1pt]
      \node(a) at (0,0)[draw] {$1$};
      \node(b) [right = of a,draw] {$2$};
      \node(c) [right = of b,draw] {$3$}; 
      \node(d) [below right = 1cm and 0.75cm of a,draw]{$4$};
      \node(e) [below right = 1cm and 0.75cm of b,draw]{$5$};
 \node(f) [below= 1.5cm of b]{$g_2$};
    \begin{scope}[->, > = latex']
		\draw (a) -- (d);
   \draw (b) -- (e);
    \end{scope}

    \end{tikzpicture}
    \end{center}
   \begin{center} 
   \begin{tikzpicture}[->,>=stealth',shorten >=1pt,auto,node distance=1.5cm,semithick,draw=black]
  \tikzstyle{every node}=[circle,line width =1pt]

      \node(a) at (0,0)[draw] {$1$};
      \node(b) [right =  of a,draw] {$2$};
      \node(c) [right = of b,draw] {$3$}; 
      \node(d) [below right = 1cm and 0.75cm  of a,draw]{$4$};
      \node(e) [below right = 1cm and 0.75cm  of b,draw]{$5$};
     
    \node(f) [below= 1.5cm of b]{$g_3$};

    \begin{scope}[->, > = latex']
    \draw (b) -- (d);
   \draw (b) -- (e);
    \end{scope}

    \end{tikzpicture}
    \qquad\qquad \begin{tikzpicture}[->,>=stealth',shorten >=1pt,auto,node distance=1.5cm,semithick]
   \tikzstyle{every node}=[circle,line width =1pt]
      \node(a) at (0,0)[draw] {$1$};
      \node(b) [right = of a,draw] {$2$};
      \node(c) [right = of b,draw] {$3$}; 
      \node(d) [below right = 1cm and 0.75cm of a,draw]{$4$};
      \node(e) [below right = 1cm and 0.75cm of b,draw]{$5$};
 \node(f) [below= 1.5cm of b]{$g_4$};
    \begin{scope}[->, > = latex']
		\draw (a) -- (d);
	\draw (c) -- (e);
    \end{scope}

    \end{tikzpicture}
    \end{center}
    \caption{The four impact graphs with two edges in the Cassiopeia graph of  Example~\ref{ex:cassiopeia0} as depicted in Figure~\ref{fig:cassiopeia0}. Suppose all coefficients are equal to one. Then, only the impact graphs $g_2$ and $g_4$ are compatible with the  context $\{X_4=3,X_5=2\}$, whereas only the impact graph $g_3$ is compatible with the context $\{X_4=X_5=2\}$, see Example~\ref{ex:cass_possible} below. \label{fig:cass.impact}}
    \end{figure}
On the other hand, the impact graph $g_2$ implies that $X_4>X_5$, so only events satisfying this restriction are possible under $g_2$. 

The union of all impact graphs compatible with a context $\{X_K=x_K\}$ describes all possible ways that an extreme innovation could spread across the network while conforming with the context. However, as seen in Example~\ref{ex:tent0}, the given context can cause max-linear combinations of variables to be constant under specific scenarios, such that they do not influence the distribution of random variables $X_v, v\notin K$ as expected.
This effect is taken care of by the removal of edges to yield the \emph{source DAG} $\C(X_K=x_K)$ compatible with the context. 

Moreover, we classify all nodes into non-constant nodes (active) and constant nodes with  specific properties (see Proposition~\ref{cor:dag.partition}). 
This classification plays an important role when modifying the solution in \eqref{def:mlin} to obtain a compact representation of the conditional distribution. 

This compact representation is given in {\bf Theorem}~\ref{thm:source.rep} and
can be seen as a version of Theorem~6.7 in \cite{GK1} and of Theorem~1 in \cite{Wang2011}. More precisely, \cite{Wang2011} studies a general max-linear model where the max-linear coefficient matrix $C^*$ is not necessarily the Kleene star of a max-linear Bayesian network, hence not necessarily idempotent; they further give a more detailed description of the conditional distribution, using  a collection of \emph{hitting scenarios}, describing specific elements of $Z$ which obtain their upper bounds. An important endeavour of the present article is to further identify characteristics of the hitting scenarios, exploiting the graphical structure of the model, and this is done in Theorem~\ref{thm:source.rep}.

We formulate three different theorems to clarify conditional independence for max-linear Bayesian networks.
All three have the following structure, using  what we shall term $*$-separation ($\starse$) in appropriate derived DAGs.

\medskip 

\noindent
{\bf Theorem} 
{\it Let $X$ be a max-linear Bayesian network over a directed acyclic graph $\D = (V,E)$.  
Then for all $I,J,K \subseteq V$,}
\begin{align*}
I \starse J \cd K \mbox{ in } \widetilde\D \implies X_I \ci X_J \cd X_K. 
\end{align*}

\medskip

The DAG $\widetilde \D$ --- derived from $\D$, $C$, and the specific {\em context} $\{X_K=x_K\}$ --- 
depends on the situation and we distinguish the following three: 
{\bf Theorem}~\ref{thm:starse.xk} refers to a fixed $C$ and also a specific {context} $\{X_K=x_K\}$, thus yielding conditional independence relations that are valid for the particular values $x_K$; 
     {\bf Theorem}~\ref{thm:starse.C} considers a fixed coefficient matrix $C$ and yields independence relations that may depend on  $C$ which are valid for all possible contexts; whereas in
     {\bf Theorem}~\ref{thm:starse.ind}, the coefficient matrix
     $C$ is arbitrary with support included in $\D$ and this yields conditional independence relations that are universally valid under these conditions.
     In all three scenarios, the derived DAG $\widetilde\D$ is different, and the $*$-separation has to be considered in this derived DAG.
     In addition, we give conditions for these criteria to be \emph{complete} in the sense of \cite{geiger:verma:pearl:90}, that is, they yield all conditional independence statements that are valid under the specified conditions.  
     In Section 6 we investigate in which sense these results are as strong as possible. For a fixed matrix $C$, Theorem~\ref{thm:starse.C} may not yield all conditional independence results, as  specific relations between the coefficients could imply additional independences, similar to what is known as path cancellations for linear Bayesian networks. We provide the tools to identify exactly what these are in Theorem~\ref{thm:faith.C}.

The paper is organized as follows. 
In Section~\ref{sec:prelim} we introduce basic concepts and notation. We define the impact graphs, describing how effects of extremes spread to other variables, and the source DAG, describing the possible sources for a given value of observations, in Section~\ref{sec:impact}.   Section~\ref{sec:representation} is devoted to deriving a compact representation of the conditional distribution and the conditional independence results are stated and proved in Section~\ref{sec:cond.indep}. Section~\ref{sec:faith} is devoted to the discussion of completeness. 
We conclude by indicating potential future work and research directions in Section~\ref{sec:outlook}.

\section{Preliminaries}\label{sec:prelim}
\subsection{Graph terminology} \label{prelim:graph}
We use the same graph notation as in \cite{KL2017}. A \emph{directed graph} is a pair $g=(V,E)$ of a \emph{node set} $V = \{1,\dots,d\}$ and \emph{edge set} $E = \{j \to i: i,j \in V, i \neq j\}$. An edge $j \to i$ points  \emph{from} $j$ \emph{to} $i$, 
with $j$ called a \emph{parent} of $i$ and $i$ is a \emph{child} of $j$. In a graph $g$, the set of parents of $i$ is $\pa_g(i)$ and the set of children of $i$ is $\ch_g(i)$. A \emph{path} from $j$ to $i$ of \emph{length} $n\geq 1$ is a sequence of distinct nodes $[j=k_0, k_1, \ldots, k_n=i]$ such that $k_{r-1} \to k_{r} \in E$ or $k_{r} \to k_{r-1} \in E$  for all $r=1,\ldots,n$ and we say that $i$ and $j$ are \emph{connected}. We also say that the path is \emph{between} $i$ and $j$. A graph is \emph{connected} if there is a path between any two vertices.

A \emph{directed path} from $j$ to $i$ has $k_{r-1} \to k_{r} \in E$ for all $r$.
 If there is a directed path from $j$ to $i$ in $g$, we say that $j$ is an \emph{ancestor} of $i$ and $i$ a \emph{descendant} of $j$. 
 A  \emph{directed cycle} is a directed path with the modification that $i=j$.
 A directed acyclic graph (abbreviated DAG) is a directed graph with no directed cycles. A DAG is \emph{well-ordered} if all edges point from low to high, that is, $j\to i \implies j<i$.
 A connected DAG is a \emph{tree} if every node has at most one parent. The \emph{root} of a tree is the unique node in the tree without parents. The \emph{height} of a tree is the length of the longest directed path in the tree.  A \emph{forest} is a collection of trees. A \emph{star} is a tree of height at most one, and we call a forest of stars a \emph{galaxy}. For a forest $g$ on node set $V$ and $i \in V$, we let $R_g(i)$ denote the root of the tree containing $i$ and $R(g)$ denotes the set of roots in $g$.
 A matrix $A \in  \R_\geq^{d \times d}$ defines a \emph{weighted directed graph} $\D(A)$, where $j \to i \in \D(A)$ if and only if its \emph{edge weight} $a_{ij} > 0$. The \emph{weight of a path} $\pi$ in $\D(A)$ is then the product of its edge weights.

\subsection{Tropical linear algebra}\label{prelim:trop}
A number of theorems in our paper are proved using techniques from tropical linear algebra. Here we recall some essential facts of this field. For a comprehensive text, we recommend \cite{BCOQ} and \cite{But2010}; see also \cite{joswig:20} and \cite{MS}.  

Tropical linear algebra is linear algebra with arithmetic in the \emph{max-times semiring} $(\mathbb{R}_\ge, \vee,\odot)$, defined by
$$a \vee b := \max(a,b), \quad a \odot b := a b \quad  \mbox{ for } a,b \in \R_\ge:=[0,\infty).$$
Note that many authors (including those above) use the isomorphic semirings max-plus or min-plus, but we have chosen max-times to conform with the literature on extreme value theory.
The operations extend to $\R_\ge^d$ coordinate-wise and to corresponding matrix multiplication for $A\in \R_\ge^{m \times n}$ and $B\in \R_\ge^{n\times p}$ as 
\[(A\odot B)_{ij}= \bigvee_{\ell=1}^n a_{i\ell}b_{\ell j}\]and we also write
$$\lambda\odot x =(\lambda x_1,\ldots, \lambda x_d)\quad  \mbox{ for } \lambda\in \R_\ge \mbox{ and } x\in \R_\ge^d.$$ 
For example,
\begin{equation}\label{ex:eig}
    \begin{pmatrix}
1 & 8  \\
2 & 3 \\
\end{pmatrix} \odot
\begin{pmatrix}
2 \\
1 \\
\end{pmatrix} =  \begin{pmatrix}
1\cdot 2 \, \vee \, 8\cdot 1  \\ 2 \cdot 2 \, \vee \, 3 \cdot 1  \\
\end{pmatrix} = \begin{pmatrix}
8 \\
4 \\
\end{pmatrix} = 4 \odot \begin{pmatrix}
2 \\
1 \\
\end{pmatrix} .
\end{equation}
The recursive structural equation system  \eqref{def:ml} can be written as a tropically linear equation:
\beam\label{eq:matrixrec}
X= C\odot X \vee Z
\eeam
where $C = (c_{ij}) \in \R_{\geq}^{d \times d}$ and $X, Z \in \Rplus^{d}$.
We consider the \emph{weak transitive closure} (\citep[Section 1.6.2]{But2010}) $\Gamma=\Gamma(C)=(\gamma_{ij})$ of $C$ given as
\begin{align}\label{eq:transclos}
\Gamma = \Gamma(C) =\bigvee_{k=1}^{d-1}C^{\odot k}.
\end{align}

Here $\gamma_{ij} > 0$ if and only if there exists a directed path in $\D(C)$ from $j$ to $i$, and $\gamma_{ij}$ equals the maximum weight over all such paths. 
We name $\D^\ast(C)$ the \emph{weighted reachability DAG} of $\D(C)$ and $\D^\ast$ the unweighted counterpart. 
When $\D(C)$ is a DAG, by \cite[Theorem 3.17]{BCOQ}, \eqref{eq:matrixrec} can be solved uniquely for $X$ as
\begin{equation}\label{ml}
  X = C^\ast \odot Z, 
\end{equation} 
where $C^\ast= I_d \vee \Gamma(C)$ is the \emph{Kleene star} of $C$ and $I_d$ the $d\times d$ identity matrix.
Since Kleene stars are idempotent, that is, $C^\ast \odot C^\ast = C^\ast$, we also have 
\begin{equation}\label{eqn:idempotent}
    X = C^\ast \odot X.
\end{equation}
If $V$ is well-ordered, the matrix $C$ is lower triangular and so are $\Gamma$ and $C^\ast$. The Kleene star $C^\ast$ corresponds to the max-linear coefficient matrix $B$ in \cite{GK1}, \cite{KL2017}, and in particular, \cite[Theorem 2.2]{GK1} is a special instance of \cite[Theorem 3.17]{BCOQ}. For $K\subseteq V$ we let
\begin{equation}\label{eqn:feasible}\mathcal{L}^C_K=\left\{x_K: \exists z\in \Rplus^V \text{ with } x_K=(C^\ast\odot z)_K\right\}\end{equation}
denote the image of the projection to $K$-coordinates of the max-linear map determined by $C^\ast$.
A matrix $A = (a_{ij}) \in \R_{\geq}^{d \times d}$ has \emph{tropical eigenvalue} $\lambda$ and \emph{tropical eigenvector} $x \in \R_{\geq}^d$ if 
\begin{equation}\label{eigen}
A \odot x = \lambda \odot x. 
\end{equation}
For example, equation \eqref{ex:eig} shows that $x = (2,1)$ is a tropical eigenvector for the given matrix, with tropical eigenvalue $\lambda=4$.

The maximum geometric mean of weights along a directed cycle in $\D(A)$ is the \emph{maximum cycle mean} of $A$, denoted $\lambda(A)$.
Note that if $\D(A)$ is acyclic then $\lambda(A)=0$. For any matrix $A$, the number $\lambda(A)\geq 0$ is always a tropical eigenvalue, called the \emph{principal eigenvalue} of $A$ \cite[Theorem~4.2.4]{But2010}. A cycle achieving the maximum mean is a \emph{critical} cycle. 
Similarly, a vector $x \in \Rplus^d$ is called a {\em tropical subeigenvector of $A$} for $\lambda>0$ if 
\begin{equation}\label{subeigen}
A \odot x \leq \lambda \odot x.
\end{equation}
The following fact about  tropical subeigenvectors will be useful. 
\begin{proposition}\label{prop:lambda.1}
Let $A \in \R_{\geq}^{V \times V}$. We then have 
\begin{enumerate}
	\item[(a)] There exists $x \in \Rplus^V$ such that $A\odot x \leq x$ if and only if $\lambda(A) \leq 1$.
	\item[(b)] Suppose $\lambda(A) = 1$, $S \subseteq V$ is the union of the support of its critical cycles, and that $x \in \Rplus^V$ satisfies $A\odot x \leq x$.  Then $A_{SS}\odot x_S = x_S$. 
	\item[(c)] There exists $x \in \Rplus^V$ such that $A\odot x < x$ if and only if $\lambda(A) < 1$. 
\end{enumerate}
\end{proposition}

\begin{proof}
Statement (a) is shown in \cite[Theorem 1.6.18]{But2010}. Now we prove (b). First consider the case $S = V$. Let $x$ be such that $A\odot x \leq x$. Fix any $i \in V$. Then $i$ belongs to some critical cycle $\sigma$ of length $r$ that achieves the tropical eigenvalue. For each edge 
$v \to u$ in this cycle, $A\odot x \leq x$ implies
\begin{equation}\label{eqn:avu}
a_{uv}x_v \leq x_u.
\end{equation}
Therefore, 
\begin{align*}
\prod_{v \to u \in \sigma}a_{uv} &\leq \prod_{v \to u \in \sigma}\frac{x_u}{x_v} \\
&= 1 \mbox{ since } \sigma \mbox{ is a cycle}.
\end{align*}
But $\sigma$ is critical, so
$$ \prod_{v \to u \in \sigma}a_{uv} = c(\sigma) = \lambda^r = 1. $$
Thus all the inequalities in \eqref{eqn:avu} must be equalities; that is, $a_{uv}x_v = x_u$ for all nodes $u,v$ in the support of $\sigma$. In particular, this holds for $u = i$. Thus, for the edge $v \to i \in \sigma$,
$$ x_i \geq (A\odot x)_i \geq a_{iv}x_v = x_i. $$
So $(A\odot x)_i = x_i$. Since $i$ was chosen arbitrarily, it follows that $A\odot x = x$. Now suppose $S \subset V$. Let $\bar{S} = V \backslash S$. Then
$$ x_S \geq (A\odot x)_S = A_{SS}\odot x_S \vee A_{S\bar{S}}\odot x_{\bar{S}} \geq A_{SS}\odot x_S. $$
Since $\lambda(A_{SS}) = 1$, applying the previous argument to $A_{SS}$ gives $A_{SS}\odot x_S = x_S$. 

Now we prove (c). Suppose $\lambda(A) < 1$. Let $x$ be an associated eigenvector to the principal eigenvalue of $A$. Then $A\odot x = \lambda(A) x < x$.   For the converse,  if $A\odot x < x$ it also satisfies $A\odot x \leq x$ so by (a) we have $\lambda(A)\leq 1$. If $\lambda(A) = 1$ then by (b), there exists some $S \subseteq V$, $|S| \geq 2$, such that $A_{SS}\odot x_S = x_S$. But then
$$ x_S=A_{SS}\odot x_S \leq (A\odot x)_S < x_S, $$
a contradiction. Thus we conclude that $\lambda(A) < 1$ and the proof is complete. 
\end{proof}
We recall one more useful fact from tropical linear algebra:
\begin{lemma}[\cite{But2010}, Lemma 1.6.19]
\label{lem:critical.edges}
Let $A \in \R_{\geq}^{V \times V}$ with $\lambda(A) = 1$ and eigenvector $x \in \Rplus^V$. Let $\sigma$ be a critical cycle in $A$. Then for all edges $v \to u \in \sigma$, 
$$ a_{uv}x_v = x_u. $$
\end{lemma}

\subsection{Conditional independence}\label{sec:indep}

Conditional independence is concerned with probability distributions on product spaces $\calx=\prod_{i\in V}\calx_i$, where $\calx_i$ are measurable spaces. For $I\subseteq V$ we write $x_I=(x_v, v\in I)$ to denote a generic element in $\calx_I=\prod_{v\in I}\calx_v$, and similarly $X_I=(X_v)_{v\in I}$. 
If $\P$ is a probability distribution on $\calx$, we use the short notation
\[I\ci J\cd K \iff X_I\ci X_J\cd X_K\] where $\ci$ denotes probabilistic conditional independence w.r.t.\ $\P$. 

Graphical models identify conditional independence relations through a \emph{separation criterion} $\sg$ applied to a graph. 
Such a separation criterion is, for example, given by \emph{$d$-separation} $\dse$ (\cite{geiger:verma:pearl:90}) for a given DAG  $\D$; see for example \cite{KF}, \cite{Lauritzen1996}, or  \cite{Lauritzen1990}  for further details. 
A separation criterion $\sg$ is \emph{complete} relative to a family $\mathcal{P}$ of probability distributions if 
$$ I \ci J \cd K  \text{ for all $\mathbb{P}\in \mathcal{P}$}\iff I \sg J \cd K. $$
A probability distribution $\mathbb{P}$ of $X$ is \emph{faithful} to $\sg$ if for all disjoint subsets $I,J,K$ of $V$ and that specific $\mathbb{P}$ it holds that $$ I \ci J \cd K \iff I \sg J \cd K. $$
Thus the distribution of $X$ is in particular \emph{Markov} w.r.t.\ $\sg$. If there exists a faithful distribution $\mathbb{P}\in \mathcal{P}$ for $\sg$, the separation criterion is obviously complete.

\section{Auxiliary graphs}\label{sec:impact}

In this section we introduce the concept of an impact graph, an impact graph compatible with a context, and a source DAG.
These are devices that translate probabilistic statements to graph-theoretic and algebraic statements, and at the same time keep track of all deterministic relationships in a max-linear Bayesian network.

\subsection{The context-free impact graph}

\begin{definition}\label{defn:impact}
Consider the max-linear Bayesian network \eqref{eq:matrixrec} with fixed coefficient matrix $C$. The (context-free) \emph{impact graph} is a random graph $G=G(Z)$ on $V$ consisting of the following edges:
\[j\to i \iff  X_i = c^\ast_{ij}Z_j\]
and we let $\mathcal{E}(g)= \{z \in \Rplus^V: G(z)=g\}$ denote the event that $\{G=g\}$.
\end{definition}

In the following, we let $\mathfrak{G}=\mathfrak{G}(C)$ denote the set of impact graphs that have positive probability for a given coefficient matrix $C$. 

\begin{remark}\label{rmk:onesource} 
Since the distributions of the $Z_j$ are atom-free, it holds with probability one that any node $i$ has at most one parent and thus if $\P(\E(g)>0)$, i.e.\ $g\in \mathfrak{G}$, $g$ will be a forest. We shall only consider configurations of $Z$ that conform with this and we emphasize that we are only ignoring a null-set in $\Omega=\Rplus^{V}$. 
\end{remark}

\begin{remark}\label{rmk:linear.pieces}
Define the restricted Kleene star $C^\ast_g$ as
\begin{equation}\label{eqn:rest_Kleene}(C^*_g)_{ij}=
\begin{cases} 1 & \mbox{if $i=j\in R(g)$}\\
c^\ast_{ij} &\mbox{if $j\to i\in g$}\\
0&\mbox{otherwise.}
\end{cases}
\end{equation}
The impact graphs induce a partition of $\Rplus^V$ into regions where the map $Z\to X$ is linear with matrix $C^*_g$. In other words, we have an alternative representation of $X$ as 
\begin{equation}\label{eq:linearpieces}
X=C^*\odot Z\stackrel{\text{a.s.}}{=}C^*_G\odot Z=C^*_GZ,
\end{equation}
where the product in the rightmost expression is a standard linear matrix product as $C^*_G$ has exactly one positive number in each row.  See also Example~\ref{ex:all.impact.graphs} below. 
\end{remark} 

The main result of this section is Theorem \ref{thm:impact}, which gives a precise and complete characterization of all impact graphs $\mathfrak{G}$ in a max-linear Bayesian network \eqref{eq:matrixrec}. To establish this characterization, we need to define the impact exchange matrix of a given forest $g$. Recall that $\ch_g(i)$ denotes the set of children of $i$ in $g$ (Section~\ref{prelim:graph}).  
\begin{definition}
Consider a DAG $\D$ with coefficient matrix $C$ and Kleene star $C^*$ and let $g$ be a forest with root set $R=R(g)$.  
The \emph{impact exchange matrix} $M(g)=M(g,C^*)$ of $g$ with respect to $C^*$ is an $|R| \times |R|$ matrix with entries defined by
$m_{rr} = 0$ for all $r \in R$, and for $r\neq r'$: 
\begin{equation}\label{eqn:Lambda}
m_{rr'} := \max_{i\in \ch_{g}(r)} \dfrac{c^*_{ir'}}{c^*_{ir}}.
\end{equation}
\end{definition}
Note that $m_{rr'}=0$ if $\ch_g(r) = \emptyset$. Finally, recall from Section~\ref{prelim:trop} that $\D^\ast$ is the reachability DAG and $\lambda(M(g))$ is the principal eigenvalue of $M(g)$. 
We now have the following fundamental theorem:
\begin{theorem}\label{thm:impact}
Consider a max-linear Bayesian network with coefficient matrix $C$ and Kleene star $C^*$. Then $g\in\mathfrak{G}$ if and only if the following four conditions hold:
\begin{enumerate}[{\rm(a)}]
\item $g$ is a subgraph of $\D^\ast$.
\item $g$ is a galaxy, i.e.\ a forest of stars.
\item If $j\to i$ in $g$ and $c_{ij}^*=c^*_{ik}c^*_{kj}$ then $k\not\to i$ and $j\to k$ in $g$.
\item $\lambda(M(g)) < 1$. 
\end{enumerate}

\end{theorem}

Before we proceed to the proof of this result, some explanation of the elements of the theorem might be appropriate. 
Theorem~\ref{thm:impact} describes all possible impact scenarios. With probability one, any outcome of the max-linear Bayesian network has a system of (extreme) root variables $Z_R$ and the value at all other nodes will be a.s. constant and appropriate multiples of these, their impact spreading across the network as determined by the galaxy $g\in \mathfrak{G}$. 

The conditions (a) and (b) in Theorem~\ref{thm:impact} are necessary, but not sufficient. 
To understand condition (d), consider the definition of the impact exchange matrix $M(g)$. Intuitively, the entry $m_{rr'}$ measures the worst possible relative cost for a node $i$ to be reassigned from root $r$ to root $r'$ in $g$. The graph induced by positive entries of $M(g)$ may have directed cycles. A directed cycle in this graph starting at a root $r$ creates an inequality involving $Z_r$. Condition (d) of Theorem~\ref{thm:impact} ensures that this inequality can be satisfied. Example~\ref{ex:eigencondition} below shows that a violation of the condition on the principal eigenvalue $\lambda(M(g))$ of the impact exchange matrix $M(g)$ yields an inconsistency, even if the other conditions are satisfied. The argument in Example~\ref{ex:eigencondition} below illustrates the key step in the proof that establishes the necessity of condition (d).

\begin{example}[Bipartite]\label{ex:eigencondition}
Consider the weighted graph $\D$ with weights given in Figure~\ref{fig:eigencondition}.  
\begin{figure}[tb] 
\begin{tikzpicture}[->,>=stealth',shorten >=1pt,node distance=1.8cm,semithick, el/.style = inner sep=2pt, align=left, every label/.append style = {font=\tiny}]
  \tikzstyle{every node}=[circle,line width =1pt]
  \node (a) [draw] {$1$};
  \node (b) [right of=a,draw] {$2$} ;
  \node (c) [below of=a,draw] {$3$};
  \node (d) [below of=b, draw] {$4$};
  \path (a) edge (c);
  \path (a) edge (d);
  \path (b) edge (c);
  \path (b) edge (d);
    \node [below right= 0.5cm of c] {$\dag(C)$};
  \node (alab)[ below left  =.3cm and 0.0cm of a]{$\frac12$};
  \node (blab)[ below right =.3cm and 0.0cm of b]{$\frac12$};
 \node (clab)[below right =.3cm and 0.05cm of a]{$1$};
 \node (dlab)[below left =.3cm and 0.05cm of b]{$1$};
  \node (a2) [right =5cm of a,draw] {$1$};
  \node (b2) [right of=a2,draw] {$2$} ;
  \node (c2) [below of=a2,draw] {$3$};
  \node (d2) [below of=b2, draw] {$4$};
  \path (a2) edge  (c2);
  \path (b2) edge  (d2);
  \node [below right= -0.1cm and -0.2cm of c2] {$\lambda(M(g))=2$};
  \node [below right =0.4cm and 0.4cm of a2]{$g$};
  \end{tikzpicture}  
  \vspace*{-1cm}
  \caption{Bipartite DAG: The subgraph $g$ to the right is not an impact graph for the weighted DAG $\D(C)$ to the left as it violates the principal eigenvalue condition (d) in Theorem~\ref{thm:impact}.} \label{fig:eigencondition}
\end{figure}
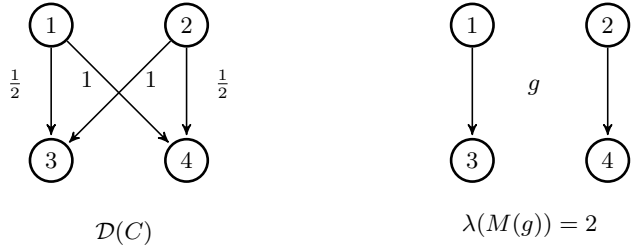 
The subgraph $g$ to the right in Figure \ref{fig:eigencondition} satisfies conditions (a)-(c) of Theorem \ref{thm:impact}. However, it fails to satisfy condition (d) because
$$ M(g) = \begin{bmatrix}
0 & 2 \\
2 & 0 \\
\end{bmatrix}, $$
since
$$ m_{12} = \max_{i \in \ch_g(1) = \{3\}} \frac{c^\ast_{i2}}{c^\ast_{i1}} = \frac{c^\ast_{32}}{c^\ast_{31}} = \frac{1}{1/2} = 2, \quad m_{21} = \max_{i \in \ch_g(2) = \{4\}} \frac{c^\ast_{i1}}{c^\ast_{i2}} = \frac{c^\ast_{41}}{c^\ast_{42}} = \frac{1}{1/2} = 2. $$
Then $\lambda(M(g)) = \sqrt{2 \cdot 2} = 2 > 1$, so $g$ is not an impact graph. Indeed if it were, $1\to 3$ would imply $X_3=\frac12 Z_1> Z_2$ but $2\to 4$ would imply $X_4=\frac12 Z_2>Z_1$ and thus $Z_1>2Z_2>4Z_1$, which is inconsistent since $Z_1>0$.
\halmos
\end{example}

\begin{proof}[Proof of Theorem \ref{thm:impact}]
First we show that all conditions are necessary. 

To prove (a), let $g \in \mathfrak{G}$. If $c^\ast_{ij}=0$, then $X_i > 0 = c^\ast_{ij} Z_j$, which means that $j \to i \notin g$.
So $g$ is a subgraph of $\D^\ast$ which proves (a). 

To establish (b) we first note that  $g$ must be a forest from Remark~\ref{rmk:onesource}.
We shall then argue that any tree in the forest has height at most one.  Suppose $j \to i \in g$. Then
$$ X_i = c^\ast_{ij}Z_j > Z_i \mbox{ on } \mathcal{E}(g). $$
Now, for any $k \in V$, either $c^\ast_{ki} = 0$ so $i \to k \notin g$ by (a), or we have by the idempotency of $C^*$ and \eqref{eqn:idempotent} that
$$X_k \geq  c^\ast_{ki}c^\ast_{ij}Z_j > c^\ast_{ki}Z_i \mbox{ on } \mathcal{E}(g),$$
and therefore there is no edge $i\to k$ in $g$ which proves (b). 

Next we establish (c). Consider a triple of nodes $i,j, k$ with $j\to i$ and $c^*_{ij}=c^*_{ik}c^*_{kj}$. Since $g$ is a forest and $j\to i$, we must have $k \to i \notin g$. Also, since $j\to i$ we have as before
\[X_i = c^*_{ij}Z_j \mbox{ on $\mathcal{E}(g)$}.\]
Using \eqref{eqn:idempotent} again, we know that $X_i \geq c^*_{ik}X_k $. Then the relation $c^*_{ij}=c^*_{ik}c^*_{kj}$ yields
\[X_k \geq c^*_{kj}Z_j=\frac{c^*_{ij}}{c^*_{ik}}{Z_j}
=\frac{X_i}{c^*_{ik}} \geq X_k \mbox{ on $\mathcal{E}(g)$}\]
and hence we must have equality  so $X_k = c^*_{kj}Z_j$ on $\mathcal{E}(g)$. This proves (c). 

For condition (d) we first note that if $\lambda(M)=0$ then it is certainly less than 1. So assume $\lambda(M)>0$. Then there exists a critical cycle $r_1 \leftarrow r_2 \dots \leftarrow r_k \leftarrow r_1$ with $r_1,\dots,r_k\in R$ such that 
\begin{equation}\label{eqn:c.geq.1}
0 < (\lambda(M))^k = m_{r_1r_2}m_{r_2r_3} \dots m_{r_kr_1}.
\end{equation}
In particular, this implies each edge in the cycle is not $0$, so for each edge, say, $r_2 \to r_1$, there exists a node $i \in V$ that achieves this maximum so that $r_1\to i$ in $g$ and
$$ \frac{c^\ast_{ir_2}}{c^\ast_{ir_1}} = m_{r_1r_2}. $$
Now, since $r_1\to i$ in $g$ and $i$ has at most one parent, this implies
$c^*_{ir_1}Z_{r_1}>c^*_{ir_2}Z_{r_2}$, 
whereby $ Z_{r_1} > m_{r_1r_2}Z_{r_2} $
by rearranging. 
Tracing this cycle, we obtain the equation 
$$ Z_{r_1} > \left(m_{r_1r_2}m_{r_2r_3} \dots m_{r_kr_1}\right) Z_{r_1}.$$
Dividing by $Z_{r_1}>0$, we obtain from \eqref{eqn:c.geq.1} that $\lambda(M)<1$.
Thus, all four conditions are necessary. \\

To see that the conditions are sufficient, let $g$ be a graph that satisfies all four conditions. Let $\epsilon > 0$ be an arbitrarily small constant, and 
\begin{equation}\label{eqn:alpha}\alpha = \max_{i,j,k,\ell: c^\ast_{ji},c^\ast_{\ell k} > 0} \frac{c^\ast_{\ell k}}{c^\ast_{ji}}.\end{equation}
Let $v$ be a tropical eigenvector of $M$ for $\lambda(M) < 1$. This means $$(M\odot v)_r = \lambda(M)v_r < v_r,$$ for $v_r>0$, so that the event
\begin{align*}
\E = \{&v_r > Z_r > (M \odot v)_r \mbox{ for } r \in R \mbox{ s.t. } v_r > 0 \, \, , \, \, Z_r > \alpha\epsilon \mbox{ for } r \in R \mbox{ s.t. } v_r = 0,  \\
&\mbox{ and }  Z_j < \epsilon \mbox{ for } j \notin R \}.
\end{align*}
satisfies $\P(\E) > 0$. We now argue that $\E$ is a subevent of $\mathcal{E}(g)$. Since the collection of events $\{\mathcal{E}(g): g\in \mathfrak{G}\}$ partitions the innovation space $\E$, the event $\E$ must be partitioned into finitely many events, each with positive probability, which we denote $\E \cap\, \mathcal{E}(g'_1), \cdots, \E \cap\,  \mathcal{E}(g'_s)$. By definition of $g$, each $i$ belongs to a unique  star with root $r$. 
Under the event $\E$, for all $r' \in R$, $r' \neq r$,
$$ Z_r > \max_{r'} m_{rr'}v_{r'} > \max_{r'} m_{rr'}Z_{r'}, $$
therefore,
$ Z_r > m_{rr'}Z_{r'} $
 that is,
$ c^\ast_{ir}Z_r > c^\ast_{ir'}Z_{r'}$
whence $r'\not\to i$ for all $g' \in \{g'_1,\dots,g'_s\}$. Similarly, for any $r' \notin R$ with $c^\ast_{ir'} > 0$, 
$$ c^\ast_{ir}Z_i > c^\ast_{ir}\alpha \geq c^\ast_{ir}\frac{c^\ast_{ir'}}{c^\ast_{ir}}\epsilon > c^\ast_{ir'}Z_{r'} $$
and hence $r'\not\to i$ in any of  $g' \in \{g'_1,\dots,g'_s\}$. Thus $r\to i$ in any $g'$ and we must have $\E \subseteq \mathcal{E}(g)$, so $\P(\mathcal{E}(g)) \geq \P(\E) > 0$, as needed.
\end{proof}

\begin{example}[Half-butterfly]\label{ex:impacts}
Let $\D$ be the weighted DAG given in the leftmost part of Figure~\ref{fig:impacts}. Its weighted reachability DAG $\D^\ast(C)$ is shown to its right. For example, we have $c^\ast_{41}=c^\ast_{43}c^\ast_{31}=3$. 
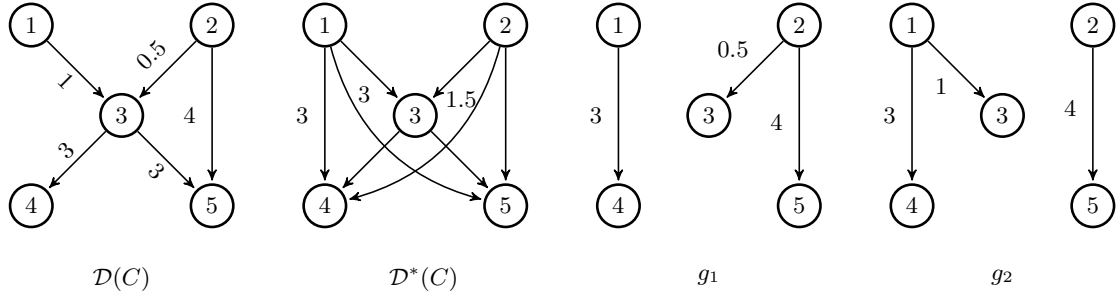
\begin{figure}
\begin{center} 
\begin{tikzpicture}[->,>=stealth',shorten >=1pt,swap,node distance=1.7cm,semithick, el/.style = {inner sep=1pt, align=left, sloped},
every label/.append style = {font=\tiny}]
  \tikzstyle{every node}=[circle,line width =1pt]
  \node (aa) [draw] {$1$};
  \node (cc) [below right of=aa,draw] {$3$} ;
  \node (bb) [above right of=cc,draw] {$2$} ;
  \node (dd) [below left of=cc,draw] {$4$};
  \node (ee) [below right of=cc, draw] {$5$};
  \path (aa) edge node[every node/.style={font=\sffamily\footnotesize},el,below]  {$1$}   (cc);
  \path (bb) edge node[every node/.style={font=\sffamily\footnotesize}, el,above]  {$0.5$} (cc);   
  \path (cc) edge  node[every node/.style={font=\sffamily\footnotesize},el,above]  {$3$} (dd);
  \path (cc) edge node[every node/.style={font=\sffamily\footnotesize},el,below]  {$3$} (ee);
  \path (bb) edge ["$4$"] (ee);
  \node (lab) [below = 1.3cm of cc] {$\D(C)$};
  
     \node (a4) [right = .9cm of bb, draw] {$1$};
  \node (c4) [below right of=a4,draw] {$3$} ;
  \node (b4) [above right of=c4,draw] {$2$} ;
  \node (d4) [below left of=c4,draw] {$4$};
  \node (e4) [below right of=c4,draw] {$5$};
  \path (a4) edge   (c4);
  \path (b4) edge  (c4);   
  \path (c4) edge  (d4);
  \path (c4) edge  (e4);
  \path (b4) edge  (e4);
  \path (a4) edge ["$3$"] (d4); 
  \path (a4) edge [bend right=30, align=left, below](e4);
  \path (b4) edge [bend left=30, align=left, above](d4);
  \node [below right =0.5cm and 0.1cm of a4]{$3$};
  \node [below left =0.5cm and 0.1cm of b4]{$1.5$};
  \node [right= 2.8cm of lab] {$\D^\ast(C)$};
  
  \node (a) [right=.9cm of b4, draw] {$1$};
  \node (c) [below right of=a,draw] {$3$} ;
  \node (b) [above right of=c,draw] {$2$} ;
  \node (d) [below left of=c,draw] {$4$};
  \node (e) [below right of=c, draw] {$5$};
  \path (a) edge ["$3$"] (d);
  \path (b) edge [pos=0.55,"$4$"](e);
  \path (b) edge ["$0.5$"](c);
    \node [below = 1.5cm of c] {$g_1$};
  
   \node (a3) [right = .9cm of b, draw] {$1$};
  \node (c3) [below right of=a3,draw] {$3$} ;
  \node (b3) [above right of=c3,draw] {$2$} ;
  \node (d3) [below left of=c3,draw] {$4$};
  \node (e3) [below right of=c3,draw] {$5$};
  \path (a3) edge ["$1$"] (c3);
  \path (a3) edge ["$3$"] (d3);
  \path (b3) edge [pos=0.45,"$4$"] (e3);
    \node [below = 1.5cm of c3] {$g_2$};
  \end{tikzpicture}  
  \caption{The half-butterfly graph $\D(C)$ and its weighted reachability DAG $\D^\ast(C)$ where only edge weights for additional edges are indicated. The galaxy $g_1$ is not an impact graph for this DAG as it violated the as it violates the triangle condition (c) in Theorem~\ref{thm:impact}, while the galaxy $g_2$ is. 
  } \label{fig:impacts}
\end{center}
\end{figure}
 Now consider the two different subgalaxies $g_1$ and $g_2$ shown to the right in Figure~\ref{fig:impacts}. We shall see that $g_2$ is an impact graph for the given coefficient matrix $C$ while $g_1$ is not. Indeed, $g_1 \notin \mathfrak{G}$ as it violates the triangle condition (c) in Theorem~\ref{thm:impact}: $1\to 4 \in g_1$ and $c^\ast_{41}=c^\ast_{43}c^\ast_{31}$ but $1 \to 3 \notin g_1$. On the other hand, $1\to 3 \in g_2$ as required. Furthermore, $g_2$ has impact exchange matrix given by
$$M(g_2) =\begin{bmatrix}
0 & \frac12 \\
\frac34 & 0 \\
\end{bmatrix}, $$
because
$$m_{12} =  \max_{i\in \ch_{g_2}(1)= \{3,4\} } \frac{c^*_{i2}}{c^*_{i1}} = \max \left\lbrace \frac{1/2}{1} , \frac{3/2}{3} \right\rbrace = \frac12, \quad m_{21} = \max_{i \in \ch_{g_2}(2) = \{5\}} \frac{c^\ast_{i1}}{c^\ast_{i2}} = \frac{c^\ast_{51}}{c^\ast_{52}} = \frac{3}{4}.$$
We have then $\lambda(M(g_2)) = \sqrt{\frac12 \cdot \frac34} = \sqrt{\frac38} < 1$ and so condition (d) also holds. We conclude by Theorem \ref{thm:impact} that $g_2 \in \mathfrak{G}$. A possible realization in terms of $Z$ is given by  $Z=(z_1,z_2,z_3,z_4,z_5)=(2,3,0.1,0.4,0.2)$  leading to $X=(x_1,x_2,x_3,x_4,x_5)=(2,3,2,6,12)$.
\halmos
\end{example}
The following simple lemma shall be used in the subsequent analysis. 
\begin{lemma}\label{lem:zy}
Consider the max-linear Bayesian network \eqref{eq:matrixrec} with fixed coefficient matrix $C$. Let $g\in \mathfrak{G}$ be an impact graph with root set $R=R(g)$. Then it holds for all $z \in \mathcal{E}(g)$ that
\begin{equation}\label{eqn:lambda.z.z}
M\odot z_R \leq z_R 
\end{equation}
where $M=M(g,C)$ is the impact exchange matrix of $g$ and $z_R = (z_r)_{r \in R}$ is the truncation of $z$ to the root set. 
\end{lemma}

\begin{proof}
Let $\E' = \{z: \mbox{$z$  does not satisfy } \eqref{eqn:lambda.z.z}\}$. Note that $\E'$ decomposes as the union of $R(R-1)$ sub-events $\E'_{rr'}$, where
$$ \E'_{rr'} = \{m_{rr'}z_{r'} > z_r \}. $$
We shall next show that $\mathcal{E}(g)\cap \E'_{rr'} = \emptyset$ for each pair $r,r' \in R$ with $r \neq r'$. 
Suppose for contradiction that there exists some $z \in \mathcal{E}(g) \cap \E'_{rr'}$. By definition, $m_{rr'} = \max_{i \in \ch_g(r)} \frac{c^\ast_{ir'}}{c^\ast_{ir}}$. Let $i \in \ch_g(r)$ be a node that achieves this maximum. Then $z \in \E'_{rr'}$ implies 
$$ c^\ast_{ir'}z_{r'} > c^\ast_{ir}z_{r}. $$
But now the max-linear representation of $X$ implies that on $\E'_{rr'}$, 
$$ x_i \neq c^\ast_{ir}z_r$$
which contradicts that $r\to i$ in $g$. Hence we conclude that $\mathcal{E}(g) \cap \E'_{rr'} = \emptyset$ and thus further that
$$\mathcal{E}(g) \cap\E'=\E(g)\cap\bigcup_{rr'} \E'_{rr'} = \emptyset,$$ as needed.
\end{proof}

\subsection{Impact graphs compatible with a context}

As mentioned in Remark~\ref{rmk:linear.pieces}, the impact graphs represent a partition of the innovation space $\E= \Rplus^V$ into regions of linearity. We can also represent these as linear maps $L_g: z \in \Rplus^{R(g)} \mapsto x \in \Rplus^V$ via
$$ L_g(z)_r=z_r,\quad L_g(z)_i = c^\ast_{ir}z_r \quad\mbox{ iff }\quad r \to i \text{ in } g.$$
We shall illustrate this in a small example.


\begin{example}[Bipartite]\label{ex:all.impact.graphs}
Consider again the DAG and coefficient matrix of Example~\ref{ex:eigencondition} as depicted in Figure~\ref{fig:eigencondition}. Figure~\ref{fig:all.impact.graphs} displays all impact graphs for this DAG, save for their symmetric counterparts. 
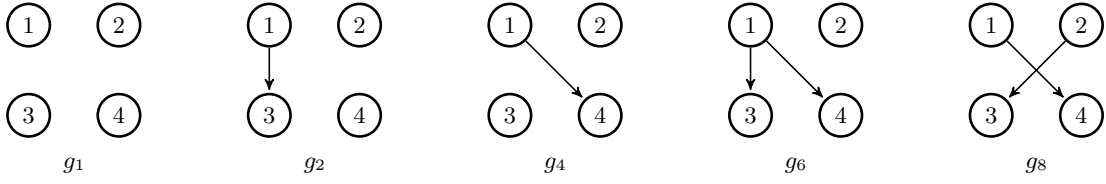
\begin{figure}
\begin{center} 
\begin{tikzpicture}[->,>=stealth',shorten >=1pt,auto,node distance=1.2cm,semithick, el/.style = {inner sep=2pt, align=left, sloped}, font = \footnotesize],
  \tikzstyle{every node}=[circle,line width =1pt]
  \node (a) [draw] {$1$};
  \node (b) [right of=a,draw] {$2$} ;
  \node (c) [below of=a,draw] {$3$};
  \node (d) [below of=b, draw] {$4$};
  
  \node (g) [below right = 0.2cm and 0.15 cm  of c] {$g_1$};
  
   \node (a1) [right = 1.4cm of b, draw] {$1$};
  \node (b1) [right of=a1,draw] {$2$} ;
  \node (c1) [below of=a1,draw] {$3$};
  \node (d1) [below of=b1, draw] {$4$};
  \path (a1) edge  (c1);
  \node (g1) [below right = 0.2cm and 0.15 cm  of c1] {$g_2$};

   \node (a2) [right = 1.4cm of b1, draw] {$1$};
  \node (b2) [right of=a2,draw] {$2$} ;
  \node (c2) [below of=a2,draw] {$3$};
  \node (d2) [below of=b2, draw] {$4$};
  \path (a2) edge  (d2);
  \node (g2) [below right = 0.2cm and 0.15 cm  of c2] {$g_4$};
  \node (a3) [right = 1.4cm of b2, draw] {$1$};
  \node (b3) [right of=a3,draw] {$2$} ;
  \node (c3) [below of=a3,draw] {$3$};
  \node (d3) [below of=b3, draw] {$4$};
  \path (a3) edge  (d3);
  \path (a3) edge (c3);
  \node (g3) [below right = 0.2cm and 0.15 cm  of c3] {$g_6$};
  \node (a4) [right = 1.4cm of b3, draw] {$1$};
  \node (b4) [right of=a4,draw] {$2$} ;
  \node (c4) [below of=a4,draw] {$3$};
  \node (d4) [below of=b4, draw] {$4$};
  \path (a4) edge  (d4);
  \path (b4) edge (c4);
  \node (g4) [below right = 0.2cm and 0.15 cm  of c4] {$g_8$};
  \end{tikzpicture}  
\end{center}
\vspace*{-5mm}
\caption{Impact graphs $\mathfrak{G}$ for the bipartite DAG in Figure~\ref{fig:eigencondition}. There are a total of eight such graphs, the remaining three ($g_3$, $g_5$, and $g_7$) obtained by the reflection $(1,3)\leftrightarrow (2,4)$ of $g_2$, $g_4$, and $g_6$. }
\label{fig:all.impact.graphs}
\end{figure}

Of the 16 edge-induced subgraphs of the DAG $\D$, only nine are forests and one of them, displayed to the right in Figure~\ref{fig:eigencondition}, violates the principal eigenvalue condition; so there are eight valid galaxies left, five of which are displayed in Figure~\ref{fig:all.impact.graphs}, and the remaining three obtained by appropriate relabeling.
The max-linear map is
\[C^*\odot z=\begin{bmatrix}
1 & 0 & 0 & 0 \\
0 & 1 & 0 & 0 \\
\frac12 & 1 & 1 & 0 \\
1 & \frac12 & 0 & 1 
\end{bmatrix}\odot\begin{bmatrix}z_1\\z_2\\z_3\\z_4\end{bmatrix}=\begin{bmatrix}z_1\\z_2\\\frac12z_1\vee z_2\vee z_3\\z_1\vee \frac12z_2\vee z_4\end{bmatrix}\]
and the corresponding matrices $C^*_g$ for the pieces of linearity are 
\[\begin{bmatrix}
1 & 0 & 0 & 0 \\
0 & 1 & 0 & 0 \\
0 & 0 & 1 & 0 \\
0 & 0 & 0 & 1 
\end{bmatrix}\quad
\begin{bmatrix}
1 & 0 & 0 & 0 \\
0 & 1 & 0 & 0 \\
\frac12 & 0 & 0 & 0 \\
0 & 0 & 0 & 1 
\end{bmatrix}\quad\begin{bmatrix}
1 & 0 & 0 & 0 \\
0 & 1 & 0 & 0 \\
0 & 0 & 1 & 0 \\
1 & 0 & 0 & 0 
\end{bmatrix}\quad\begin{bmatrix}
1 & 0 & 0 & 0 \\
0 & 1 & 0 & 0 \\
\frac12 & 0 & 0 & 0 \\
1 & 0 & 0 & 0 
\end{bmatrix}\quad\begin{bmatrix}
1 & 0 & 0 & 0 \\
0 & 1 & 0 & 0 \\
0 & 1 & 0 & 0 \\
1 & 0 & 0 & 0 
\end{bmatrix}
\]
mapping $z$ respectively into
\[\begin{bmatrix}z_1\\z_2\\z_3\\z_4\end{bmatrix}\qquad
\begin{bmatrix}z_1\\z_2\\\frac12z_1\\z_4\end{bmatrix}\qquad 
\begin{bmatrix}z_1\\z_2\\z_3\\z_1\end{bmatrix}\qquad\begin{bmatrix}z_1\\z_2\\\frac12z_1\\z_1\end{bmatrix}\qquad\begin{bmatrix}z_1\\z_2\\z_2\\z_1\end{bmatrix}.  \]
These maps can also be considered as maps $L_g$ from the root set  to the node set and would then have matrices
\[\begin{bmatrix}
1 & 0 & 0 & 0 \\
0 & 1 & 0 & 0 \\
0 & 0 & 1 & 0 \\
0 & 0 & 0 & 1 
\end{bmatrix}\quad
\begin{bmatrix}
1 & 0  & 0 \\
0 & 1  & 0 \\
\frac12 & 0 &  0 \\
0 & 0 & 1 
\end{bmatrix}\quad\begin{bmatrix}
1 & 0 & 0  \\
0 & 1 & 0 \\
0 & 0 & 1 \\
1 & 0 & 0 
\end{bmatrix}\quad\begin{bmatrix}
1 & 0  \\
0 & 1   \\
\frac12 & 0   \\
1 & 0 
\end{bmatrix}\quad\begin{bmatrix}
1 & 0   \\
0 & 1   \\
0 & 1   \\
1 & 0 
\end{bmatrix}
\]
where the roots $(1,2,4)$ in $g_2$ have been renumbered as $(1,2,3)$. Indeed these matrices are simply obtained by removing zero-columns in the first set of matrices. Note that the rank $r_g$ of the maps are all equal to $r_g=|R(g)|$, the number of stars in the galaxy, i.e.\ $4,3,3,2,2$ in these cases. 
\halmos
\end{example}
\begin{definition}\label{def:possible}
Let $K\subseteq V$ and $\Pi_K(x)=x_K$ be the projection onto coordinates in $K$. For $x_K\in \mathcal{L}^C_K$ as defined in \eqref{eqn:feasible}, we define
\begin{enumerate}[(a)]
\item A graph $g \in \mathfrak{G}$ is \emph{compatible with the context} $\{X_K =x_K\}$ if  the following are true:
\begin{enumerate}[(i)] 
\item $\E(g)\cap \{X_K=x_K\}\neq \emptyset$,
    \item the rank of $\Pi_K\circ L_g$ is minimal among those $g\in \mathfrak{G}$ which satisfy (i).
\end{enumerate}
\item The set of compatible graphs $g$ is called the \emph{impact graphs for the context $\{X_K =x_K\}$}, denoted $\mathfrak{G}(X_K =x_K)$. 
\item  We further say that the context $\{X_K =x_K\}$ is \emph{possible} if $\mathfrak{G}(X_K =x_K)\neq \emptyset$ and \emph{possible under $g$} if $g\in \mathfrak{G}(X_K =x_K)$. Else the context $\{X_K=x_K\}$ is said to be \emph{impossible} or \emph{impossible under $g$} respectively. For brevity we shall also use the expression that $x_K$ is \emph{possible}. 
\end{enumerate}
\end{definition}

Note that although all events of the form $\{X_K=x_K\}$ have probability zero, we are now distinguishing between those that are exceptions from events of the form $\E(g)$ (impossible contexts) and those that are not (possible contexts). In other words, $x_K$ might still satisfy $x_K\in \mathcal{L}^C_K$  without being possible.  In the following we shall only pay attention to possible contexts. Furthermore, this definition also applies to the special case $K=V$ so we now can speak about $\{X=x\}$ being possible or impossible under $g\in \mathfrak{G}$.

The rank condition (a) (ii) ensures that if any subevent $\E(g^*)$ includes $x_K$  and the map $\Pi_K\circ L_{g^*}$ has higher rank than $\Pi_K\circ L_{g}$, then the entire collection of contexts $\{X_K=x_K\}$ in the image of $\Pi_K\circ L_{g}$ is a null-set in $\E(g^*)$. Therefore, the set of points in $\mathcal{L}^C_K$ that are not possible has measure zero and can be ignored when discussing conditional distributions.
\begin{example}\label{ex:cass_possible}
Consider the Cassiopeia graph in Example~\ref{ex:cassiopeia0} with all coefficients equal to one and the event $\{X_4=X_5=2\}$. 
The impact graphs $g_3$ and $g_4$ are the only impact graphs among those in Figure~\ref{fig:cass.impact} that satisfy condition (i) in Definition~\ref{def:possible}, as the other impact graphs imply strict inequalities between $x_4$ and $x_5$.
In addition, the empty galaxy, and all galaxies with a single edge satisfy condition (i). However, the rank of $\Pi_K\circ L_{g_3}$ is one, whereas the rank of all other maps $\Pi_K\circ L_{g}$ is two. Hence only $g_3$ is compatible with $\{X_4=X_5=2\}$.\halmos
\end{example}

\begin{definition}\label{defn:fixed} 
Suppose $\{X_K = x_K\}$ is possible. 
We say that $X_j$ is \emph{constant on $\{X_K =x_K\}$} if there exists $x^*_j \in \Rplus$ such that
$$\{X_{K\cup j}=x_{K\cup j}\} \text{ is possible if and only if $x_j=x^*_j$}.$$
Similarly, $X_j$ is \emph{constant on
$\{X_K =x_K\}\cap \E(g)$} if there exists $x^*_j \in \Rplus$ such that
$$\{X_{K\cup j}=x_{K\cup j}\} \text{ is possible under $g$  if and only if $x_j=x^*_j$.}$$  
Define the set of \emph{constant nodes} on $\{X_K = x_K\}$ as
$$K^\ast := K^\ast(X_K =x_K) := \{j \in V: X_j \mbox{ is constant on } \{X_K =x_K\}\} $$ 
and nodes that are \emph{constant under $g$} as
$$K^\ast(g) := K^\ast(X_K =x_K,g) :=  \{j \in V: X_j \mbox{ is constant on  $\{X_K =x_K\}\cap \E(g)$}\}.$$
\end{definition}

Note that $K \subseteq K^\ast \subseteq K^\ast(g)$ for specific $g\in\mathfrak{G}$. Often these inclusions can be strict (see Example \ref{ex:butterfly2}). The following lemma characterizes these sets. Recall from Theorem~\ref{thm:impact} that each impact graph $g \in \mathfrak{G}$ is a galaxy.

\begin{lemma}\label{lem:k.ast.y} 
Suppose $g\in \mathfrak{G}(X_K = x_K)$ and $S=V(\sigma)$ is the node set of a star $\sigma$ in the galaxy $g$. Then, either
\begin{itemize}
    \item[(a)] $S \cap K \neq \emptyset$, in which case $S \subseteq K^\ast(g)$ and we call $S$ a \emph{constant star}; or
    \item[(b)] $S \cap K = \emptyset$, in which case $S \cap K^\ast(g) = \emptyset$.
\end{itemize}
In particular, 
$$K^\ast(g) = \bigcup_{\sigma\in g: S\cap K\neq \emptyset }V(\sigma).$$
\end{lemma}
\begin{proof}
We first consider (a): Note that if $j \to i \in g$, then on $\mathcal{E}(g)$, $X_i = c^\ast_{ij}Z_j$ and thus $X_i = c^\ast_{ij}X_j$. Therefore, if either $X_i$ or $X_j$ is constant on $\{X_K = x_K \} \cap \mathcal{E}(g)$, then both must be  constant. So if one node in $S$ is in $K$, all nodes in $S$ must be in $K^\ast(g)$. This proves (a).

Next we show (b): Let $R$ be the set of root nodes in $g$, $R^1 \subset R$ be the set of root nodes for all stars $S$ in $g$ such that $S \cap K  = \emptyset$, and $R^c = R \setminus R^1$ be the set of roots of the constant stars. By the first statement, $ X_r$ is  constant for all $r \in R^c$ on $\E(g) \cap \{X_K \in x_K\}$. Indeed, on $ \mathcal{E}(g) \cap \{X_K = x_K\}$, $M \odot z_R \leq z_R$ by  Lemma~\ref{lem:zy}. Furthermore, by the minimal rank condition, we must have strict inequality, that is, $M \odot z_R < z_R$. This equation splits up into the following lower and upper-bounds for $z_{R^1}$ in terms of $z_{R^c}$: 
\begin{align}
M_{R^1R^c}\odot z_{R^c} &< z_{R^1}, \label{eqn:lower.z.j1.x} \\
M_{R^cR^1}\odot z_{R^1} &< z_{R^c}. \label{eqn:upper.z.j1.x}
\end{align}
Since $g \in \mathfrak{G}(X_K = x_K)$, the upper and lower bounds cannot coincide. In particular, there exist two values of $x_r$ such that $\{X_{K\cup r}=x_{K\cup r}\}$ is possible under $g$. This implies that $R^1 \cap K^\ast(g) = \emptyset$, and since $R^1$ are the roots, none of their children can be in $K^\ast(g)$. This completes the proof.
\end{proof}

\begin{example}[Bipartite]\label{ex:constant}
Consider again the DAG and coefficient matrix of Example~\ref{ex:eigencondition} and let $K=\{3\}$. 
\begin{figure}
\begin{center} 
\begin{tikzpicture}[->,>=stealth',shorten >=1pt,auto,node distance=1.2cm,semithick, el/.style = {inner sep=2pt, align=left, sloped}, font = \footnotesize],
  \tikzstyle{every node}=[circle,line width =1pt]
  
  \node (a) [draw] {$1$};
  \node (b) [right of=a,draw] {$2$} ;
  \node (c) [below of=a,draw,fill=red!75] {$3$};
  \node (d) [below of=b, draw] {$4$};
  
  \node (g) [below right = 0.2cm and 0.15 cm  of c] {$g_1$};
  
   \node (a1) [right = 1.4cm of b, draw,fill=red!75] {$1$};
  \node (b1) [right of=a1,draw] {$2$} ;
  \node (c1) [below of=a1,draw,fill=red!75] {$3$};
  \node (d1) [below of=b1, draw] {$4$};
  \path (a1) edge  (c1);
  
  \node (g1) [below right = 0.2cm and 0.15 cm  of c1] {$g_2$};

   \node (a2) [right = 1.4cm of b1, draw] {$1$};
  \node (b2) [right of=a2,draw] {$2$} ;
  \node (c2) [below of=a2,draw,fill=red!75] {$3$};
  \node (d2) [below of=b2, draw] {$4$};
  \path (a2) edge  (d2);
  
  \node (g2) [below right = 0.2cm and 0.15 cm  of c2] {$g_4$};
  
  \node (a3) [right = 1.4cm of b2, draw,fill=red!75] {$1$};
  \node (b3) [right of=a3,draw] {$2$} ;
  \node (c3) [below of=a3,draw,fill=red!75] {$3$};
  \node (d3) [below of=b3, draw,fill=red!75] {$4$};
  \path (a3) edge  (d3);
  \path (a3) edge (c3);
  
  \node (g3) [below right = 0.2cm and 0.15 cm  of c3] {$g_6$};
  
  \node (a4) [right = 1.4cm of b3, draw] {$1$};
  \node (b4) [right of=a4,draw,fill=red!75] {$2$} ;
  \node (c4) [below of=a4,draw,fill=red!75] {$3$};
  \node (d4) [below of=b4, draw] {$4$};
  \path (a4) edge  (d4);
  \path (b4) edge (c4);
  
  \node (g4) [below right = 0.2cm and 0.15 cm  of c4] {$g_8$};
  \end{tikzpicture}  
\end{center}
\vspace*{-5mm}
\caption{Impact graphs in $\mathfrak{G}(X_3=x_3)$ for the bipartite DAG in Figures~\ref{fig:eigencondition} and \ref{fig:all.impact.graphs}. Red nodes are constant under $g$ for the context $\{X_3=x_3\}$, i.e.\ elements of $K^\ast(g)$. These are nodes that are in the same star as the node $3$.} \label{fig:constant}
\end{figure}
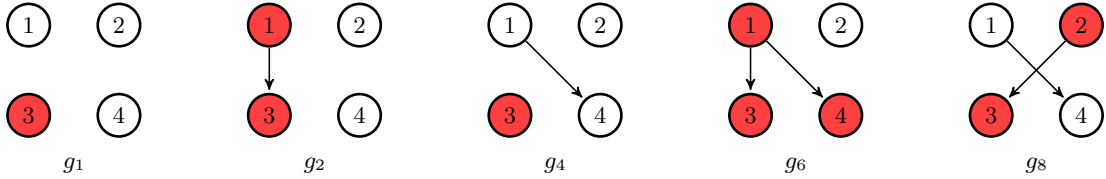
Figure~\ref{fig:constant} again displays impact graphs for this DAG, now with the constant nodes shaded. Nodes are constant in the context $\{X_3=x_3\}$ if and only if they belong to the same star as the node $3$. \halmos
\end{example}
Lemma~\ref{lem:y.xk} below identifies a crucial property of a compatible impact graph. 
\begin{lemma}\label{lem:y.xk}
Let $g \in \mathfrak{G}$ be an impact graph. If $g$ is compatible with  $\{X_K =x_K\}$, we have 
for all $i,j,h\in V$ that
\begin{align}
\frac{x_h}{c^\ast_{hj}} & < \frac{x_i}{c^\ast_{ij}}, \quad h,i \in K^\ast(g) \quad \implies \quad j \neq R_g(i)\label{eqn:extra1},\\
\exists j \in V: \, \, \frac{x_h}{c^\ast_{hj}} &= \frac{x_i}{c^\ast_{ij}}  , \quad h,i \in K^\ast(g) \quad\implies \quad R_g(i) = R_g(h).   \label{eqn:extra2} 
\end{align}
\end{lemma} 

\begin{proof} 
Consider \eqref{eqn:extra1}. Clearly $j\neq i$; suppose then for contradiction that $j\to i$. Since $i \in K^\ast(g)$, Lemma \ref{lem:k.ast.y} implies that $j \in K^\ast(g)$. Then $x_i = c^\ast_{ij}x_j$, so \eqref{eqn:extra1} implies that ${x_h}/{c^\ast_{hj}} < {x_i}/{c^\ast_{ij}} = x_j$, so $x_h < c^\ast_{hj}x_j$. But this contradicts that $x_h \geq c^\ast_{hj}x_j$. 
Now consider \eqref{eqn:extra2}. Write $r = R_g(i)$ and $r' = R_g(h)$. Then $x_i= c^\ast_{ir} x_r$ and $x_h= c^\ast_{hr'}x_{r'}$. Suppose for contradiction that $r \neq r'$.  Substituting into the hypothesis of \eqref{eqn:extra2} we get
$$ \frac{c^\ast_{hr'}}{c^\ast_{hj}}x_{r'} = \frac{c^\ast_{ir}}{c^\ast_{ij}}x_r, $$
which is a linear relation on the roots $x_r$ and $x_{r'}$ of two different stars in $g$. But this contradicts that $g$ has minimal rank according to Definition~\ref{def:possible} and thus  $\{X_K=x_K\}$ is not  possible under $g$. Hence \eqref{eqn:extra2} must hold.
\end{proof}

\begin{example}[Half-butterfly]\label{ex:butterfly2} To illustrate Lemma~\ref{lem:y.xk} we again
consider the DAG and coefficient matrix $C$ of Example~\ref{ex:impacts}, depicted in Figure~\ref{fig:impacts} and the context $\{X_K=x_K\}$ where $K=\{ 4, 5 \}$ and $x_4=x_5=1$. We claim that the set of constant nodes is $K^\ast = \{3,4,5\}$ because  the events $\{X_4 = X_5 = 1\}$ and $\{X_4 = X_5 = 1, X_3 = 1/3\}$ are almost surely identical, and that there are exactly two impact graphs compatible with this context, depicted in Figure~\ref{fig:butterfly2}. 
\begin{figure}
\begin{center}
\begin{tikzpicture}[->,>=stealth',shorten >=1pt,swap,node distance=1.6cm,semithick, el/.style = {inner sep=2pt, align=left, sloped},
every label/.append style = {font=\tiny}]
  \tikzstyle{every node}=[circle,line width =1pt]
  \node (aa) [draw] {$1$};
  \node (cc) [below right of=aa,draw] {$3$} ;
  \node (bb) [above right of=cc,draw] {$2$} ;
  \node (dd) [below left of=cc,draw] {$4$};
  \node (ee) [below right of=cc, draw] {$5$};
  \path (aa) edge node[every node/.style={font=\sffamily\footnotesize},el,below]  {$1$}   (cc);
  \path (bb) edge node[every node/.style={font=\sffamily\footnotesize}, el,above]  {$0.5$} (cc);   
  \path (cc) edge  node[every node/.style={font=\sffamily\footnotesize},el,above]  {$3$} (dd);
  \path (cc) edge node[every node/.style={font=\sffamily\footnotesize},el,below]  {$3$} (ee);
  \path (bb) edge ["$4$"] (ee);
  \node (lab) [below = 1.3cm of cc] {$\dag(C)$};
  
     \node (a4) [right = 1cm of bb, draw] {$1$};
  \node (c4) [below right of=a4,draw] {$3$} ;
  \node (b4) [above right of=c4,draw] {$2$} ;
  \node (d4) [below left of=c4,draw] {$4$};
  \node (e4) [below right of=c4,draw] {$5$};
  \path (a4) edge   (c4);
  \path (b4) edge  (c4);   
  \path (c4) edge  (d4);
  \path (c4) edge  (e4);
  \path (b4) edge  (e4);
  \path (a4) edge ["$3$"] (d4); 
  \path (a4) edge [bend right=30, align=left, below](e4);
  \path (b4) edge [bend left=30, align=left, above](d4);
  \node [below right =0.5cm and 0.1cm of a4]{$3$};
  \node [below left =0.5cm and 0.1cm of b4]{$1.5$};
  \node [right= 3.1cm of lab] {$\D^\ast(C)$};
  
  \node (a) [right=1cm of b4, draw,fill=red!75] {$1$};
  \node (c) [below right of=a,draw,fill=red!75] {$3$} ;
  \node (b) [above right of=c,draw] {$2$} ;
  \node (d) [below left of=c,draw,fill=red!75] {$4$};
  \node (e) [below right of=c, draw,fill=red!75] {$5$};
  \path (a) edge [bend right=30, align=left, below](e);
  \path (a) edge ["$3$"] (d);
  \node [below right =0.5cm and 0.1cm of a]{$3$};
  \node [below right =0.1cm and 0.5cm of a]{$1$};
  \path (a) edge (c);
    \node [below = 1.5cm of c] {$g_3$};
  
   \node (a3) [right = 1cm of b, draw] {$1$};
  \node (c3) [below right of=a3,draw,fill=red!75] {$3$} ;
  \node (b3) [above right of=c3,draw] {$2$} ;
  \node (d3) [below left of=c3,draw,fill=red!75] {$4$};
  \node (e3) [below right of=c3,draw,fill=red!75] {$5$};
  \path (c3) edge ["$3$"] (e3);
  \path (c3) edge ["$3$"] (d3);
    \node [below = 1.5cm of c3] {$g_4$};
  \end{tikzpicture} 
  \end{center}
\caption{Impact graphs $g_3$ and $g_4$ for the half-butterfly compatible with the context $\{X_4=X_5=1\}$ are displayed to the right. To the left, the original DAG $\D(C)$ and its weighted reachability DAG $\D^\ast(C)$ are shown.}  \label{fig:butterfly2}
\end{figure}
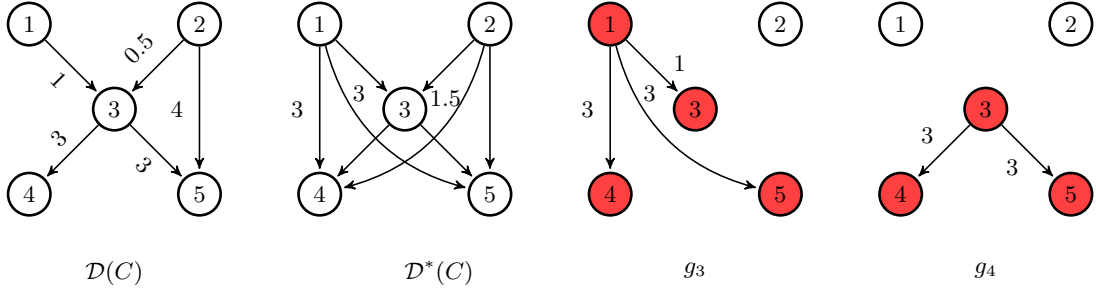

To see this, let $g$ be a compatible impact graph in $\mathfrak{G}(X_4 = X_5 = 1)$. Since ${x_4}/{c^\ast_{43}}={1}/{3}={x_5}/{c^\ast_{53}}$ we may apply \eqref{eqn:extra2} with $i = 4, h = 5, j=3$ to conclude that $4$ and $5$ belong to the same star in $g$, with common root $R_g(4) = R_g(5)$. 

By Theorem~\ref{thm:impact}(a), $g$ is a subgraph of $\D^\ast(C)$ and hence the roots must be parents in this graph, i.e.\ we have $R_g(4) \in \{1,2,3,4\}$, and $R_g(5) \in \{1,2,3,5\}$. So $R_g(4) = R_g(5)$ implies $R_g(4) = R_g(5) \in \{1,2,3\}$. On the other hand, it also holds that ${x_5}/{c^\ast_{52}}={1}/{4}<{1}/{1.5} = {x_4}/{c^\ast_{42}}$. Then \eqref{eqn:extra1} implies that $2 \notin \pa_g(4)$, so $R_g(4) = R_g(5) \in \{1,3\}$. By Theorem~\ref{thm:impact}(b), $g$ is a star so each node can not have more than one parent. So either it must hold that $R_g(4) = R_g(5) = 1$, or that $R_g(4) = R_g(5) = 3$. In the second case, $1$ and $2$ are left as isolated roots. In the first case, $1\to 4 \in g$ implies that we also have $1\to 3\in g$ by Theorem \ref{thm:impact}(c), that is, $3$ must belong to the same star as $4$ with $1$ as a root. This gives the two impact graphs to the right in Figure~\ref{fig:butterfly2}. 
In both cases, there is at most one non-isolated root, so $M(g)$ has no cycles and thus $\lambda(M(g))= 0 <1$ required for $g$ to be an impact graph. Thus both graphs are in $\mathfrak{G}(X_4 = X_5 = 1)$. By Definition~\ref{defn:fixed} the constant nodes are those that are in the same star as $4,5$ in the compatible impact graphs, so $K^\ast = \{3,4,5\}$ as we then have $X_3 = 1/3$ on $\{X_4 = X_5 = 1\}$, so $\{X_4 = X_5 = 1\} = \{X_4 = X_5 = 1, X_3 = 1/3\}$. 

We can double-check that $\mathfrak{G}(X_4=X_5=1) = \mathfrak{G}(X_4=X_5=1, X_3 = 1/3)$ by computing the latter set of graphs directly. Let $g \in \mathfrak{G}(X_4=X_5=1, X_3 = 1/3)$. Now we first apply \eqref{eqn:extra2} with $i=j=3$ and $h = 5$  to conclude that $R_g(3) = R_g(5)$. Thus again we have by Theorem \ref{thm:impact}(a) that $R_g(3) \in \{1,2,3\}$ and $R_g(5) \in \{1,2,3,5\}$. Therefore, $R_g(5) \neq 5$. Apply again \eqref{eqn:extra1} with $j = 2$, $i = 3$ and $h = 5$ to conclude that $2 \neq R_g(3)$. So $R_g(3) \in \{1,3\}$. The two cases $R_g(3) = 1$ and $R_g(3) = 3$ yield the two impact graphs to the right in Figure~\ref{fig:butterfly2} as expected.
\halmos
\end{example}
\begin{remark} 
Although in Definition~\ref{defn:impact} we have defined the impact graph $G=G(Z)$ (for almost all $Z$), $G$ can also be expressed in terms of $X$, as we indeed have for any $g\in \mathfrak{G}$ which is compatible with $\{X=x\}$ that
\begin{equation}\label{eqn:xcondition}
    j\to i \in g \implies x_i/x_j=c^*_{ij} \text{ on } \E(g)
\end{equation}
since on $\E(g)$ we must have $j\in R(g)$ and thus $X_j=Z_j$.  Hence with probability one there is a unique $g\in \mathfrak{G}$ that is compatible with $\{X=x\}$. Another way of expressing this is to say that the map $z\to g$ is almost surely $\sigma(X)$-measurable, where $\sigma(X)$ is the $\sigma$-algebra generated by the max-linear map $z\to x$ given by $x=C^*\odot z$. 
\end{remark}

\subsection{The source DAG}\label{sec:source}

Impact graphs describe how extreme events at their roots spread deterministically to other nodes. In this section we shall capitalize on this, but from the perspective of identifying which are the possible sources of extreme values responsible for a given possible context of the form $\{X_K=x_K\}$ (see Definition~\ref{def:possible}(c)). This will eventually make it possible for us to 
answer interesting queries concerning conditional independence. 

We first let $\mathcal{I}(X_K=x_K)$ denote the union of impact graphs which are compatible with the context: 
\begin{equation}\label{eqn:total_impact}
\mathcal{I}(X_K=x_K)=\bigcup_{g \in \mathfrak{G}(X_K =x_K)}g
\end{equation}
and we shall refer to this as the \emph{total impact graph} and note that it is a subgraph of the reachability DAG $\D^*$. 
In other words, this graph yields all possible ways that impact could have spread across the network in a way that conforms with the observation $\{X_K=x_K\}$. 

\begin{definition}\label{defn:redundant edge}
Let $K\subset V$ and $K^*(g)$ be the set of constant nodes under $g$ as in Definition~\ref{defn:fixed}.
An edge $j\to i$ in $\mathcal{I}(X_K=x_K)$
is \emph{redundant in the context $\{X_K=x_K\}$} if either 
\begin{itemize}
    \item $j \in K^\ast$, or
    \item $i \notin K^\ast$ and $X_j$ is  constant under all $g\in \mathfrak{G}(X_K=x_K)$ that contain the edge. 
\end{itemize}
 The set of redundant edges is denoted $E^- = E^-(X_K=x_K)$.
\end{definition}

\begin{definition}\label{defn:context.graph}
The {\em source DAG\/} $\C(X_K = x_K)$  of a possible context $\{X_K = x_K\}$ is the graph obtained from $\mathcal{I}(X_K=x_K)$ by removing redundant edges; see Definition~\ref{defn:redundant edge}. 
\end{definition}

\begin{example}[Tent]\label{ex:tent}
Consider the DAG $\D$ to the left in Figure~\ref{fig:tent} 
with all edge weights $c_{ij}=1$. 
Let $K=\{ 4,5 \}$ and $x_4=x_5=2$.
Note that $\mathcal{I}(X_K=x_K)=\D(C)$. However, $\C(X_K = x_K)$ is the strict subgraph of $\mathcal{I}(X_K = x_K)$ obtained by removing the dashed edges from the graph to the right in Figure \ref{fig:tent}. 

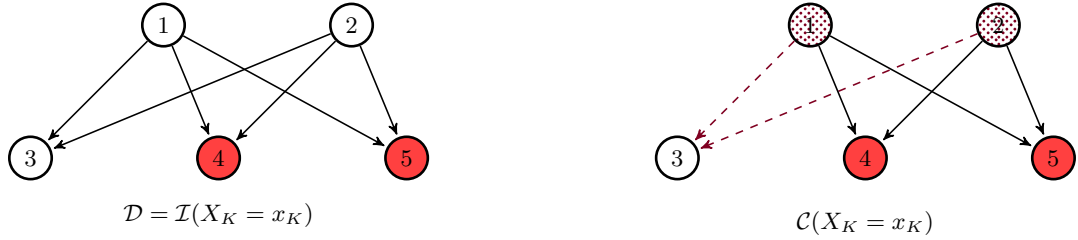
\begin{figure}
\begin{center}
\begin{tikzpicture}[->,>=stealth',shorten >=1pt,auto,node distance=2.5cm,semithick]
  \tikzstyle{every node}=[circle,line width =1pt]
  \node (1) [draw] {$1$};
  \node (2) [right of=1,draw] {$2$};
  \node (3) [below left of=1,draw] {$3$};  
  \node (4) [right of = 3,draw,fill=red!75] {$4$};
  \node (5) [right of = 4,draw,fill=red!75] {$5$};
  \path (1) edge (3);
  \path (1) edge (4);  
  \path (1) edge (5);    
  \path (2) edge (3);
  \path (2) edge (4);
  \path (2) edge (5);
  \node [below = -0.95 cm of 4] {$\dag = \mathcal{I}(X_K=x_K)$};
  
  \node (a3) [right= 3cm of 5, draw] {$3$};
 \node (a1) [above right of=a3,draw, pattern=crosshatch dots, pattern color = burgundy, pattern color = burgundy] {$1$}; 
  \node (a2) [right of=a1,draw,pattern=crosshatch dots, pattern color = burgundy] {$2$};
  \node (a4) [right of = a3,draw,fill=red!75] {$4$};
  \node (a5) [right of = a4,draw,fill=red!75] {$5$};
  \path (a1) edge[dashed, color = burgundy] (a3);
  \path (a2) edge[dashed, color = burgundy] (a3);
  \path (a1) edge (a4);  
  \path (a1) edge (a5);    
  \path (a2) edge (a4);
  \path (a2) edge (a5);
   \node [below = -0.5 cm of a4] {$\C(X_K = x_K)$};
\end{tikzpicture}
\vspace*{-1cm}
\end{center} 
\caption{{Tent graph:} To the left,  this displays $\D(C)= \D^\ast(C)=\mathcal{I}(X_K=x_K)$ when all coefficients are equal to 1. To the right, the source DAG $\C(X_K = x_K)$ for $K=\{4,5\}$ and $x_4=x_5=2$ is obtained by removing the dashed edges, which are redundant as each of the dotted nodes is constant in any impact graph containing the edge.} \label{fig:tent}
\end{figure}

The edge $1\to3$ is in $E^-$ since $1$ is constant under all impact graphs containing this edge, and similarly with the edge $2\to 3$.
To see this, note that $1 \to 3 \in g \in \mathfrak{G}(X_4 = X_5 = 2)$ 
if and only if $1 \to 4, 1 \to 5 \in g$, which then implies $1 \in K^\ast(g)$. 
Therefore, $1 \to 3$ is a redundant edge and not included in $\C(X_4 = X_5 = 2)$. A similar argument applies to the edge $2 \to 3$. 
From the node partition of Proposition~\ref{cor:dag.partition} below we see that the active nodes are $A=\{1,2,3\}$ and the constant nodes $K^*=K=L=\{4,5\}$.
\halmos
\end{example}

We first prove some results on the structure of the source DAG before linking it up to probabilistic statements. 
In particular we establish that the source DAG admits a nice partition structure, see Figure \ref{fig:conveyor} for an illustration. 

\begin{proposition} \label{cor:dag.partition}
Fix a possible context $\{X_K = x_K\}$, let $\mathcal{I} = \mathcal{I}(X_K=x_K)$ and $\C=\C(X_K = x_K)$ be the corresponding total impact graph and source DAG, respectively. Then for either of these graphs, its node set $V$ can be partitioned into disjoint sets $A \cup U \cup H\cup L$, where
\begin{enumerate}
    \item[(a)] $A$: $a \in A \iff a \notin K^\ast$ is the set of \emph{active nodes} (non-constant);
    \item[(b)] $U$: $u\in U \iff u\in K^\ast$ and $\exists k \in K^\ast, k \neq u$ such that $x_u = c^\ast_{uk}x_k$
    \item[(c)] $H$: $h \in H \iff h \in K^\ast$ and $\exists g \in \mathfrak{G}(X_K = x_K)$ such that $h \in R(g)$, and
    \item[(d)] $L$:  $\ell \in L\iff \ell \in K^\ast \backslash (H \cup U)$
\end{enumerate}
In addition, we have the following.
\begin{enumerate}
    \item[(e)] For all $k \in H \cup L$, $\pa_{\C}(k) = \pa_{\I}(k)$. 
    \item[(f)] For $k,k' \in H \cup L, k \neq k'$, either $\pa_{\C}(k) \cap \pa_{\C}(k') = \emptyset$ or $\pa_{\C}(k) = \pa_{\C}(k') \neq \emptyset$. 
    \item[(g)] If $h \in H$, then $\pa_{\C}(h) \cap \pa_{\C}(k) = \emptyset$ for all $k \in H \cup L, k \neq h$. 
    \item[(h)] The set $H \cup L$ can be partitioned into equivalence classes where $k\equiv k'\iff \pa_{\C}(k)=\pa_{\C}(k')$. Under this equivalence relation, elements of $H$ are singletons, and $L$ is partitioned into disjoint subsets $L = L_1 \cup \dots \cup L_m$. 
    \item[(i)] Any $\ell\in L$ has at least two parents. 
    \item[(j)] For all $a\in A$ and $\ell\in L$, there exists some $i \in \pa_{\C}(\ell)$ such that $i \notin \pa_{\C}(a)$.
\end{enumerate}
\end{proposition}
\begin{proof}
By definition, $V = A \cup U \cup H \cup L$, and all pairs are mutually disjoint except for possibly $U$ and $H$. Indeed, suppose $u \in U$. Let $k \in K^\ast$ be such that $x_u = c^\ast_{uk}x_k$, $k \neq u$. Let $g \in \mathfrak{G}(X_K = x_K)$. By \eqref{eqn:extra2}, $R_g(u) = R_g(k)$. But $\D$ is a DAG, so $c^\ast_{uk} > 0$ implies $c^\ast_{ku} = 0$, so in particular,  $u \neq R_g(k)$, hence, $u$ cannot be a root in $g$. Thus $u \notin H$, so $U \cap H = \emptyset$. This proves (a) to (d). 
Consider (e). By definition, $\pa_{\C}(k) \subseteq \pa_{\mathcal{I}}(k)$. Suppose for contradiction that the containment is strict, that is, there exists some $i \in V$ such that $i \to k \in \mathcal{I}$ but $i \to k \notin \C$. Then $i \to k \in E^-$. Since $k \in K^\ast$, we must have $i \in K^\ast$. Then $k \in U$, so $k \notin H \cup L$, a contradiction. This proves (e). 
Consider (f). 
 Let $k,k'$ be two such nodes. Let $i \in \pa_{\C}(k) \cap \pa_{\C}(k')$. If this set is empty then we are done. Otherwise, consider ${x_k}/{c^\ast_{ki}}$ and ${x_{k'}}/{c^\ast_{k'i}}$. 
If one of these two quantities are bigger, then either $i \to k$ or $i \to k'$ is not in $g$ for all $g \in \mathfrak{G}(X_K = x_k)$ by \eqref{eqn:extra1}, so $i \notin \pa_{\C}(k) \cap \pa_{\C}(k')$. So these two quantities must be equal. By \eqref{eqn:extra2}, for all $g \in \mathfrak{G}(X_K = x_K)$, $\pa_g(k) = \pa_g(k')$. Thus $\pa_{\mathcal{I}}(k) = \pa_{\mathcal{I}}(k')$. Since $k, k' \in H \cup L$, (e) then implies (f). 
Now consider (g). Suppose for contradiction that there exists some $k \in H \cup L$ such that $\pa_{\C}(h) = \pa_{\C}(k)$. As argued previously, this implies $h$ and $k$ cannot be the root of any $g \in \mathfrak{G}(X_K = x_K)$. So in particular, $h \notin H$, and we obtain the desired contradiction. 
Statement (h) follows immediately from (f) and (g). 
Now we prove (i). Suppose for contradiction that $\ell \in L$ has only one parent $i \in V$. Since $i \to \ell \notin E^-$, $i \to \ell \in \mathcal{I}(X_K = x_K)$. In other words, for all $g \in \mathfrak{G}(X_K = x_K)$, $R_g(\ell) = i$. By Lemma \ref{lem:k.ast.y}(a), this implies $i \in K^\ast$, so $\ell \in U$, a contradiction, as desired.
Now we prove (j). Suppose for contradiction that there exists some $a \in A$ and $\ell \in L$ such that $\pa_{\C}(\ell) \subseteq \pa_{\C}(a)$. Let $r \in \pa_{\C}(\ell)$ be a node with smallest coefficient $c^\ast_{aj}{x_{\ell}}/{c^\ast_{\ell j}}$ among $j \in \pa_{\C}(\ell)$, that is,
$$ c^\ast_{ar}\frac{x_{\ell}}{c^\ast_{\ell r}} \leq c^\ast_{aj}\frac{x_{\ell}}{c^\ast_{\ell j}} \mbox{ for all } j \in \pa_{\C}(\ell). $$
Since $r \to a \in \C(X_K = x_K)$, there exists some $g \in \mathfrak{G}(X_K = x_K)$ such that $r \to a \in g$ and $r \notin K^\ast(g)$. Thus this implies $r \to \ell \notin g$, so there exists some $j \in \pa_{\C}(\ell)$ with $j \to \ell \in g$. Since $g$ is a galaxy, $j \to a \notin g$.  Then by definition, on the event $\mathcal{E}(g)$, $j \to \ell \in g$ and $r \to \ell \notin g$ together imply
$$ x_{\ell} = c^\ast_{\ell j} Z_{j} > c^\ast_{\ell r}Z_{r}. $$
Rearranging gives 
$$ c^\ast_{aj}Z_{j} = c^\ast_{aj}\frac{x_{\ell}}{c^\ast_{\ell j}} \geq c^\ast_{ar}\frac{x_{\ell}}{c^\ast_{\ell r}} > c^\ast_{ar}Z_r,  $$
but this contradicts the fact that $j \to a \notin g$ and $r \to a \notin g$, since these two imply
$$ X_a = c^\ast_{a j'} Z_{j'} < c^\ast_{a r}Z_{r}. $$
So we have a contradiction, as needed.
 \end{proof}
 
 The nodes in $U$ have no direct effect on the conditional distribution, as their effect is mitigated through their (constant) parents.
 Proposition~\ref{cor:dag.partition} is illustrated in Figure~\ref{fig:conveyor}.

\begin{figure}
\begin{center}
\begin{tikzpicture}[->,>=stealth',shorten >=1pt,auto,node distance=2.3cm,semithick]
  \tikzstyle{every node}=[circle,line width =1pt, minimum height =.4cm]
  \node (1) [draw,pattern=crosshatch dots, pattern color = burgundy] {};
  \node (2) [below left of=1,draw,pattern=crosshatch dots, pattern color = burgundy] {};
  \node (4) [below right of=2,draw,fill=red!75] {};
  \node (a) [below =1cm of 4,xshift=1.37cm,draw,pattern color = blue, pattern = vertical lines]{};
 
  \node (3) [right of =2, draw,pattern=crosshatch dots, pattern color = burgundy] {};
  \path (1) edge (2);
  \path (1) edge (3);   
  \path (1) edge (4);
  \path (1) edge (2);
  \path (2) edge (4);
  \path (3) edge (4);
  \node (5) [right of =3,xshift=-1cm, draw,pattern=crosshatch dots, pattern color = burgundy]{};
  \node (6) [below of=5,yshift=.7cm,draw,fill=red!75]{};
  
  
  \node (7) [right of=5, xshift=-1cm,draw,pattern=crosshatch dots, pattern color = burgundy]{};
  \node (8) [right of=6,xshift = -1cm, draw]{};
  
  \node (9) [above right of=7, draw]{};
  \path (1) edge (5);
  \path (5) edge (a);
  \path (6) edge [dashed, color = burgundy] (a);  
  \path (7) edge (a);
  \path (5) edge (6);
  \path (7) edge (6);
  \path (7) edge (8);
  \path (9) edge (7);
  \node (10) [above right of=9, draw]{};
  \node (12) [below of=10, draw,pattern=crosshatch dots, pattern color = burgundy]{};
  \node (16) [right of=12,draw,pattern=crosshatch dots, pattern color = burgundy] {};
  \node (11) [below left of=12,draw] {};  
  \node (13) [right of = 11,xshift=-1cm,draw,fill=red!75] {};
  \node (14) [right of = 13,xshift=-1.2cm,draw,fill=red!75] {};
  \node (15) [right of=14,xshift=-1.2cm,draw,fill=red!75] {};
  
  \node (h) [ below =1cm of 8,draw]{};
  \node (17) [right of=15,xshift=-1.2cm,draw]{};
  \path (8) edge (h);
  \path (6) edge [dashed, color = burgundy] (h);
  \path (10) edge (9);
  \path (9) edge (11);
  \path (10) edge (12);
  \path (12) edge (11);
  \path (12) edge (13);
  \path (12) edge (14);
  \path (12) edge (15);
  \path (12) edge [dashed, color = burgundy] (17);
  
  \path (16) edge (13);
  \path (16) edge (14);
  \path (16) edge (15);
  \path (16) edge [dashed, color = burgundy] (17);
  \node (g) [ below = 1cm of 17,draw]{};
  \node (18) [above right of=17, xshift=-1cm,draw,pattern=crosshatch dots, pattern color = burgundy]{};
  \node (22) [right of=18,draw,xshift=-1cm,pattern=crosshatch dots, pattern color = burgundy] {}; 
  \node (19) [right of = 17,draw,xshift=-1.2cm,fill=red!75] {};
  \node (f) [below of = 19,draw,pattern color = blue, pattern = vertical lines]{};
  \node (20) [right of = 19,xshift=-1.2cm,draw,fill=red!75] {};
  \path (18) edge (17);
  \path (18) edge (19);
  \path (18) edge (20);
  \path (17) edge  (g);
  \path (22) edge (f); 
  \path (18) edge (f);
  \path (19) edge [dashed, color = burgundy] (f);
  \path (22) edge (19);
  \path (22) edge (20);

   \tikzstyle{background}=[rectangle, draw, 
                                                fill=white,
                                                inner sep=2mm,
                                                rounded corners=2mm]
  
      \begin{pgfonlayer}{background}
      
       \node [background,
                    fit= (4) (6), label = below:$H$]{};
      
                     \node [background,
                    fit= (13) (14) (15), label = below:$L_1$]{};
                    
                     \node [background,
                    fit= (19) (20), label = below:$L_2$]{};
  \end{pgfonlayer}
  
\end{tikzpicture}
\end{center}
\caption{Illustration of the partition in Proposition~\ref{cor:dag.partition}. Nodes in $H$ and $L$ are gray, nodes in $U$ are filled with blue lines. Active nodes $A$ are white or dotted burgundy, where the dotted nodes are parents of constant nodes in $K = H \cup L \cup U$. All edges, including the dashed edges, are in $\mathcal{I}(X_K = x_K)$, while only the solid edges are in $\C(X_K = x_K)$.} \label{fig:conveyor}
\end{figure}
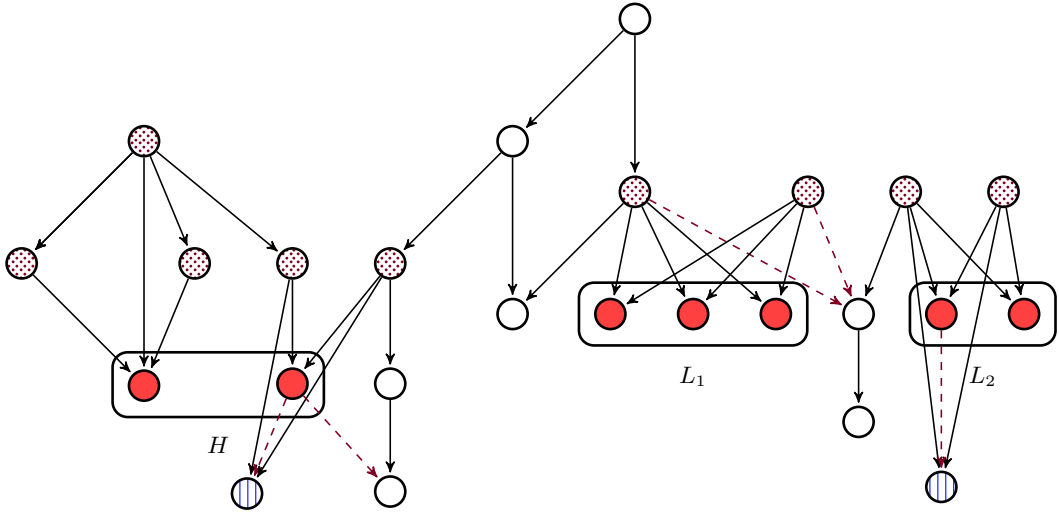

\begin{corollary}\label{cor:swap.out.S}
Let $\C(X_K = x_K)$ be the source DAG of a possible context $\{X_K = x_K\}$ and consider the node partition $V =A\cup K^*= A\cup H \cup L\cup U $ as given by Proposition~\ref{cor:dag.partition}. If $j \to i \in g$ for some $g \in \mathfrak{G}(X_K = x_K)$, $j,i \notin K^\ast$ and $j \in K^\ast(g)$, then $j \to h \in g$ for some $h \in H \cup L$.
\end{corollary}
\begin{proof}
Let  $S=V(\sigma) \subseteq V$ be the set of nodes in the star $\sigma$ with root $j$ in the galaxy $g$. Since $\mathfrak{G}(X_K = x_K) = \mathfrak{G}(X_{K^\ast} = x_{K^\ast})$, apply Lemma~\ref{lem:k.ast.y}(a) to $\mathfrak{G}(X_{K^\ast} = x_{K^\ast})$ giving $S \cap K^\ast \neq \emptyset$. Let $u \in S \cap K^\ast$. If $u \notin U$, then take $h = u$ and we are done. Else, by Proposition~\ref{cor:dag.partition}, there exist some $h \neq u, h \in K^\ast$ such that $x_u = c^\ast_{uh}x_{h}$. By Theorem~\ref{thm:impact}(c), $j \to h \in g$, so $h \in S \cap K^\ast$. If $h \notin U$ then we are done, else we repeat the above argument once more to find another node in $S \cap K^\ast$. Since $\mathcal{D}$ is a DAG, every time we repeat this argument we obtain a new node. Since the graph is finite, this procedure eventually terminates and returns some node $h \in S \cap K^\ast$ and $h \notin U$. By Proposition \ref{cor:dag.partition}, $h \in K \cup L$. 
\end{proof}

\section{Representing the conditional distribution}\label{sec:representation}
Before we derive conditional independence results, we need to have a good control of conditional distributions in a max-linear Bayesian network. We first derive a basic representation in Section~\ref{sec:basic} and subsequently a more compact representation without redundancy in Section~\ref{sec:compact}.

\subsection{Basic representation}\label{sec:basic}
Let $K \subset V$ and $\bar{K} = V \setminus K$. The conditional distribution of $X_{\bar{K}} \cd X_K = x_K$ can be represented by a system of max-linear equations 
in the $Z_{\bar{K}}$ variables; 
more precisely, we have:

\begin{proposition}\label{lem:a.representation}
The following is a representation for $X \cd X_K = x_K$ with respect to the innovations $Z$
\begin{equation}\label{eqn:non.minimal}
X_{\overline{K}} = C^\ast_{\overline{K}K} \odot x_K \, \vee \, C^\ast_{\overline{K}\overline{K}} \odot Z_{\overline{K}}, 
\end{equation}
where the distribution of $Z$ is that of independent components, conditioned to satisfy
\begin{equation}\label{eqn:condition}
x_K = C^\ast_{KK} \odot Z_K \vee  C^\ast_{K\overline{K}} \odot Z_{\overline{K}}. 
\end{equation}
\end{proposition}
\begin{proof}
By \eqref{eqn:idempotent} we have $X =  C^\ast  \odot X$ so 
$$X \geq   
\begin{bmatrix}
0 &  0 \\
C^\ast_{\overline{K}K} & 0
   \end{bmatrix} 
    \odot 
 \begin{bmatrix} 
 X_K \\  X_{\overline{K}}
 \end{bmatrix}.  
$$ Now, $X = C^\ast \odot Z$, therefore, 
$$X = \begin{bmatrix}
0 &  0 \\
C^\ast_{\overline{K}K} & 0
   \end{bmatrix} 
    \odot 
 \begin{bmatrix} 
 X_K \\  X_{\overline{K}}
 \end{bmatrix}   \vee C^\ast  \odot Z
\quad\mbox{and}\quad
X = \begin{bmatrix}
C^\ast_{KK} &  C^\ast_{K\bar{K}} \\
C^\ast_{\bar{K}K} & C^\ast_{\bar{K}\bar{K}}
   \end{bmatrix} 
    \odot 
 \begin{bmatrix} 
 Z_K \\  Z_{\overline{K}}
 \end{bmatrix}.$$
Writing out these equations, 
we obtain 
\begin{align}
X_{\overline{K}} &= C^\ast_{\overline{K}K}  \odot x_K \vee
C^\ast_{\overline{K}K} \odot Z_K \vee
 C^\ast_{\overline{K}\overline{K}} \odot Z_{\overline{K}}, \label{eqn:redundant.bark}  \\
x_K &= C^\ast_{KK}\odot Z_K \vee C^\ast_{K\overline{K}} \odot Z_{\overline{K}}. \label{eqn:redundant.k}
\end{align}
The second equation is \eqref{eqn:condition}. For the first equation, note that $c^\ast_{ii} = 1$ for all $i$, so $x_K \geq Z_K$. Thus
$C^\ast_{\overline{K}K} \odot x_K \vee
C^\ast_{\overline{K}K} \odot Z_K = C^\ast_{\overline{K}K} \odot x_K$,
so \eqref{eqn:redundant.bark} is equivalent to \eqref{eqn:non.minimal}. Thus the context $\{X_K=x_K\}$ is equal to the conjunction of the events \eqref{eqn:condition} and \eqref{eqn:redundant.k}. The result follows.
\end{proof}

\begin{example}\label{ex:cassiopeia2} 
We shall illustrate Proposition~\ref{lem:a.representation} for the Cassiopeia graph of Figure~\ref{fig:cassiopeia0} in Example~\ref{ex:cassiopeia0}.
Assume that we have $c_{ji}=c^*_{ji}=1$ for all edges in this DAG and let $x_K=(x_4,x_5)$. Then \eqref{eqn:non.minimal} becomes
\[\begin{bmatrix}X_1\\X_2\\X_3\end{bmatrix}= \begin{bmatrix}0&0\\0&0\\0&0\end{bmatrix}\odot\begin{bmatrix}x_4\\x_5\end{bmatrix} \vee I_3\odot \begin{bmatrix}Z_1\\Z_2\\Z_3\end{bmatrix}=  \begin{bmatrix}Z_1\\Z_2\\Z_3\end{bmatrix},\]
whereas \eqref{eqn:condition} becomes 
\[\begin{bmatrix}x_4\\x_5\end{bmatrix} = \begin{bmatrix}1&0\\0&1\end{bmatrix}\odot\begin{bmatrix}Z_4\\Z_5\end{bmatrix} \vee \begin{bmatrix}1&1&0\\0&1&1\end{bmatrix}\odot\begin{bmatrix}Z_1\\Z_2\\Z_3\end{bmatrix}. \] 
This means that $x_4 \geq Z_4$, $x_5 \geq Z_5$ and 
\[\begin{bmatrix}x_4\\x_5\end{bmatrix}\geq \begin{bmatrix}1&1&0\\0&1&1\end{bmatrix}\odot\begin{bmatrix}Z_1\\Z_2\\Z_3\end{bmatrix}=  \begin{bmatrix}Z_1\vee Z_2\\Z_2\vee Z_3\end{bmatrix}. \] 

Depending on whether $x_4<x_5$,   $x_4>x_5$, or $x_4=x_5$, these inequalities are a.s.\ equivalent to respectively
\[\begin{bmatrix}x_4\\x_5\end{bmatrix}\geq \begin{bmatrix}Z_1\vee Z_2\\ Z_3\end{bmatrix}, \quad
\begin{bmatrix}x_4\\x_5\end{bmatrix}\geq \begin{bmatrix}Z_1\\ Z_2\vee Z_3\end{bmatrix}, \quad
\begin{bmatrix}x_4\\x_5\end{bmatrix}\geq \begin{bmatrix}Z_1\\ Z_3\end{bmatrix} \text{ and } Z_2=x_4=x_5.\]
Thus the conditioning under this restriction renders $Z_i$ bounded in all cases, and in the third case $Z_2$ becomes  constant. Note also that in these reduced inequalities, $Z_1$ and $Z_3$ never occur together in any inequality, rendering $X_1\ci X_3\cd X_{\{4,5\}}$. 
\halmos
\end{example}

\subsection{Compact representation}\label{sec:compact}

The main result of this section, Theorem \ref{thm:source.rep}, states that the source DAG gives a representation of the active nodes $X_A \cd X_K = x_K$ with respect to the innovations $Z$. Compared to the representation of the conditional distribution in Proposition~\ref{lem:a.representation}, this is a representation with fewer terms. Most importantly, we shall show below that the system of equations involving $Z$ can be separated into blocks where no terms are redundant.

\begin{theorem}\label{thm:source.rep} Let $\C(X_K = x_K)$ be the source DAG of a possible context $\{X_K = x_K\}$, with 
node partition $V= A \cup H \cup L \cup U= A \cup H \cup (L_1 \dots \cup L_m)\cup U$ as in Proposition~\ref{cor:dag.partition}. 
For each $t = 1,\dots,m$, select a node $\ell_t \in L_t$. 
Then the following system of equations yields a representation for $X_A \cd X_K = x_K$ with respect to $Z$: 
\begin{equation}\label{eqn:minimal.representation}
    X_a = \alpha_a \vee Z_a \vee \bigvee_{j \in \pa_{\C}(a)} c^\ast_{aj}Z_j, \quad 
    a \in A, 
\end{equation}
where the constants $\alpha_a$ are given by 
\begin{equation}\label{eqn:alpha.i}
\alpha_a = \bigvee_{k \in K^\ast} c^\ast_{ak}x_{k} \vee\left(\bigvee_{\substack{j\in A \\ j \to a \in E^-}}\; \bigvee_{k \in \ch_{\D}(j)\cap (H\cup L)} c^\ast_{aj} \, \frac{x_{k}}{c^\ast_{kj}}\right), \quad a\in A, 
\end{equation}
and the distribution of $Z$ 
is that of independent components, conditioned to satisfy the bounds  
\begin{align}
Z_i &\leq \bigwedge_{k \in K^\ast: c^\ast_{ki} >0}\frac{x_k}{c^\ast_{ki}}, \quad 
i \in V,\label{eqn:z.bound}\end{align}
as well as the equations
\begin{align}
x_h &= Z_h \vee \bigvee_{j \in \pa_{\C}(h)} c^\ast_{hj}Z_j, \quad
h \in H,   \label{eqn:equality.1} \\
x_{\ell_t} &= \bigvee_{j \in \pa_{\C}(\ell_t)} c^\ast_{\ell_t j}Z_j, \quad 
t = 1,\dots,m.  \label{eqn:equality.2}
\end{align}
Furthermore, each innovation term on the right-hand side of \eqref{eqn:minimal.representation}, \eqref{eqn:equality.1} and \eqref{eqn:equality.2} has positive probability of being the term that achieves equality. 
\end{theorem}

\begin{proof} 
We shall begin with the representation of $X \cd X_K = x_K$ given by Proposition~\ref{lem:a.representation} and then simplify terms until we obtain the representation above. The contexts $\{X_K = x_K\}$ and $\{X_{K^\ast}=x_{K^\ast}\}$ are clearly equivalent, so we may without loss of generality assume that $K = K^\ast$ 
and $A = \bar{K}$. This gives {for \eqref{eqn:non.minimal} and \eqref{eqn:condition}} the representations
\begin{align}
    X_A &= C^\ast_{A K^\ast}\odot x_{K^\ast} \vee C^\ast_{AA}\odot Z_A, \label{eqn:top.2} \\
    x_{K^\ast} &= C^\ast_{K^\ast K^\ast}\odot Z_{K^\ast} \vee C^\ast_{K^\ast A}\odot Z_A. \label{eqn:bottom.2}
\end{align}
First we simplify \eqref{eqn:bottom.2}.
With $K^*=H\cup L\cup U$ we expand this system of equations as follows:
\begin{align}
    x_{H} &= C^\ast_{H K^\ast}\odot Z_{K^\ast} \vee C^\ast_{H A}\odot Z_A \label{eqn:bottom2.h} \\
    x_{L} &= C^\ast_{L K^\ast}\odot Z_{K^\ast} \vee C^\ast_{L A}\odot Z_A \label{eqn:bottom2.l} \\
    x_{U} &= C^\ast_{U K^\ast}\odot Z_{K^\ast} \vee C^\ast_{U A}\odot Z_A \label{eqn:bottom2.u}.
\end{align}
For each $i \in V$ 
and each $k \in K^\ast$, all inequalities on $Z_i$ implied by \eqref{eqn:bottom.2} are
$$ x_k \geq c^\ast_{ki}Z_i  $$ whenever $c^\ast_{ki} > 0$, and in particular this is equivalent to \eqref{eqn:z.bound}. 

Next we keep track of the equalities. For $u \in U$, by Proposition~\ref{cor:dag.partition}, $x_u = c^\ast_{uk}x_k$ for some $k \in K^\ast$, $k \neq u$. We have
\begin{align*}
x_u &= c^\ast_{uk}x_k \\
&= c^\ast_{uk}C^\ast_{kK^\ast}\odot Z_{K^\ast} \vee c^\ast_{uk}C^\ast_{kA}\odot Z_A \quad\mbox{ by } \eqref{eqn:bottom.2} \\
&\leq C^\ast_{uK^\ast}\odot Z_{K^\ast} \vee C^\ast_{uA}\odot Z_A \\
&= x_u \quad \mbox{ by } \eqref{eqn:bottom.2}.
\end{align*}
Thus we conclude
$$C^\ast_{uK^\ast}\odot Z_{K^\ast} \vee C^\ast_{uA}\odot Z_A=c^\ast_{uk}\left(C^\ast_{kK^\ast}\odot Z_{K^\ast} \vee C^\ast_{kA}\odot Z_A\right),\quad k \in K^\ast, k \neq u,
$$
whence
the constraint imposed upon $Z$ by $x_u$ is identical to the constraint imposed upon $Z$ by $x_k$. Therefore all equations in \eqref{eqn:bottom2.u} are redundant. So \eqref{eqn:bottom.2} is equivalent to the conjunction of \eqref{eqn:bottom2.h} and \eqref{eqn:bottom2.l}. 

Next we simplify the terms that appear on the right-hand side of \eqref{eqn:bottom2.h} and \eqref{eqn:bottom2.l}.
Fix $k \in H \cup L$. 
By definition of the source DAG, we keep a term $c^\ast_{ki}Z_i$ for $i \neq k$ if and only if $i \to k \in g$ for some $g \in \mathfrak{G}(X_K = x_K)$, and we keep the term $c^\ast_{kk}Z_k = Z_k$ if and only if $k \in R(g)$ for some $g \in \mathfrak{G}(X_K = x_K)$. Since each $g$ is a galaxy, each remaining term has a positive probability of achieving the maximum. 
 Since each event under $\mathfrak{G}(X_K = x_K)$ with positive probability must correspond to some $g$,  there is always one that achieves the maximum among the remaining terms. 
Since $k \in K^\ast$, it holds that $i \to k \in g$ for some $g \in \mathfrak{G}(X_K =x_K)$ if and only if we also have $i \to k \in \C(X_K =x_K)$. 
Therefore, by Proposition~\ref{cor:dag.partition}, \eqref{eqn:bottom2.h} simplifies to \eqref{eqn:equality.1}, and \eqref{eqn:bottom2.l} simplifies to
\begin{equation}\label{eqn:x.ell}
x_\ell = \bigvee_{j \in \pa_{\C}(\ell)}c^\ast_{\ell j}Z_j. 
\end{equation}
Write $L = L_1 \cup \dots \cup L_m$ as given by Proposition~\ref{cor:dag.partition}.
If $\ell,\ell' \in L_t$ for some $t = 1,\dots,m$, then they share the same set of parents. By Lemma \ref{lem:y.xk}, this implies
$$ \frac{x_\ell}{c^\ast_{\ell j}} = \frac{x_{\ell'}}{c^\ast_{\ell' j}} $$
for all $j \in \pa_{\C}(\ell) = \pa_{\C}(\ell')$. So \eqref{eqn:x.ell} for $x_\ell$ and $x_{\ell'}$ are constant multiples of each other. So for each $t = 1,\dots,m$ \eqref{eqn:x.ell} for all $\ell \in L_t$ is equivalent to the single equation \eqref{eqn:equality.2}. 

Now we simplify \eqref{eqn:top.2}. Like in the previous step, we can for $a,j\in A$ drop terms $c^\ast_{aj}Z_j$ where $j \to a \notin g$ for every $g \in \mathfrak{G}(X_K = x_K)$. 
This gives
$$ X_a = C^\ast_{aK^\ast}\odot x_{K^\ast} \vee \left(\bigvee_{j \in A, j \to a \in E^-} c^\ast_{aj}Z_j\right) \vee Z_a \vee \bigvee_{j \in \pa_{\C}(a)}c^\ast_{aj}Z_j,\quad {a\in A}. $$
Now we argue that each term in $\bigvee_{j\in A, j \to a \in E^-} c^\ast_{aj}Z_j$
can be replaced by an appropriate constant. Let $j \in A$ be a node with $j \to a \in E^-$. 
Let $\E$ be the sub-event of $\mathfrak{G}(X_K = x_K)$ where $j$ is the root of $a$, that is, 
$$ \E = \bigcup\left\{\E(g): g \in \mathfrak{G}(X_K = x_K), j \to a \in g \right\}. $$
By definition, for each $g \in \mathfrak{G}(X_K = x_K)$ such that $j \to a \in g$, we have that $a,j \in K^\ast(g)$. Since $a,j \notin K^\ast(g)$, it follows that there exists some $k = k(g) \in K^\ast(g)$ such that $j \to k\in g$.  Thus, {on the event} $\E(g) \cap \{X_K = x_K\}$, we have
$$ X_a = c^\ast_{aj}Z_j = c^\ast_{aj}\frac{x_{k(g)}}{c^\ast_{k(g)j}},\quad {a\in A}. $$
By Corollary \ref{cor:swap.out.S}, we may assume that $k(g) \in H \cup L$. Therefore, on $\E$, 
\begin{equation}\label{eqn:on.e}
X_a = \bigvee_{k \in \ch_{\D}(j)\cap (H\cup L)} c^\ast_{aj}\frac{x_{k}}{c^\ast_{kj}},\quad {a\in A}. 
\end{equation}
By definition of $\E$, on the complement $\{X_K = x_K\} \setminus \E$, $X_a > c^\ast_{aj}Z_j$. Therefore, the term $c^\ast_{aj}Z_j$ can be dropped from the representation of $X_a$ and be replaced by the right-hand side of \eqref{eqn:on.e}. This gives
$$ X_a = C^\ast_{aK^\ast}\odot x_{K^\ast} \vee \left(\bigvee_{j\in A, j \to a \in E^-}\; \bigvee_{k \in \ch_{\D}(j)\cap (H\cup L)} c^\ast_{aj}\frac{x_{k}}{c^\ast_{kj}}\right) \vee Z_a \vee \bigvee_{j \in \pa_{\C}(a)}c^\ast_{aj}Z_j,\quad {a\in A}. $$
This is \eqref{eqn:minimal.representation}, with $\alpha_a$ equal to the constant terms in the equation above, which is the formula in \eqref{eqn:alpha.i}. 
Finally, by definition of $\C(X_K = x_K)$, for each $j \in \pa_{\C}(i)$ there exists some $g \in \mathfrak{G}(X_K = x_K)$ such that $j \to i$ and $j \notin K^\ast(g)$. Since $g$ is a galaxy and $X_i$ is \emph{not} constant on the  event $\E(g) \cap \{X_K = x_K\}$, on this event, $c^\ast_{ij}Z_j$ is the unique term that achieves the maximum in \eqref{eqn:minimal.representation}.
\end{proof}

\begin{remark}\label{rem:3.2}
We note that in \eqref{eqn:z.bound}, only the bounds for the variables $Z_{A\cup H}$ are directly relevant for the conditional distribution of $X_A$ given $X_K=x_K$, as the variables $Z_{L\cup U}$ do not enter into any of the equations \eqref{eqn:minimal.representation}, \eqref{eqn:alpha.i}, \eqref{eqn:equality.1}, or \eqref{eqn:equality.2}. However, we have included these in Theorem~\ref{thm:source.rep} to provide a full description of the conditional distribution of $Z$ given $X_K=x_K$, which may be of interest for other purposes.
\end{remark}

The following example of an umbrella graph illustrates some aspects of this representation.

\begin{figure}
\begin{center}
\begin{tikzpicture}[->,>=stealth',shorten >=1pt,auto,node distance=2cm,semithick]
  \tikzstyle{every node}=[circle,line width =1pt,font=\footnotesize]
  \node (1) [draw] {$1$};
  \node (4) [below left of=1,xshift=1.7cm,draw]{$4$}; 
  \node (5) [right of=4,draw] {$5$};
  \node (2) [below left of=4,draw] {$2$};  
  \node (6) [right of = 2,draw,fill=red!75] {$6$};
  \node (7) [right of = 6,draw,fill=red!75] {$7$};
  \node (3) [right of=7,draw] {$3$};
  \path (1) edge [densely dashdotted, blue](2);
  \path (1) edge [densely dashdotted, blue,bend left=20, align=left, below](3);
  \path (1) edge (7);    
  \path (4) edge [densely dashdotted, blue](2);
  \path (4) edge (6);
  \path (4) edge (7);
  \path (4) edge (3);
  \path (5) edge (2);
  \path (5) edge (6);
  \path (5) edge (7);
  \path (5) edge  [densely dashdotted, blue](3);
  
  \node (a1) [right of=1, xshift= 6cm, draw] {$1$};
  \node (a4) [below left of=a1,xshift=1.7cm,draw,pattern=crosshatch dots, pattern color = burgundy,pattern color = burgundy] {$4$};
  \node (a5) [right of=a4,draw, pattern=crosshatch dots, pattern color = burgundy] {$5$};
  \node (a2) [below left of=a4,draw] {$2$};  
  \node (a6) [right of = a2,draw,fill=red!75] {$6$};
  \node (a7) [right of = a6,draw,fill=red!75] {$7$};
  \node (a3) [right of=a7,draw] {$3$};
  \path (a1) edge [densely dashdotted, blue, ](a2);
  \path (a1) edge [densely dashdotted, blue,bend left=20, align=left, below](a3); 
  \path (a4) edge [densely dashdotted, blue](a2);
  \path (a4) edge (a6);
  \path (a4) edge (a7);
  \path (a5) edge (a6);
  \path (a5) edge (a7);
  \path (a5) edge [densely dashdotted, blue](a3);
  
\end{tikzpicture}
\end{center}
\caption{The umbrella graph: 
To the left,  $\D(C)=\D^*(C)=\D^*_K(C)$; to the right: the source DAG $\C(X_K = x_K)$ for $K= \{ 6, 7 \}$ and $x_6=x_7=3$. Black edges have weights 1, dash-dotted blue edges have weights 2.} \label{fig:umbrella}
\end{figure}
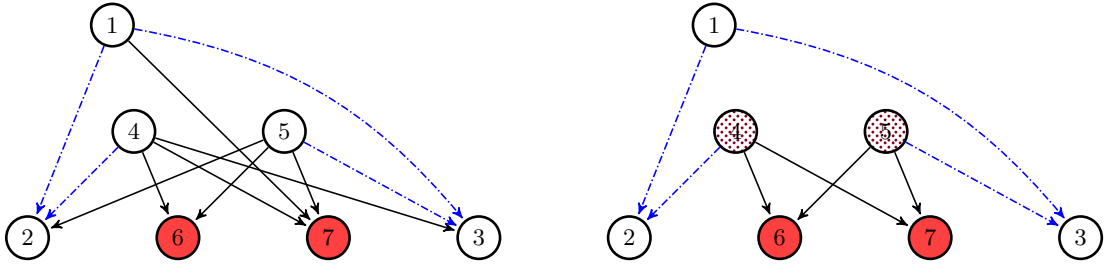
\begin{example}[Umbrella]\label{ex:umbrella}
Consider the graph to the left in Figure \ref{fig:umbrella} where solid black edges have weights 1 and dash-dotted blue edges have weights 2.
The partition in Theorem~\ref{thm:source.rep} yields nodes $A=V\setminus K^*=V\setminus K= \{ 1,2,3,4,5 \}$, and $K=L=\{6,7\}$.  The non-zero constants \eqref{eqn:alpha.i} in the representation of the active variables are $\alpha_2=\alpha_3=3$, calculated as 
$$\alpha_2=0\vee \bigvee_{5: 5\to 2}\;\bigvee_{k\in \{6,7\}}c^\ast_{25}\frac{x_k}{c^\ast_{k5}}=\max(3,3)=3,$$ since $c^\ast_{2k}=0$ for $k=6,7$ and there is only one edge $5\to 2\in E^-$ pointing to $2$. A similar calculation yields  $\alpha_3=3$. The full representation \eqref{eqn:minimal.representation} then becomes
\begin{align*}
    X_1 &= Z_1 ,\quad X_4 = Z_4 ,\quad X_5 = Z_5 \\
    X_2 &= 3\vee Z_2 \vee 2Z_1 \vee 2Z_4  \\
    X_3 &= 3\vee Z_3 \vee 2Z_1 \vee 2Z_5  
\end{align*}
 with inequalities from \eqref{eqn:z.bound} yielding the bounds $Z_1, Z_4, Z_5 \leq 3$ (and $Z_6,Z_7\leq 3$). 
 Further, we have the equality \eqref{eqn:equality.2} yielding
 \begin{equation}\label{eqn:lbounds.ex}x_6=x_7=3=Z_4\vee Z_5\end{equation}
 whereas \eqref{eqn:equality.1} is void.
 
 For $Z_1$, we claim that it cannot simultaneously  achieve the bound in both of the equations for $X_2$ and $X_3$, illustrating Proposition~\ref{cor:dag.partition}(g). In fact, if $X_2 = 2Z_1$ then $3 < 2Z_1 \leq 6$, 
 so that $2Z_4 < 2Z_1\leq 6$. Then \eqref{eqn:lbounds.ex} yields that $Z_4 < 3$ and thus $Z_5=3$, but then $X_3 = 6 > 2Z_1$.
 Moreover, since $Z_4$ and $Z_5$ both enter into the equation \eqref{eqn:lbounds.ex} we conclude that $ X_2 \notci X_3 \cd  X_{\{6,7\}} =(3,3).$ \halmos 
\end{example}

We next present some important consequences of Theorem~\ref{thm:source.rep}, enabling us to identify conditional independencies.

\begin{corollary}\label{lem:z.independent}
For each pair $i,j \in V$,  either
$Z_i, Z_j$ appear together in exactly one equation among those in \eqref{eqn:equality.1} and \eqref{eqn:equality.2}, or they do not appear together in any of those equations. In the first case they are conditionally dependent, i.e.\  $Z_i\notci Z_j \cd X_K = x_K$. In the second case they are conditionally independent, i.e.\  $Z_i \ci Z_j \cd X_K = x_K$. 
\end{corollary}

\begin{proof}
By Theorem~\ref{thm:source.rep}, the distribution of $Z\cd X_K = x_K$ is the distribution of $Z$ given the events defined by  \eqref{eqn:z.bound}, \eqref{eqn:equality.1}, and \eqref{eqn:equality.2}. The bounds \eqref{eqn:z.bound} only involve one variable at a time and thus play no role for independence issues.
Groups of $Z$'s that appear in different equations in \eqref{eqn:equality.1} and \eqref{eqn:equality.2} are independent. 
It remains to show that, if $Z_i,Z_j$ appear in the same equation, then they are dependent. Indeed, suppose that $Z_i,Z_j$ appear in \eqref{eqn:equality.1} for some $h \in H$, with coefficients $a_i,a_j > 0$. The event $\E_h$ defined by this equation can be rewritten as 
$$a_iZ_i \leq x_h, a_jZ_j \leq x_h\quad \mbox{or}\quad a_iZ_i \leq x_h, a_jZ_j \leq x_h, a_{j'}Z_{j'} \leq x_h \mbox{ for some other } j' \in \pa_{\C}(h),$$ and exactly one of these terms achieves equality.  
Further, 
each term has a positive probability of achieving  equality. That is,
$$ \P(a_iZ_i = x_h \cd \E_h)> 0 \text{ and } \P(a_jZ_j = x_h \cd \E_h) > 0,$$
but
$$ \P(a_jZ_j = x_h \cd a_iZ_i = x_h, \E_h) = 0.$$
Therefore, $Z_i \notci Z_j \cd \E_h$. A similar argument applies for the equation \eqref{eqn:equality.2}.
\end{proof}

\begin{corollary}\label{lem:i.k.common.source}
For $i,j \in A$, $j \to i \in \C(X_K = x_K)$, suppose that 
$j\to k\in \C(X_K = x_K)$ for some $k\in H\cup L$. Then there exists $g \in \mathfrak{G}(X_K = x_K)$ such that $j \to i, j \to k \in g$. In other words, 
\begin{equation}\label{eqn:i.k.common.source}
\P(X_i = c^\ast_{ij}Z_j, Z_j = x_k/c^\ast_{kj}\cd X_K = x_K) > 0
\end{equation}
\end{corollary}
\begin{proof} 
Since $j \to i \in \C(X_K = x_K)$, there exists some $g \in \mathfrak{G}(X_K = x_K)$ such that $j \to i \in g$ and $j \notin K^\ast(g)$, so in particular, $j \to k \notin g$. Let $\E(g)$ as usual denote the $Z$-values corresponding to the impact graph $g$. 
Since $j \to k \notin g$, there exists some other $j' \neq j$ such that $z_{j'} = x_k/c^\ast_{kj'}$ and 
$z_j < x_k/c^\ast_{kj}$ for all $z \in \E(g)$. Transform the region $\E(g)$ to another region $\phi(\E(g))$ via the following linear map $\phi$, where
$$ \phi(z)_j = x_k/c^\ast_{kj}, \phi(z)_{j'} = z_j, \phi(z)_{j''} = z_{j''} \mbox{ for all } j'' \neq j,j'. $$
Since this is an invertible map and $\P(\E(g)\cd X_K = x_K) > 0$, we have $\P(\phi(\E(g))\cd X_K = x_K) > 0$.
Thus there exists some $g' \in \mathfrak{G}(X_K = x_K)$ such that $\P(\mathcal{E}(g') \cap \{X_K = x_K\} \cap \phi(\E(g))) > 0$. By definition of $\phi$, for such $g'$ we must have $j \to i, j \to k \in g'$. 
This concludes the proof. 
\end{proof}

\begin{corollary}\label{cor:atomic}
For each $a \in A$, the atomic component of the distribution of $X_a$ is supported precisely on the following points: 
\begin{enumerate}[{\rm(a)}] 
    \item $\alpha_a$ defined by \eqref{eqn:alpha.i} if $\alpha_a > 0$
    \item ${c^\ast_{aj}x_k}/{c^\ast_{kj}}$ for each $j \in \pa_{\C}(a) \cap \pa_{\C}(k)$ for some $k \in H \cup L$ 
\end{enumerate}
\end{corollary}

\begin{proof} 
Suppose $X_a$ has an atomic component at some $c \in \Rplus$. This happens if and only if there exists some $g \in \mathfrak{G}(X_K = x_K)$ such that $X_a = c$ on $\mathcal{E}(g) \cap \{X_K = x_K\}$. In particular, we must have $a \in K^\ast(g)$. Since $a \notin K^\ast$, $a \in K^\ast(g)$ if and only if $j \to a \in g$ for some $j \in V$, and either $j \in K^\ast$ or $\ch_g(j) \cap K^\ast \neq \emptyset$. We consider these two cases separately. 
\begin{enumerate}[{\rm(a)}]
    \item Suppose $j \in K^\ast$. Then $c = c^\ast_{aj}x_j \leq \alpha_a$ by \eqref{eqn:alpha.i}. If $c < \alpha_a$ then $\P(X_a = c \cd  X_K = x_K) = 0$ by \eqref{eqn:minimal.representation}, a contradiction. So $c = \alpha_a$.
    \item Suppose $j \notin K^\ast$. By Corollary \ref{cor:swap.out.S}, there exists some $k \in \ch_g(j) \cap H \cap L$. By Corollary \ref{lem:i.k.common.source}, $\P(X_a = c^\ast_{aj}Z_j, Z_j = x_k/c^\ast_{jk}\cd  X_K = x_K) > 0$, so the distribution of $X_a$ has an atom at ${c^\ast_{aj}x_k}/{c^\ast_{kj}}$. 
\end{enumerate}
So all the atomic components of the distribution of $X_a$ must be of the form given. 
\end{proof}

\section{Markov properties of max-linear Bayesian networks}\label{sec:cond.indep}

In this section we introduce the relevant separation criteria and state and prove the three conditional independence theorems. We first consider the most difficult context-specific case and then use the results for this case to derive the more generic results which are valid in all contexts. 

\subsection{Graphs and separation}\label{sec:separation}

In addition to the source DAG as defined in Definition~\ref{defn:context.graph}, we shall need the following graphs to identify Markov properties of a max-linear Bayesian network. 
\begin{definition}\label{dfn:3.1}
Fix a DAG $\D$ on $V$ and $K \subset V$. Say that a directed path $\pi$ from $j$ to $i$ \emph{factors through $K$} if there exists a node $k \in \pi$, $k \neq i,j$ such that $k \in K$. The \emph{conditional reachability DAG} $\D^\ast_K$ is the graph on $V$ consisting of the following edges: $j \to i \in \D^\ast_K$ if and only if there exists a directed path from $j$ to $i$ that does \emph{not} factor through $K$. 
\end{definition}

\begin{definition}\label{dfn:3.2}
Fix a DAG $\D$ on $V$, $K \subset V$ and a coefficient matrix $C$ supported by $\D$. The \emph{critical DAG} $\D^\ast_K(C)$ is the graph on $V$ consisting of the following edges: $j \to i \in \D^\ast_K(C)$ if and only if $c^\ast_{ij} > 0$ and no critical directed path from $j$ to $i$ factors through $K$. 
\end{definition}
Note that in contrast to Definition~\ref{dfn:3.1} the existence of a \emph{single} critical path through $K$ removes the corresponding edge in the critical DAG $\D^\ast_K(C)$; this conforms with Example~\ref{ex:diamond0} in the introduction where it is sufficient to block a single critical path to obtain conditional independence.


When $K = \emptyset$, we write $\dag^*=\dag^*_\emptyset$ for the reachability DAG of $\D$, and $\dag^*(C)=\dag^*_\emptyset(C)$ if the support of $C$ is $\dag$. 
The source DAG $\C(X_K = x_K)$ for $K=\emptyset$ does not have a direct meaning, but by convention we let this be $\mathcal{C}(X_\emptyset=x_\emptyset)=\dag^*$.

\begin{lemma}\label{lem:inclusion}
Let $C$ be a coefficient matrix with support $\D$. Furthermore, let $K \subset V$ and let $\{X_K = x_K\}$ be a possible context.
Then 
\begin{align}\label{eqn:inclusion}
\D^\ast_K \supseteq \D^\ast_K(C) \supseteq \C(X_K = x_K).
\end{align}
\end{lemma}

\begin{proof}
First we prove that $\D^\ast_K(C) \subseteq \D^\ast_K$. Let $j \to i \in \D^\ast_K(C)$. Since $c^\ast_{ji} > 0$ and no critical directed paths from $j$ to $i$ factor through $K$, there exists at least one critical directed path from $j$ to $i$ that does not factor through $K$. Therefore, $j \to i \in \D^\ast_K$. Now we prove that $\C(X_K = x_K) \subseteq \D^\ast_K(C)$. Suppose $j \to i \in \C(X_K = x_K)$. Clearly we must have $c^\ast_{ji} > 0$. Suppose for contradiction that $j \to k \to i$ is critical for some $k \in K$. Then on $\{X_K = x_K\}$, 
$$ c^\ast_{ij}Z_j = c^\ast_{ik}c^\ast_{kj}Z_j \leq c^\ast_{ik}x_k. $$
So $j \to k \notin \C(X_K = x_K)$, a contradiction. Therefore all critical paths from $j$ to $i$ do not factor through $K$, so $j \to i \in \D^\ast_K(C)$ by definition. 
\end{proof}

\begin{definition}\label{dfn:3.3}
An undirected path $\pi$ between $j$ and $i$ in a DAG is \emph{$*$-connecting} relative to $K$ if and only if it is one of the paths in Figure~\ref{fig:d.ast.connecting}.
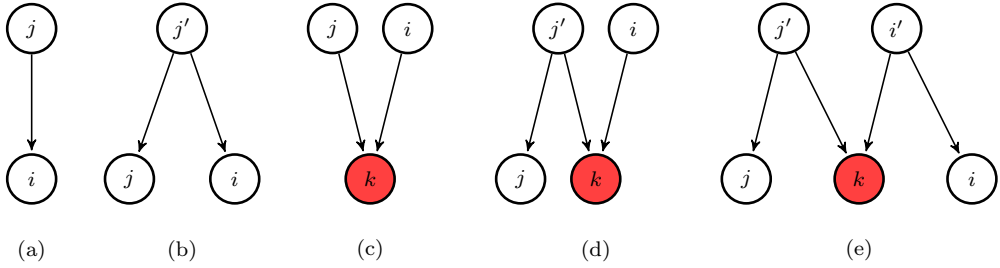
\begin{figure}
\begin{center}
\begin{tikzpicture}[->,>=stealth',shorten >=1pt,auto,node distance=2cm,semithick]
  \tikzstyle{every node}=[circle,line width =1pt,font=\scriptsize,minimum height =0.65cm]
  \node (i0) [draw] {$j$};
  \node (j0) [below of=i0,draw] {$i$} ;
  \path (i0) edge (j0) ;
  \node [below = 2.2cm of i0] {(a)};
  
\node (ksl) [right of =i0,draw]{$j'$};
\node (jsl) [below of = ksl,xshift =-0.7cm,draw] {$j$};
\node (isl) [below of = ksl,xshift =0.7cm,draw] {$i$};
\path (ksl) edge (jsl);
\path (ksl) edge (isl);

\node [below = 2.2cm of ksl] {(b)};

\node (i1) [right of =ksl, draw] {$j$};
  \node (j1) [below of=i1,xshift=0.5cm,draw,fill=red!75] {$k$} ;
  \node (k1) [right of=i1,xshift=-1cm,draw] {$i$} ;
  \path (i1) edge (j1) ;
  \path (k1) edge (j1);  
  
  \node [below= 2.2cm  of i1,xshift =0.5cm] {(c)};

  
  \node (i2) [draw,right of = i1,xshift=1cm] {$j'$};
  \node (j2) [below of=i2,xshift=0.5cm,draw,fill=red!75] {$k$} ;
  \node (k2) [right of=i2,xshift=-1cm,draw] {$i$} ;
  \node (i2p) [draw,left of = j2,xshift=1cm] {$j$};    
  \path (i2) edge (j2) ;
  \path (k2) edge (j2);   
  \path (i2) edge (i2p);
  
  \node [below= 2.2cm  of i2,xshift =0.5cm] {(d)};
  
  
  \node (i3) [draw,right of = k2] {$j'$};
  \node (j3) [below of=i3,xshift=1cm,draw,fill=red!75] {$k$} ;
  \node (k3) [right of=i3,xshift =-0.5cm,draw] {$i'$} ;
  \node (i3p) [draw,left of = j3,xshift=0.5cm] {$j$};    
  \node (j3p) [draw,right of = j3,xshift=-0.5cm] {$i$};      
  \path (i3) edge (j3) ;
  \path (k3) edge (j3);   
  \path (i3) edge (i3p);  
  \path (k3) edge (j3p);  
  
  \node [below= 2.2cm  of i3,xshift =1cm] {(e)};
\end{tikzpicture}
\end{center}
\caption{
Types of $*$-connecting paths between $i$ and $j$. Nodes that are colored red are in $K$. 
} \label{fig:d.ast.connecting}
\end{figure}
\end{definition}

We shall consider $\ast$-connecting paths in the conditional reachability DAG $\dag^\ast_K$, in the critical DAG $\D^*_K(C)$, and in the source DAG $\C(X_K = x_K)$. Edges in these DAGs represent directed paths in the original DAG $\D$ and each of the paths in Figure~\ref{fig:d.ast.connecting}  may represent longer paths in the original DAG $\D$.
Note also that any  $\ast$-connecting path in a derived DAG corresponds to a $d$-connecting path in $\dag$, but not vice versa, as illustrated in Example~\ref{ex:cassiopeia1} below.

We now define three independence models by applying $\ast$-separation to $\D^\ast_K$, $\D^\ast_K(C)$ and the source DAG $\C(X_K=x_K)$, respectively. 

\begin{definition}\label{def:reachsep}
For three disjoint subsets $I$, $J$, and $K$  of the node set $V$ we say that $I$ and $J$ are \emph{$\D^\ast$-separated} 
by $K$ in $\D$ if there are no $*$-connecting paths from $I$ to $J$ in $\dag^*_K$ and  we then write $I\tildese J\cd K$ or $I\starse J\cd K$ in $\dag^*_K$.
\end{definition}

\begin{definition}\label{def:critsep}
For three disjoint subsets $I$, $J$, and $K$  of the node set $V$ we say that $I$ and $J$ are \emph{critically separated} 
by $K$ in $\dag$ if there is no $*$-connecting path $\pi$ from  $I$ to $J$ in $\dag^*_K(C)$.
We then write $I\critsep J\cd K$ {or  $I\starse J\cd K$ in $\dag^*_K(C)$.}
\end{definition}

\begin{definition}\label{def:contextsep}
For three disjoint subsets $I$, $J$, and $K$  of the node set $V$ we say that $I$ and $J$ are \emph{source separated} 
by $X_K = x_K$ in $\dag$ if there are no $*$-connecting paths from  $I$ to $J$ in $\C(X_K = x_K)$. 
We then write $I\perp_{(C^\ast,x_K)} J\cd K$ or $I\starse J\cd K$ in $\C(X_k=x_k)$.
\end{definition}

\begin{example}[Cassiopeia]\label{ex:cassiopeia1} 
Example \ref{ex:cassiopeia0} illustrates that $\D^*$-separation is strictly weaker than $d$-separation. Here $\dag=\dag^*_K=\dag^*_K(C)$  for any $C$ with support $\dag$. The path between $1$ and $3$ is $d$-connecting, but it is not $\ast$-connecting. \halmos
\end{example}


We emphasize that our separation criteria follow the form of the moralization procedure in \cite{Lauritzen1990}, which is not stated in a directly path-based form. Rather, we check for separation by constructing derived graphs --- $\D^\ast_K$, $\D^\ast_K(C)$, and $\C(X_k=x_k)$ --- and then use a single common separation criteria for all of these. This formulation shall simplify some of the proofs.
As a consequence of Lemma~\ref{lem:inclusion} we get:\\

\begin{corollary}\label{cor:ciprelations} For $I,J,K$ disjoint subsets of $V$ and any possible context $\{ X_K = x_K\}$, it holds that 
\[I \dse J \cd K \implies I \tildese J\cd K \implies I\critsep J\cd K\implies I\perp_{(C^\ast,x_K)} J\cd K,\]
where $\dse$ denotes $d$-separation.
\end{corollary}
We note that these implications are strict, as illustrated in the next example and other examples further below. 
\begin{example}[Diamond]\label{ex:diamond} Consider the DAG in Figure~\ref{fig:diamondsep} in a situation where
the path $1\to 2\to 4$ is critical: $c_{42}c_{21} \geq c_{43}c_{31} $. 
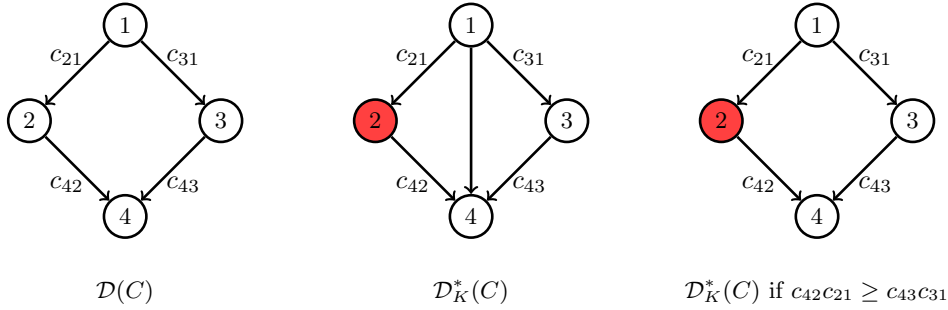
\begin{figure}
\begin{center}
\begin{tikzpicture}
\begin{scope}[->,every node/.style={circle,draw},line width=1pt, node distance=1.8cm]
\node (1) {$1$};
\node (2) [below left of=1] {$2$};
\foreach \from/\to in {1/2}
\draw (\from) -- (\to);
\path[every node/.style={font=\sffamily\small}]
(1) -- (2) node [near start, left] {$c_{21}$};
\node (3) [below right of=1] {$3$};
\foreach \from/\to in {1/3}
\draw (\from) -- (\to);
\path[every node/.style={font=\small}]
(1) -- (3) node [near start, right] {$c_{31}$};
\node (4) [below right of=2] {$4$};
\foreach \from/\to in {2/4,3/4}
\draw (\from) -- (\to);
\path[every node/.style={font=\small}]
(2) -- (4) node [near end, left] {$c_{42}$};
\path[every node/.style={font=\small}]
(3) -- (4) node [near end, right] {$c_{43}$};

\node (11)[right = 4cm of 1] {$1$};
\node (21) [below left of=11,fill=red!75] {$2$};
\foreach \from/\to in {11/21}
\draw (\from) -- (\to);
\path[every node/.style={font=\sffamily\small}]
(11) -- (21) node [near start, left] {$c_{21}$};
\node (31) [below right of=11] {$3$};
\foreach \from/\to in {11/31}
\draw (\from) -- (\to);
\path[every node/.style={font=\small}]
(11) -- (31) node [near start, right] {$c_{31}$};
\node (41) [below right of=21] {$4$};
\foreach \from/\to in {21/41,31/41}
\draw (\from) -- (\to);
\path[every node/.style={font=\small}]
(21) -- (41) node [near end, left] {$c_{42}$};
\path[every node/.style={font=\small}]
(31) -- (41) node [near end, right] {$c_{43}$};
\draw (11) -- (41);

\node (12)[right = 4cm of 11] {$1$};
\node (22) [below left of=12,fill=red!75] {$2$};
\foreach \from/\to in {12/22}
\draw (\from) -- (\to);
\path[every node/.style={font=\sffamily\small}]
(12) -- (22) node [near start, left] {$c_{21}$};
\node (32) [below right of=12] {$3$};
\foreach \from/\to in {12/32}
\draw (\from) -- (\to);
\path[every node/.style={font=\small}]
(12) -- (32) node [near start, right] {$c_{31}$};
\node (42) [below right of=22] {$4$};
\foreach \from/\to in {22/42,32/42}
\draw (\from) -- (\to);
\path[every node/.style={font=\small}]
(22) -- (42) node [near end, left] {$c_{42}$};
\path[every node/.style={font=\small}]
(32) -- (42) node [near end, right] {$c_{43}$};
\end{scope}
\begin{scope}[every node/.style ={rectangle},node distance = 1cm]
\node[below of = 4]{$\dag(C)$};
\node [below of = 41]{$\dag_K^\ast(C)$};
\node [below of = 42]{$\dag_K^\ast(C)$ if $c_{42}c_{21} \geq c_{43}c_{31} $};
\end{scope}
\end{tikzpicture} \caption{Diamond graph with $K=\{2\}$. The conditional reachability DAG (middle figure) is equal to the reachability DAG $\D^\ast$, whereas the edge $1\to 4$ is missing in the critical DAG (right-hand figure) since the path $1\to2\to 4$ is critical and factors through $K$. Note that the path $1\to3\to 4$ in $\dag^*_K(C)$ is not $*$-connecting as it is not one of the configurations in Figure~\ref{fig:d.ast.connecting}.} \label{fig:diamondsep}
\end{center}
\end{figure} 
It then holds that $1 \ci 4 \mid 2$ even though there is a $d$-connecting path $1 \to 3 \to 4$. By Definition~\ref{dfn:3.3}, this path is not $*$-connecting in $\dag^*_K(C)$ so $1\critsep 4\cd 2$. Note also that $\tildese$ is strictly weaker than $\critsep$, as $1\critsep 4\cd 2$ if $c_{21}c_{42} \geq c_{31} c_{43}$, but it holds that $\neg(1\tildese 4\cd 2)$ since $1\to3\to 4$ is $*$-connecting in $\dag^*_K$. \halmos
\end{example}

\subsection{The context-specific case}
We first consider the case of a specific context and a fixed coefficient matrix $C$.
To prove our main Theorem~\ref{thm:starse.xk} 
we need the following lemma.


\begin{lemma}\label{lem:not.ast.connecting}
Let $\C(X_K = x_K)$ be the source DAG of a possible context $\{X_K = x_K\}$. Then $i,j \in A$ (the active nodes of Proposition~\ref{cor:dag.partition})
are source separated
if and only if 
\begin{equation}\label{eqn:type.a.b}
    \left(\{i\} \cup \pa_{\C}(i) \right) \cap  \left(\{j\} \cup \pa_{\C}(j) \right) = \emptyset,
\end{equation}
and that there is no triple of nodes $i',j',k$ such that
\begin{equation}\label{eqn:ip.jp.k}
i' \in  \left(\{i\} \cup \pa_{\C}(i) \right), \quad j' \in \left(\{j\} \cup \pa_{\C}(j) \right), \quad k \in H \cup L \mbox{ and }i', j' \in \pa_{\C}(k).
\end{equation}
\end{lemma}

\begin{proof} The nodes 
$i$ and $j$ are $\ast$-connected if and only if there is
a path $\pi$ in $\C(X_K = x_K)$ that matches one of the five configurations in Figure~\ref{fig:d.ast.connecting}. 
One can choose a path $\pi$ of type (a) or (b) if and only if \eqref{eqn:type.a.b} does not hold. For types (c) to (e), let $j' = j, i' = i$ for case (c), $j' = j$ for case (d), and $j'$, $i'$ be as-is for case (e).
By definition of $\C(X_K = x_K)$,  it holds that $\pi \subset \C(X_K = x_K)$ if and only if \eqref{eqn:ip.jp.k} holds for this particular triple of nodes $i',j',k$. 
\end{proof}

We are now ready for the proof of the main theorem of this section. 

\begin{theorem}[Context-specific, fixed $C$]\label{thm:starse.xk}
Let $X$ be a max-linear Bayesian network over a directed acyclic graph $\D = (V,E)$ with fixed coefficient matrix $C$.  
Let $K \subseteq V$ and $\C(X_K = x_K)$ be the source DAG of the possible context $\{X_K = x_K\}$. Then for all $I,J \subseteq V$, 
$$I \starse J \cd K \mbox{ in } \C(X_K = x_K) \implies X_I \ci X_J \cd X_K =x_K $$
\end{theorem}

\begin{proof}
Suppose that $I$ and $J$ are source separated by $\{X_K=x_K\}$. By Lemma~\ref{lem:not.ast.connecting}, this implies that
$$ \left(I \cup \pa_{\C}(I) \right) \cap \left(J \cup \pa_{\C}(J) \right) = \emptyset, $$ 
and that there are no pairs $i' \in I \cup \pa_{\C}(I)$, $j' \in J \cup \pa_{\C}(J)$ that simultaneously appear in the same equation among those in \eqref{eqn:equality.1} and \eqref{eqn:equality.2}. By Corollary \ref{lem:z.independent}, this implies 
$$ \{Z_i: i \in I \cup \pa_{\C}(I)\} \ci \{Z_j: j \in J \cup \pa_{\C}(J)\} \cd X_K = x_K $$
and by the representation \eqref{eqn:minimal.representation}, this implies $X_I \ci X_J \cd  X_K = x_K$. 
\end{proof}

\begin{example}[Tent]\label{ex:tent2}

Applying Theorem \ref{thm:starse.xk} to the source DAG in Figure~\ref{fig:tent} of Example~\ref{ex:tent} yields the 
conditional independence statement 
$X_3 \ci ( X_1,X_2) \cd  X_4=X_5=2$, as also stated in the introduction, see Example~\ref{ex:cassiopeia0}.
\halmos
\end{example}

\subsection{The context-free cases}
In the previous subsection we identified sufficient conditions for conditional independence given a specific possible context $\{X_K=x_K\}$. We now exploit this result to derive conditions for independence that are valid in any context.

\begin{theorem}[Context-free, fixed $C$]\label{thm:starse.C}
Let $X$ be a max-linear Bayesian network over a directed acyclic graph $\D = (V,E)$ with fixed coefficient matrix $C$. Then for all $I,J,K \subseteq V$,
$$I\starse J\cd K \text{ in } \dag^*_K(C)\implies X_I \ci X_J \cd X_K. $$
\end{theorem}
\begin{proof}
It is enough to prove the result for  $K \neq \emptyset$.
Suppose that there are no $\ast$-connecting paths in $\D^\ast_K(C)$. For any possible context $\{X_K = x_K \}$, by Lemma~\ref{lem:inclusion}, $\D^\ast_K(C) \supseteq \C(X_K = x_K)$, therefore there is no $\ast$-connecting path in $\C(X_K = x_K)$. Thus we have $X_i \ci X_j \cd  X_K$ by Theorem \ref{thm:starse.xk}. 
\end{proof}
Finally, we can give the generic Markov condition which does not involve knowledge of the coefficient matrix $C$:

\begin{theorem}[Context-free, independent of $C$]\label{thm:starse.ind}
Let $X$ be a max-linear Bayesian network over a directed acyclic graph $\D = (V,E)$.
Then for all $I,J,K \subseteq V$,
$$I \starse J \cd K \mbox{ in } \D^\ast_K \implies  X_I \ci X_J \cd X_K \mbox{ for all $C$ with support included in $\D$.} $$
\end{theorem}
\begin{proof} 
By Lemma \ref{lem:inclusion}, $\D^\ast_K \supseteq \D^\ast_K(C)$, so if there are no $\ast$-connecting paths in $\D^\ast_K$, then there are also no $\ast$-connecting paths in $\D^\ast_K(C)$ for all $C$ supported by $\D$. Thus $X_i \ci X_j \cd  X_K$ for all such $C$ by Theorem \ref{thm:starse.C}. 
\end{proof}
\begin{remark}
We note that Theorem~\ref{thm:starse.ind} is corresponding to what is known as the \emph{global Markov property} for Bayesian networks, i.e.\ it establishes that separation in a suitable graph always implies conditional independence simultaneously for all possible values of the conditioning variables, and this statement holds for any choice of coefficient matrix $C$.\end{remark}
\section{Completeness}\label{sec:faith}

In this section, we shall investigate to what extent the separation criteria developed in  Section~\ref{sec:cond.indep}  are complete for conditional independence in max-linear Bayesian networks, i.e.\ yield all conditional independence relations that are valid. 
As before, we divide the discussion into the context-specific and context-free cases.

\subsection{The context-specific case} We first establish the converse to Theorem~\ref{thm:starse.xk} in the context-specific case. 
The next lemma is used several times in the proof. 
\begin{lemma}\label{lem:can.both.appear}
Suppose there is a $\ast$-connecting path between $i$ and $j$ in $\C(X_K = x_K)$ of types (a) or (b) in Figure~\ref{fig:d.ast.connecting}. Suppose further for type (b) that there exists some $g \in \mathfrak{G}(X_K = x_K)$ such that $j' \to i, j' \to j \in g$ and $j' \notin K^\ast(g)$. Then $X_i \notci X_j \cd  X_K = x_K$. 
\end{lemma}
\begin{proof}
 By Corollary \ref{cor:atomic}(b), type (b) implies
    $$  \P\left\{\frac{X_i}{c^\ast_{ij'}} = \frac{X_j}{c^\ast_{jj'}} \neq \mbox{ an atomic value of $X_i$ or $X_j$}\; \Bigg|\; X_K = x_K\right\} \geq \P(g \cd X_K = x_K) > 0,  $$
    and in type (a), we have the same inequality with $j = j'$. In either case, $X_i \notci X_j \cd  X_K = x_K$, as claimed. 
\end{proof} 
The main difficulty in proving Theorem~\ref{thm:starse.faith} below is that having two edges $j' \to i, j' \to j \in \C(X_K = x_K)$ does not in general (as in the above Lemma~\ref{lem:can.both.appear}) imply that there exists a compatible impact graph $g \in \mathfrak{G}(X_K = x_K)$, where both of these edges appear simultaneously. Indeed, Example~\ref{ex:umbrella} above shows that this need not be the case, whereas Corollary~\ref{lem:i.k.common.source} establishes this fact in a specific case. 
\begin{theorem}[Context-specific completeness]\label{thm:starse.faith}
Let $X$ be a max-linear Bayesian network over a directed acyclic graph $\D = (V,E)$ with fixed coefficient matrix $C$.  
Let $K \subseteq V$ and $\C(X_K = x_K)$ be the source DAG of a possible context $\{X_K = x_K\}$. For all subsets $I,J \subseteq V$ it holds that
$$X_I \ci X_J \cd X_K =x_K \implies I \starse J \cd K \mbox{ in } \C(X_K = x_K) $$
\end{theorem}

\begin{proof}
To prove this, we separately consider the five different types of $*$-connectivity in Figure~\ref{fig:d.ast.connecting} and in each of them establish that the variables are dependent, using the node partition $V\setminus U= A\cup H \cup L$ and the representation as in Theorem~\ref{thm:source.rep} and its corollaries. 

First we claim that it is sufficient to consider the case where  $I = \{i\}$ and $J = \{j\}$ are singletons with $i,j\in A$. For if $I$ and $J$ are $*$-connected, there must be
$i \in I, j \in J$ such that $i$ and $j$ are $\ast$-connected, so if $i$ and $j$ are dependent, so are $I$ and $J$.

 Suppose then that $i$ and $j$ are $\ast$-connected, i.e.\ there is a path $\pi$ of the types shown in Figure~\ref{fig:d.ast.connecting}. The proof considers the five different cases (a)--(e)  of this figure in turn and gives an appropriate event for each one to establish conditional dependence. 

\medskip

\noindent{\bf Case (a):}
This follows directly from Lemma~\ref{lem:can.both.appear}.

\noindent{\bf Case (b):}
For each $t = 1,\dots,m$, let 
$$ \tilde{I}_t = \pa_{\C}(\ell_t) \cap \pa_{\C}(i), \tilde{J}_t = \pa_{\C}(\ell_t) \cap \pa_{\C}(j). $$
There are two mutually exclusive subcases.
\paragraph{Case {(b)I}}
\emph{For each $t = 1,\dots,m$ we have
 $\tilde{I}_t \cup \tilde{J}_t \subsetneq \pa_{\C}(\ell_t). $} 
 
In particular, for each such $t$, there exists some $r_t \in \pa_{\C}(\ell_t)$ such that
  \begin{equation}\label{eqn:no.shared.ell}
r_t \notin  \pa_{\C}(i) \cup \pa_{\C}(j).
\end{equation}
Our goal is to construct an appropriate $g$ and appeal to Lemma \ref{lem:can.both.appear}. 
Apply \eqref{eqn:z.bound} to $j'$, let $\beta_{j'}$ be the constant on the right-hand side of this inequality. For a sufficiently small constant $\epsilon>0$, consider the event $\mathcal{E}$ defined by
\begin{itemize}
  \item $\beta_{j'} - \epsilon < Z_{j'} < \beta_{j'}$ ($Z_{j'}$ is only slightly smaller than the largest value possible in the context $\{X_K = x_K\}$)
  \item $Z_r < \epsilon$ for all $r \in \pa_{\C}(i) \cup \pa_{\C}(j) \setminus \{j'\}$ (any other parent of $i$ or $j$, except $j'$, has very small $Z$-value)  
  \item $Z_i,Z_j < \epsilon$ ($i$ and $j$ also have very small $Z$ values)
  \item for each $h \in H$, set $Z_{r'} < \epsilon$ for all $r' \in \pa_{\C}(h) \backslash \{j'\}$, and $Z_h = x_h$ (any node in $h$ realizes itself: its parents have small $Z$-values, and its own $Z$-value is $x_h$.)
  \item for each $\ell_t$ for $t = 1,\dots,m$, let $r_t$ satisfy \eqref{eqn:no.shared.ell}, and set it to achieve the maximum in \eqref{eqn:equality.2}. (Each block $L_t$ gets a parent whose $Z$-value is not already constrained by the previous conditions). 
\end{itemize}
In the above, the only nodes that were mentioned but did not get set to be less than $\epsilon$ are $Z_{j'}$, $Z_{r_t}$ for $t = 1,\dots,m$ and $Z_h$ for $h \in H$. By Proposition \ref{cor:dag.partition} and \eqref{eqn:no.shared.ell}, these nodes are all distinct, so the event $\E$ is well-defined. Furthermore, by Proposition \ref{cor:dag.partition} and  Corollary~\ref{lem:z.independent}, $\{Z_{r_t}, Z_h: t = 1,\dots, m, h \in H\}$ are independent, and either $Z_{j'}$ is independent of $\{Z_{r_t}, Z_h: t = 1,\dots, m, h \in H\}$, or it is independent of all but exactly one of them, say, $Z_u$ for $u \in \{r_t: t = 1,\dots, m\} \cup H$. In both cases, by Theorem \ref{thm:source.rep} and Corollary \ref{lem:z.independent}, 
$\P(\mathcal{E} \cd  X_K = x_K) > 0$. So there exists at least one $g \in \mathcal{C}(X_K = x_K)$  such that $\P(\E(g) \cap \mathcal{E}\cd X_K = x_K) > 0$. By construction of this event,  $j' \to i, j' \to j \in g$ and $j' \notin K^\ast(g)$. Hence $X_i \notci X_j \cd  X_K = x_K$ by Lemma \ref{lem:can.both.appear}.

 \paragraph{Case {(b)II}} \emph{There exists at least one $t = 1,\dots,m$ such that }
   \begin{equation}\label{eqn:shared.ell}
\tilde{I}_t \cup \tilde{J}_t = \pa_{\C}(\ell_t).
\end{equation}
Fix such a $t$.
Define
\begin{align*}
\E_1= \left\{X_i < \min_{r \in \tilde{I}_t}\frac{c^\ast_{ir}x_{\ell_t}}{c^\ast_{\ell_t r}}\right\}
\quad\mbox{and}\quad
\E_2= \left\{X_j < \min_{r \in \tilde{J}_t}\frac{c^\ast_{jr}x_{\ell_t}}{c^\ast_{\ell_t r}}\right\}.
\end{align*}
 By Proposition \ref{cor:dag.partition}(j), $\tilde{I}_t, \tilde{J}_t \neq \emptyset$, so the above events are well-defined. Let $r_0$ denote a node $r \in \tilde{J}_t$ that achieves the minimum in $\E_2$ above. Since $r_0 \to j \in \C(X_K = x_K)$, there exists some $g \in \mathfrak{G}(X_K = x_K)$ such that $r_0 \to j$, $r_0 \notin K^\ast(g)$. This implies that on $\E(g)$, 
$$Z_{r_0} < {x_{\ell_t}}/{c^\ast_{\ell_t r_0}}, \quad X_j = c^\ast_{jr_0}Z_{r_0}.$$
Together these imply that on $\E(g)$,
$$ X_j < \frac{c^\ast_{jr_0}x_{\ell_t}}{c^\ast_{\ell_t r_0}} = \min_{r \in \tilde{J}_t}\frac{c^\ast_{jr}x_{\ell_t}}{c^\ast_{\ell_t r}}. $$
So $\E(g) \subseteq \E_2$. Therefore, $\P(\E_2 \cd  X_K = x_K) > 0$ and, by symmetry,  $\P(\E_1 \cd  X_K = x_K) > 0$. 

By \eqref{eqn:equality.2} in Theorem~\ref{thm:source.rep}, for each $g \in \mathfrak{G}(X_K = x_K)$,
$\pa_g(\ell_t) \in \pa_{\C}(\ell_t)$. By \eqref{eqn:shared.ell}, either $\pa_g(\ell_t) \in \tilde{I}_t$ or $\pa_g(\ell_t) \in \tilde{J}_t$; note that both can occur simultaneously as we are \emph{not} claiming that $\tilde{I}_t \cap \tilde{J}_t = \emptyset$.
Consider all $g$ such that
$\pa_g(\ell_t) \in \tilde{J}_t$. Let $r = \pa_g(\ell_t)$. By definition of the max-linear model,
$$ X_j \geq c^\ast_{jr}Z_r = \frac{c^\ast_{jr}x_{\ell_t}}{c^\ast_{\ell_t r}} \mbox{ on } \mathcal{E}(g) \mbox{ for any } g \mbox{ s.t. } \quad r = \pa_g(\ell_t) \in \tilde{J}_t. $$
In particular, for any $g$ such that $\pa_g(\ell_t) \in \tilde{J}_t$, 
$$ \P(\E(g) \cap \E_2 | X_K = x_K) = 0. $$
By the same argument, for any $g$ such that $\pa_g(\ell_t) \in \tilde{I}_t$, 
$$ \P(\E(g) \cap \E_1 | X_K = x_K) = 0. $$
But $\pa_g(\ell_t) \in \tilde{I}_t \cup \tilde{J}_t$ for all $g \in \mathfrak{G}(X_K = x_K)$ as mentioned above. Therefore, there is no $g \in \mathfrak{G}(X_K = x_K)$ such that $\E(g) \subseteq \E_1 \cap \E_2$. That is,
$$ \P(\E_1 \cap \E_2 | X_K = x_K) = 0. $$
But $\P(\E_1  | X_K = x_K) > 0, \P(\E_2  | X_K = x_K) > 0$, so the events $\E_1$ and $\E_2$ are not independent conditioned on $\{X_K = x_K\}$. Therefore, $X_i \notci X_j \cd  X_K = x_K$.

\noindent{\bf Cases (c), (d) and (e):}
We may assume that cases (a) and (b) do not apply. In particular, \eqref{eqn:type.a.b} holds. For case (c), let
$$\E_1 = \left\{X_i = \frac{x_k}{c^\ast_{ki}}, X_K = x_K\right\}, \quad \E_2 = \left\{X_j = \frac{x_k}{c^\ast_{kj}}, X_K = x_K\right\}.$$
For case (d), let 
$$\E_1 = \left\{X_i = \frac{x_k}{c^\ast_{ki}}, X_K = x_K\right\}, \quad \E_2= \left\{X_j =\frac{c^\ast_{jj'}x_k}{c^\ast_{kj'}}, X_K = x_K\right\}. $$
For case (e), let 
$$ \E_1 = \left\{X_i = \frac{c^\ast_{ii'}x_k}{c^\ast_{ki'}}, X_K = x_K\right\}, \quad \E_2 = \left\{X_j = \frac{c^\ast_{jj'}x_k}{c^\ast_{kj'}}, X_K = x_K\right\}.$$
We now claim that in all three cases we have
$$ \P(\E_1 \cd X_K = x_K) > 0 \mbox{ and } \P(\E_2 \cd X_K = x_K) > 0. $$
Indeed, these follow in case (c) from $i \to k, j \to k \in \C(X_K = x_K)$, and in cases (d) and (e) from Corollary~\ref{lem:i.k.common.source} applied to the triples $k \leftarrow j' \to j$ and $k \leftarrow i' \to i$. 
By \eqref{eqn:extra2} in Lemma~\ref{lem:y.xk}, any $g \in \mathfrak{G}(\E_1)$ must have $R_g(k) = R_g(i)$. Similarly, any $g \in \mathfrak{G}(\E_2)$ must have $R_g(k) = R_g(j)$. But \eqref{eqn:type.a.b} implies $R_g(i) \neq R_g(j)$ for all $g \in \mathfrak{G}(X_K = x_K)$. Therefore, $$ \P(\E_1 \cd \E_2, X_K = x_K) = \P(\E_2 \cd  \E_1, X_K = x_K) = 0. $$
So $X_i \notci X_j \cd X_K = x_K$ in each of the three cases, as needed.
 %
Since all cases have been considered, this concludes the proof.
\end{proof}
\subsection{The context-free cases}
Next we consider the context-free case for a given coefficient matrix $C$.  We begin by showing that the direct converse to Theorem~\ref{thm:starse.C} is false, as demonstrated in the following example.

\begin{example}\label{ex:counter.1.3}
Consider the graph in Figure~\ref{fig:counterexample} with all edge weights equal to one.
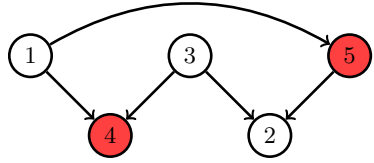
\begin{figure}
\begin{center}
\begin{tikzpicture}[->,every node/.style={circle,draw},line width=1pt, node distance=1.5cm]
  \node (1)  {$1$};
  \node (4) [below right of=1,fill=red!75]{$4$};
  \node (3) [above right of=4] {$3$};
  \node (2) [below right of=3] {$2$};
  \node (5) [above right of=2,fill=red!75] {$5$}; 
\foreach \from/\to in {1/4,3/4,3/2,5/2}
\draw (\from) -- (\to);   
   \path (1) edge [bend left=30, align=left, above](5);
 \end{tikzpicture}  
\end{center}
\caption{The counterexample with $\D=\D_K^*$ and observed nodes $K=\{4,5\}$. 
 Here it holds that $X_1\ci X_2\cd  X_{\{4,5\}}$ even though 1 and 2 are $\ast$-connected relative to $K$ with the path $1\to 4\leftarrow 3\to 2$.
 }\label{fig:counterexample}
\end{figure}
 We have
 \begin{align*}
     X_5 &= Z_5 \vee Z_1 \\
     X_4 &= Z_4 \vee Z_1 \vee Z_3 \\
     X_2 &= Z_2 \vee Z_3 \vee Z_5
 \end{align*}
The important feature of this example is that $c^\ast_{21}=c^\ast_{25} c^\ast_{51}$, i.e.\ there is a critical directed path from 1 to 2 that factors through $K$, so $1 \to 2 \notin \D^\ast_K(C)$ and $1 \to 2 \notin \D^\ast_K$. On the other hand, $\pi = 1 \to 4 \leftarrow 3 \to 2$ is a $\ast$-connecting path. Nevertheless, we claim below that $X_1\ci X_2\cd  X_{4,5}$.
 
 Indeed, if $x_5 \geq x_4$, then also $x_5 \geq Z_3$ so $\C(x_4,x_5)$ is a subgraph of the graph to the left in Figure~\ref{fig:subgraphs}
 \begin{figure}
\begin{center}
\begin{tikzpicture}[->,every node/.style={circle,draw},line width=1pt, node distance=1.5cm]
  \node (1)  {$1$};
  \node (4) [below right of=1,fill=red!75]{$4$};
  \node (3) [above right of=4] {$3$};
  \node (2) [below right of=3] {$2$};
  \node (5) [above right of=2,fill=red!75] {$5$}; 
\foreach \from/\to in {1/4,3/4,5/2}
\draw (\from) -- (\to);   

\node (a1) [right= 2cm of 5] {$1$};
  \node (a4) [below right of=a1,fill=red!75]{$4$};
  \node (a3) [above right of=a4] {$3$};
  \node (a2) [below right of=a3] {$2$};
  \node (a5) [above right of=a2,fill=red!75] {$5$}; 
\foreach \from/\to in {a3/a4,a3/a2,a5/a2}
\draw (\from) -- (\to);   
   \path (1) edge [bend left=30, align=left, above](5);
 \end{tikzpicture}  
\end{center}
\caption{The source DAG $\mathcal{C}(x_4,x_5)$ in a context satisfying $\{x_5\geq x_4\}$ is a subgraph of the graph to the left and of the graph to the right if $\{x_5<x_4\}$.}\label{fig:subgraphs}
\end{figure}
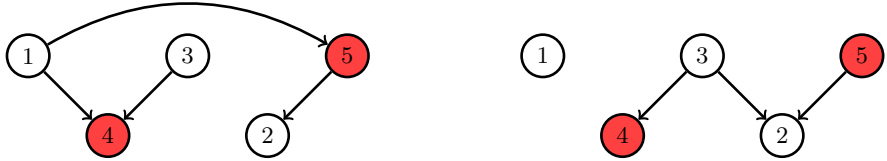
On the other hand, if $x_5 < x_4$, then $1 \to 4 \notin  \C(x_4,x_5)$ so that $\C(x_4,x_5)$ is a subgraph of the graph to the right in Figure~\ref{fig:subgraphs}.

In both cases there is no $*$-connecting path between 1 and 2, hence by Theorem \ref{thm:starse.xk} we have 
 $X_1\ci X_2\cd  X_{\{4,5\}}$. \halmos
\end{example}
We note that the phenomenon here has some similarity to `path cancellations' in standard linear Bayesian networks, where specific values of the coefficients may allow dependence relations to cancel and yield conditional independence which does not follow from the separation criterion. Below we shall discuss and resolve this type of problem. Here the set of coefficients corresponding to such a phenomenon may have positive Lebesgue measure, in contrast to the standard case, which makes this discussion necessary. The key concept is that of an effective edge or path, as further described below. 
\subsubsection{Effective edges and paths}
To obtain converses for the context-free cases, we wish to construct a  possible context $\{X_K=x_K\}$ that violates the context-specific Markov condition. However, Example~\ref{ex:counter.1.3} above shows that this is not always possible.  
We need to ensure that no inequalities along $*$-connecting paths imply further equalities and to control this we need the following concept.

\begin{definition}
\label{defn:bmatrix}
Let $X$ be a max-linear Bayesian network over a directed acyclic graph $\D = (V,E)$ with fixed coefficient matrix $C$ and $K\subset V$. For an edge $j \to i \in \D^\ast_K(C)$, the \emph{substitution matrix}  $\Xi_K^{ij}$ of this edge relative to $K$ is a $|K| \times |K|$ matrix with the following 
	non-zero entries:
	\begin{align}
	(\xi^{ij}_K)_{k\ell } &= \frac{c^\ast_{kj}c^\ast_{i\ell}}{c^\ast_{ij}} \mbox{ for } k \in K \cap \ch_{\D^\ast}(j), \; \ell \in K \cap (\pa_{\D^\ast}(i)\cup \{i\}), k\neq \ell. \label{eqn:bij.kl}
	\end{align}
If $\pi$ is a $\ast$-connecting path between $i$ and $j$, then its \emph{substitution matrix} $\Xi_K^\pi$ relative to $K$ is defined as
$$ \Xi_K^\pi = \bigvee_{v \to u \in \pi} \Xi_K^{uv}. $$
\end{definition}

\begin{remark}
Example~\ref{ex:counter.1.3} above features a path  $\pi = 1 \to 4 \leftarrow 3 \to 2$ in $\D^\ast_K(C)$, but there is no $x_K$ such that $\pi \subset \C(X_K = x_K)$.  More importantly, as we show in Proposition~\ref{prop:lambda.less.than.1} below, existence of such an $x_K$ is equivalent to the additional condition \eqref{eqn:lambda.bpi} ensuring that the path is effective, as defined below. 
\end{remark}
\begin{definition}\label{def:eff.path} A $*$-connecting path
$\pi$ from  $I$ to $J$ in $\dag^*_K(C)$ is said to be \emph{effective}  if it  satisfies the tropical eigenvalue condition
\begin{equation}\label{eqn:lambda.bpi}
	\lambda(\Gamma_{KK} \vee \Xi_K^\pi) < 1,
	\end{equation}
	where  $\Xi_K^\pi$ is the substitution matrix of $\pi$ with respect to $K$ and
	$\Gamma_{KK}$ is the restriction of the weak transitive closure $\Gamma(C)$ as in \eqref{eq:transclos} to the components in $K$.
\end{definition}

\begin{example}The
condition \eqref{eqn:lambda.bpi} is necessary in general. 
In Example~\ref{ex:counter.1.3} we have a single $\ast$-connecting path $\pi$ in $\D^\ast_K(C)$ between $1$ and $2$ and for this path \eqref{eqn:lambda.bpi} fails, as we shall now show.
The substitution matrix for the path $\pi=1\to 4\leftarrow 3\to 2$ is 
$$\Xi_K^\pi= \Xi_K^{41}\vee \Xi_K^{43}\vee \Xi_K^{23}.$$
We find positive entries 
$$b_{54}^{41}= \frac{c^*_{51}}{c^*_{41}} = 1 \quad \text{ and } \quad b_{45}^{23}= \frac{c^*_{43}c^*_{25}}{c^*_{23}} = 1,$$ so 
$$\Gamma_{KK}\vee \Xi_K^\pi =
\begin{pmatrix}
0 & 1\\
0  & 0 
\end{pmatrix} \vee \begin{pmatrix}
0 & 1\\
1  & 0 
\end{pmatrix} = 
\begin{pmatrix}
0 & 1\\
1  & 0 
\end{pmatrix}
$$ 
and hence we get $$	\lambda(\Gamma_{KK} \vee \Xi_K^\pi) = 1, $$ which violates \eqref{eqn:lambda.bpi}.  
Here, as noticed in Example~\ref{ex:counter.1.3}, $X_i \ci X_j \cd X_K$ despite the existence of a $*$-connecting path. 
 \halmos
\end{example}

It turns out that condition \eqref{eqn:lambda.bpi} is often automatically satisfied.  As we shall study effective edges in a specific context, we need the following concept.

\begin{definition}
The \emph{completion} of the coefficient matrix $C$ with respect to a possible context $\{X_K = x_K\}$ is the $|V| \times |V|$ coefficient matrix $\bar{C}$, with
$$ \bar{c}_{ij} = 
\left\{\begin{array}{cc}
x_i/x_j & \mbox{ if } i,j \in K^\ast, \\
c_{ij} & \mbox{ else}.
\end{array} \right. $$
\end{definition}

We write $\bar{C}^\ast=(\bar c^*_{kh})$ for the Kleene star of $\bar{C}$ and
note that all cycles in  $\D(\bar C)$ that only involve nodes in $K^\ast$ have weight equal to one:
\[\bar c_{i_1i_2}\bar c_{i_2i_3}\cdots \bar c_{i_ki_1}=
\frac{x_{i_1}}{x_{i_2}}\frac{x_{i_2}}{x_{i_3}}\cdots \frac{x_{i_k}}{x_{i_1}}=1.\]
\begin{lemma}
Let $\bar{C}$ be the completion of $C$ with respect to a possible context $\{X_K = x_K\}$. Then $\lambda(\bar{C}) = 1$. 
\end{lemma}
\begin{proof}
If $|K|=1$ this is obviously true for a self-loop. Assume that  $|K|\geq 2$. Since $\D(C)$ is acyclic and all cycles in $\D(\bar{C})$ involving only nodes in $K^\ast$ have length 1, it is sufficient to consider simple cycles $\pi = 1 \to 2 \dots \to r \to 1$, with $1,r \in K^\ast$ and other nodes not in $K^\ast$. Write $\bar{c}(\pi)$ for the product of the edge weights of this cycle in $\bar{C}$. We claim that $\bar{c}(\pi) \leq 1$. Indeed,
$$ \bar{c}(\pi) \leq c^\ast_{r2}c^\ast_{21}\bar{c}_{1r} = c^\ast_{r2}c^\ast_{21}\frac{x_1}{x_r}\leq \frac{c^\ast_{r1} x_1}{x_r}\leq 1$$
where we have used that
$c^\ast_{r1}\geq c^\ast_{21}c^\ast_{r2}$ and the context $\{X_K=x_K\}$ is possible, so $x_r\geq c^*_{r1}x_1$. 
Hence the maximum cycle mean is  $\lambda(\bar C)=1$, as desired. 
\end{proof}

\begin{corollary}\label{cor:direct} 
For $k,h\in K^\ast$ we have that $\bar{c}^\ast_{kh} = \bar{c}_{kh} = {x_k}/{x_h}$. 
\end{corollary}
\begin{proof}
If $h=k$ this is obviously true. Now assume $h\neq k$.
By definition of the Kleene star, $\bar{c}^\ast_{kh} \geq \bar{c}_{kh}$ and 
since $\lambda(\bar{C}) = 1$, we also have $\lambda(\bar{C}^\ast) = 1$. 
Since $k \to h$ and $h \to k$ are edges in $\bar{C}$, we may consider the cycle $k\to h\to k$ and get
$$1=\lambda(\bar{C}^\ast) \geq \bar{c}^\ast_{kh}\bar{c}^\ast_{hk} \geq \bar{c}_{kh}\bar{c}_{hk} = \frac{x_k}{x_h}\frac{x_h}{x_k} = 1.$$
Thus we must have equalities; that is $\bar{c}^\ast_{kh} = \bar{c}_{kh}$ and $\bar{c}^\ast_{hk} = \bar{c}_{hk}$.
\end{proof}

\begin{definition}
Say that an edge $j \to i$ in $\D^\ast_K(C)$ is \emph{effective} in the possible  context $\{X_K = x_K\}$ if $j \notin K^\ast$, 
	no critical directed paths from $j$ to $i$ 
	factor through $K^\ast$, and $c^\ast_{ij} = \bar{c}^\ast_{ij}$.
Let $E^+(X_K = x_K)$ denote the set of effective edges in the context $\{X_K = x_K\}$. Edges in $\D^\ast_K(C)$ which are not effective are \emph{ineffective}. Finally, a path $\pi$ is \emph{effective} in a context if all its edges are. 
\end{definition}
Now we give an algebraic characterization of edges that are effective in a context. 

\begin{lemma}\label{lem:e+.alg}
Let $j \to i \in \D^\ast_K(C)$ and consider a possible context $\{X_K=x_K\}$. Then $j \to i \in E^+(X_K = x_K)$ if and only if for all $k \in K^\ast \cap \ch_{\D^\ast}(j)$, $\ell \in K^\ast \cap (\pa_{\D^\ast}(i)\cup \{i\})$, it holds that 
	\begin{equation}\label{eqn:bij.kl.x}
	(\xi^{ij}_{K^\ast})_{k\ell}x_\ell < x_k
	\end{equation}
	for $\Xi_{K^\ast}^{ij}$ being the substition matrix relative to $K^\ast$ as defined in \eqref{eqn:bij.kl}. 
\end{lemma}

\begin{proof}
Suppose $j \to i \in E^+(X_K = x_K)$. For each  $k \in K^\ast \cap \ch_{\D^\ast}(j)$ and $\ell \in K^\ast \cap (\pa_{\D^\ast}(i)\cup \{i\})$, the path $j \to k \to \ell \to i$ (or $j \to k \to  i$ if $i=\ell$) has $\bar{C}$-weight
$$ c^\ast_{i\ell}\frac{x_\ell}{x_k}c^\ast_{kj}. $$
Since this path factors through $K^\ast$, it is not critical, so 
$$ c^\ast_{i\ell}\frac{x_\ell}{x_k}c^\ast_{kj} < c^\ast_{ij}. $$
Rearranging gives \eqref{eqn:bij.kl.x}. Conversely, suppose that \eqref{eqn:bij.kl.x} holds. Let $\pi$ be a path from $j$ to $i$ in $\D(\bar{C})$ that factors through $K^\ast$. If it only goes through one node of $K^\ast$, then it is also a path in $C$ that factors through $K^\ast$, so $\bar{c}^\ast(\pi) = c^\ast(\pi)$. Since $j \to i \in \D^\ast_K(C)$, by definition of $\D^\ast_K(C)$, we have
$$ \bar{c}^\ast(\pi) = c^\ast(\pi) < c^\ast_{ij}. $$
If $\pi$ goes through two or more nodes of $K^\ast$, then without loss of generality we can assume
$$\pi = j \to \dots \to k_1 \to \dots \to k_2 \to \dots \to k_r \to \dots \to i,$$ 
where $r \geq 2$, $k_1,\dots,k_r \in K^\ast$, and $\to \dots \to$ are sequences of critical edges that do not go through $K^\ast$. By this criticality assumption, we get the equality
$$ \bar{c}^\ast(\pi) = c^\ast_{k_1j}\bar{c}^\ast_{k_2k_1}\dots\bar{c}^\ast_{k_rk_{r-1}}c^\ast_{ik_r}. $$
By Corollary \ref{cor:direct}, 
$$ \bar{c}^\ast_{k_2k_1}\dots\bar{c}^\ast_{k_rk_{r-1}} = \bar{c}^\ast_{k_rk_1} = \frac{x_{k_1}}{x_{k_r}}. $$
Note that $k_1 \in \ch_{\D^\ast}(j)$ and $k_r \in \pa_{\D^\ast}(i)$. Apply \eqref{eqn:bij.kl.x} with $k_1 = k$ and $k_r = \ell$ yields
$$ \frac{c^\ast_{kj}c^\ast_{i\ell}}{c^\ast_{ij}}x_\ell < x_k.$$
Rearranging, we get
$$ \bar{c}^\ast(\pi) = c^\ast_{kj}\frac{x_\ell}{x_r}c^\ast_{i\ell} < c^\ast_{ij}. $$
This shows that any critical path $\pi$ that factors through $K^\ast$ has weight strictly less than $c^\ast_{ij}$, as desired. 
\end{proof}

A simple corollary is the following, showing that all edges in the source DAG for a given context are indeed effective in that context.

\begin{corollary}\label{cor:e.subset}
If $j \to i$ is an edge in $ \C(X_K = x_K)$ then $j \to i \in E^+(X_K = x_K)$. 
\end{corollary}

\begin{proof}
Assume that $j \to i \in \C(X_K = x_K)$ so we have $j \notin K^\ast$ and $c^\ast_{ij} > 0$. 
First we claim that $j \to i \in \D^\ast_K(C)$. Indeed, suppose for contradiction that a critical path from $j$ to $i$ in $\D$ factors through some node $k \in K$, then 
$$ c^\ast_{ij}Z_j = c^\ast_{ik}c^\ast_{kj}Z_j \leq c^\ast_{ik}x_k. $$
Since 
$$ X_i = c^\ast_{ik}x_k \vee c^\ast_{ij}Z_j \vee \dots, $$
this implies that $j \to i \notin \C(X_K = x_K)$, a contradiction as needed.  Hence $j \to i \in \D^\ast_K(C)$.

Next suppose for contradiction that $j \to i \notin E^+(X_K = x_K)$. 
By Lemma \ref{lem:e+.alg}, this implies for some $k \in K^\ast \cap \ch_{\D^\ast}(j)$ and $\ell \in K^\ast \cap (\pa_{\D^\ast}(i)\cup \{i\})$, 
$$ (\xi^{ij}_{K^\ast})_{k\ell}x_\ell \geq x_k. $$
We apply the definition of the substitution matrix $\Xi_{K^\ast}^{ij}$ in \eqref{eqn:bij.kl} and rearrange; then we get
$$ c^\ast_{i\ell}x_\ell \geq \frac{c^\ast_{ij}}{c^\ast_{kj}}x_k. $$
Since $j \to k \in \D^\ast_K(C)$, it holds that $x_k \geq c^\ast_{kj}Z_j$ so
$$ \frac{c^\ast_{ij}}{c^\ast_{kj}}x_k \geq c^\ast_{ij}Z_j. $$
Since
$$ X_i = c^\ast_{i\ell}x_\ell \vee c^\ast_{ij}Z_j \vee \dots, $$
it follows that $j \to i \notin \C(X_K = x_K)$, a contradiction as needed. The proof is complete.
\end{proof}

\begin{example}\label{ex:sll.newthm}
 Consider the graph to the left in Figure \ref{fig:counter.3.13}.
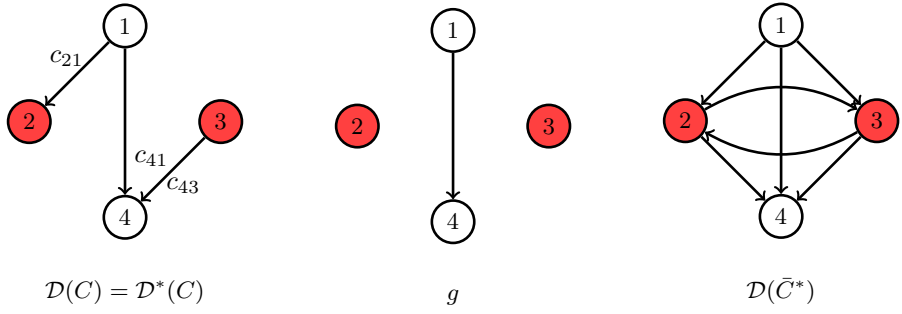
\begin{figure}
\begin{center}
\begin{tikzpicture}
\begin{scope}[->,every node/.style={circle,draw},line width=1pt, node distance=1.8cm]
\node (1) {$1$};
\node (2) [below left of=1, fill=red!75] {$2$};
\foreach \from/\to in {1/2}
\draw (\from) -- (\to);
\path[every node/.style={font=\sffamily\small}]
(1) -- (2) node [near start, left] {$c_{21}$};
\node (3) [below right of=1, fill = red!75] {$3$}; 
\node (4) [below right of=2] {$4$};

\draw (3) -- (4);
\path[every node/.style={font=\small}]
(3) -- (4) node [near end, right] {$c_{43}$};
\draw (1) -- (4);
\path[every node/.style={font=\small}]
(1) -- (4) node [near end, right] {$c_{41}$};
\end{scope}
\begin{scope} \node (5) [below of = 4] {$\D(C)=\D^*(C)$};
\end{scope}
\end{tikzpicture}
\hspace{1cm}
\begin{tikzpicture}
\begin{scope}[->,every node/.style={circle,draw},line width=1pt, node distance=1.8cm]
\node (1) {$1$};
\node (2) [below left of=1, fill=red!75] {$2$};
\node (3) [below right of=1, fill = red!75] {$3$}; 
\node (4) [below right of=2] {$4$};
\draw (1) -- (4);
\end{scope}
\begin{scope} \node (5) [below of = 4] {$g$};
\end{scope}
\end{tikzpicture}
\hspace{1cm}
\begin{tikzpicture}
\begin{scope}[->,every node/.style={circle,draw},line width=1pt, node distance=1.8cm]
\node (1) {$1$};
\node (2) [below left of=1, fill=red!75] {$2$};
\node (3) [below right of=1, fill = red!75] {$3$}; 
\node (4) [below right of=2] {$4$};
\foreach \from/\to in {1/2,1/3,1/4,2/4,3/4}
\draw (\from) -- (\to);
   \path (2) edge [bend left=30, align=left, above](3);
    \path (3) edge [bend left=30, align=left, above](2);
\end{scope}
\begin{scope} \node (5) [below of = 4] {$\D(\bar C^\ast)$}; 
\end{scope}
\end{tikzpicture}
\end{center}
\caption{The graph to the left  displays $\D(C) = \D(C^\ast)$ with coefficients. The impact graph $g$ in the middle is not compatible with  $\{X_K=x_K\}$ if ${c_{43}c_{21}x_3} \geq c_{41}x_2$, as shown in Example~\ref{ex:sll.newthm}, thus rendering the edge $1\to 4$ ineffective. The graph to the right is the completion $\D(\bar{C}^\ast)$.}\label{fig:counter.3.13}
\end{figure}
 Here $C = C^\ast$. In this case  we have $X_1 \notci X_4 \cd X_2, X_3$ although this is not true in all contexts.  We first show that the graph $g$ in the middle is not compatible with $\{X_K=x_K\}$ if ${c_{43}c_{21}x_3}\geq c_{41}x_2$. 
We thus write out the max-linear model:
\begin{align}
X_1 &= Z_1 \nonumber \\ 
x_3 &= Z_3 \nonumber \\ 
x_2 &= c_{21}Z_1 \vee Z_2 \label{eqn:x2} \\ 
X_4 &= c_{43}x_3 \vee c_{41}Z_1 \vee Z_4. \label{eqn:x4}
\end{align}
From \eqref{eqn:x2}, we have that $Z_1 \leq {x_2}/{c_{21}}$, so $c_{41}Z_1 \leq {c_{41}x_2}/{c_{21}}$. Thus, if ${c_{41}x_2}/{c_{21}} \leq c_{43}x_3$, or equivalently, $ c_{41}x_2 \leq {c_{43}c_{21}x_3} $, we also have $c_{41}Z_1<c_{43}x_3$ and hence $(x_2,x_3)$ is not in the image of $L_g$ so $g$ is not compatible with the context.

Further, the support of $\bar{C}^\ast$ is shown to the right of Figure \ref{fig:counter.3.13}. With the addition of the edges $\bar{c}_{23} = x_2/x_3$ and $\bar{c}_{32} = x_3/x_2$, we have
$$ \bar{c}^\ast_{41} = c_{41} \vee c_{43}\bar{c}_{32}c_{21} = c_{41} \vee c_{43}\frac{x_3}{x_2}c_{21}. $$
In particular, $1 \to 4$ is not effective w.r.t.\  $\{X_K = x_K\}$ if $c_{41} < c_{43}({x_3}/x_2)c_{21}$. \halmos
\end{example}

\begin{remark}
By definition of $\bar{C}$ and the critical graph $\D^\ast_{K^\ast}(C)$, if $j \to i \in E^+(X_K = x_K)$, then it holds that $j \to i \in \D^\ast_{K^\ast}(C) \subseteq \D^\ast_{K}(C)$. But the converse fails. That is, $E^+(X_K = x_K)$ can be a strictly smaller set of edges than those in $\D^\ast_K(C)$ or $\D^\ast_{K^\ast}(C)$.
\end{remark}

The following says that if a path is effective in a context, it is effective in the sense of Definition~\ref{def:eff.path}. Note that, crucially, Definition~\ref{def:eff.path} refers to the original set of conditioned variables $K$, while being effective in a given context $\{X_K = x_K\}$ is a property that involves the potentially bigger set $K^\ast$ of variables which are  constant in this context.
\begin{proposition}\label{prop:lambda.less.than.1}
Let $\pi$ be a $\ast$-connecting path in $\D^\ast_K(C)$. If  $\pi$ is effective in a possible context $\{X_K=x_K\}$, then $\lambda(\Gamma_{KK} \vee \Xi_K^\pi) < 1$.
\end{proposition}

\begin{proof}
For each edge $j \to i \in \pi$, let $\Xi_K^{ij}$ be the substitution matrix of this edge with respect to $K$ (cf.\ Definition \ref{defn:bmatrix}). Since $K \subseteq K^\ast(X_K = x_K)$, by Lemma~\ref{lem:e+.alg}, 
$$ \Xi_K^{ij}\odot x_K < x_K. $$
Thus
$$ \Big(\bigvee_{j \to i \in \pi} \Xi_K^{ij}\Big)\odot x_K = \Xi_K^\pi\odot  x_K \leq x_K. $$
Since $x_K$ satisfies the max-linear model, we have
$$ x_K = (C^\ast\odot x)_K \geq \Gamma_{KK}\odot x_K. $$
So 
$$ (\Gamma_{KK} \vee \Xi_K^\pi)\odot  x_K=\Gamma_{KK}\odot x_K\vee \Xi_K^\pi\odot  x_K   \leq x_K. $$
By Proposition \ref{prop:lambda.1}(a), this implies $\lambda(\Gamma_{KK} \vee \Xi^\pi) \leq 1$. Now we want to argue that this eigenvalue must be strictly less than 1.
By Proposition \ref{prop:lambda.1}(b), it is sufficient to show that there does not exist a cycle in $\D(\Gamma_{KK} \vee \Xi_K^\pi)$ with weight 1. 

Suppose then for contradiction that there exists  a cycle $\sigma$ with weight $w(\sigma)=1$,  let
$S$ be its support, and let  $A_{SS} = (\Gamma_{KK} \vee \Xi_K^\pi)_{SS} $.
Since $\D(\Gamma_{KK})$ is a DAG and $\Xi_K^\pi$ has zero diagonal, we must have $|S| \geq 2$. Again by Proposition \ref{prop:lambda.1}(b) we have 
$$ A_{SS}\odot x_S = x_S. $$
Now consider an edge $v \to u \in \sigma$;
by definition of $A_{SS}$ we have
$$a_{uv} = c^\ast_{uv} \vee \bigvee_{j \to i \in \pi, i \notin K} (\xi^{ij}_K)_{uv} \vee \bigvee_{j \to i \in \pi, i \in K} (\xi^{ij}_K)_{uv}.$$
 By Lemma \ref{lem:critical.edges} we further have
$ a_{uv}x_v = x_u$ and 
by \eqref{eqn:bij.kl}
$$ \bigvee_{j \to i \in \pi, i \notin K} (\xi^{ij}_K)_{uv}x_v < x_u. $$
Thus
$$a_{uv} = c^\ast_{uv} \vee \bigvee_{j \to i \in \pi, i \in K} (\xi^{ij}_K)_{uv}$$
for all edges $v \to u \in \sigma$. By definition, for $i \in K$, $(\xi^{ij}_K)_{uv} > 0$ if and only if $v = i$. In other words, for each edge $v \to u \in \sigma$ such that $a_{uv} > c^\ast_{uv}$ it holds that $v \in K \cap \pi$. Since $\pi$ is a $\ast$-connecting path, $|K \cap \pi| \leq 1$, there is at most one edge $v \to u$ of $\sigma$ where $a_{uv} = (\xi^{vj}_K)_{uv} > c^\ast_{uv}$, while for all other edges $v' \to u'$ of $\sigma$, $a_{u'v'} = c^\ast_{u'v'}$. 
Since $\D(C)$ is a DAG, there must be exactly one such edge. Therefore, 
$$ w(\sigma) = (\xi^{vj}_K)_{uv}c^\ast_{vu_1}c^\ast_{u_1u_2}\dots c^\ast_{u_ru} = (\xi^{vj}_K)_{uv}c^\ast_{vu} = \frac{c^\ast_{uj}c^\ast_{vu}}{c^\ast_{vj}}, $$
with $v,u \in K$, $c^\ast_{vu} >  0$ and $j \to v, j \to u \in \D^\ast_K(C)$. But $j \to u \to v$ is a path from $j$ to $v$ that factors through $K$. Since $j \to v \in \D^\ast_K(C)$, we have
$$ c^\ast_{vu}c^\ast_{uj} < c^\ast_{vj}. $$
Rearranging gives $w(\sigma) < 1$, which is our desired contradiction. 
\end{proof}

\subsubsection{Context-free completeness}
To establish context-free completeness, we must understand the geometry of the set $\mathcal{L}^C_K$  defined in \eqref{eqn:feasible}.
For an edge $j \to i \in \D^\ast_K(C)$ we let $x_K(j \to i)$ be the set of $x_K$ such that $j \to i$ is an effective edge in the possible context $\{X_K = x_K\}$. That is,  we abbreviate
$$x_K(j \to i) = \{x_K: j \to i \in E^+(X_K = x_K)\}$$
and for a path $\pi$ we similarly write
$$x_K(\pi) = \bigcap_{v \to u \in \pi} x_K(v \to u).$$
Now consider a  $\ast$-connecting path $\pi$ in $\D^\ast_K(C)$ and let
$$ \Sigma(\pi) = \{x_K\in \mathcal{L}^C_K: (\Gamma_{KK} \vee \Xi_K^\pi)\odot x_K < x_K\}.  $$
\begin{lemma}\label{lem:dim.spi}
Let $\pi$ be a $\ast$-connecting path in $\D^\ast_K(C)$. Suppose there exists $x_K$ such that all edges of $\pi$ are effective in the possible context $\{X_K=x_K\}$. Then $\Sigma(\pi)$ is a non-empty full-dimensional subset of $\Rplus^K$. 
\end{lemma}

\begin{proof}
By Proposition \ref{prop:lambda.less.than.1}, $\lambda(\Gamma_{KK} \vee \Xi_K^\pi) < 1$ so $\Sigma(\pi) \neq \emptyset$ by Proposition \ref{prop:lambda.1}(c). Now, the smooth and invertible map $x \mapsto \log(x)$ has a regular total derivative and  maps  $\Sigma(\pi)$  to the relative interior of a classical polyhedron $Q$ defined by strict inequalities of the form
$$ y \in Q \iff \mbox{ for all } u,v \in V: y_v - y_u > 
\log((\Gamma_{KK} \vee \Xi_K^\pi)_{uv}), 
$$
where we have let $\log(0)=-\infty$.
Thus, $Q$ is an intersection of finitely many open half-spaces. Since $\Sigma(\pi) \neq \emptyset$ we have $Q \neq \emptyset$ and $Q$ is open and full-dimensional. So $\Sigma(\pi)$ is open and full-dimensional. 
\end{proof}

\begin{proposition}\label{lem:L.emptyset}
Consider a $\ast$-connecting path $\pi$ in $\D^\ast_K(C)$ with $\Sigma(\pi) \neq \emptyset$. Then there exists some $x_K \in \Sigma(\pi)$ such that
in the possible context $\{X_K = x_K\}$ with corresponding node partition $V = A \cup H \cup L \cup U$, we have
$L = \emptyset$, $K^\ast = K$, 
and	all edges of $\pi$ are effective with respect to $\{X_K = x_K\}$ . 
\end{proposition}

\begin{proof} 
For each $v \in V$ and each pair $h,k \in K$ such that $c^\ast_{hv}, c^\ast_{kv} > 0$, let
$$ \mathcal{L}_{hkv} = \left\{x_K: \frac{x_h}{c^\ast_{hi}} = \frac{x_k}{c^\ast_{ki}} \mbox{ for some } i \in V\right\}\quad
\mbox{and}\quad \mathcal{L} = \bigcup_{h,k,v: \mathcal{L}_{hkv} \neq \emptyset} \mathcal{L}_{hkv}. $$
Note that $\mathcal{L}$ is a finite union of subspaces, each of codimension 1 in $\Rplus^V$. By Lemma \ref{lem:dim.spi}, $\Sigma(\pi)$ is full-dimensional and non-empty, so $\Sigma(\pi) \setminus (\mathcal{L} \cap \Sigma(\pi))$ is non-empty. Let $x_K$ be in this set. Write $V = A \cup H \cup L \cup U$ w.r.t. the context $\{X_K = x_K\}$. Since $x_K \in \Sigma(\pi)$, 
$$ \Gamma_{KK}\odot x_K \leq (\Gamma_{KK} \vee \Xi_K^\pi)\odot x_K < x_K. $$
Thus there are no pairs $h,k \in K$ with $h\neq k$ such that $x_h = c^\ast_{hk}x_k$. So $U = \emptyset$. In addition, $x_K \notin \mathcal{L}$. Thus by definition, $L = \emptyset$. So $K^\ast(X_K = x_K) = H$. Define the event $\E \subset \Rplus^V$ via
\begin{itemize}
	\item $Z_k = x_k$ for all $k \in K$
	\item $\max_{k \in K: c^\ast_{ik} > 0} c^\ast_{ik}x_k < Z_i < \min_{k \in K: c^\ast_{ki} > 0} {x_k}/{c^\ast_{ki}}$ for all remaining $i \in V$ (where the minimum over an empty set is $0$ and the maximum over an empty set is $\infty$). 
\end{itemize}
Since $U = \emptyset$, for each $i$, 
$\max_{k \in K: c^\ast_{ik} > 0} c^\ast_{ik}x_k < \min_{k \in K: c^\ast_{ki} > 0} {x_k}/{c^\ast_{ki}}$. Thus $\E$ is well-defined. By construction, $\E \subset \{X_K = x_K\}$ and it is full-dimensional w.r.t.\ this set. Thus there exists at least one $g \in \mathfrak{G}(X_K = x_K)$ such that $\E(g) \cap \E \neq \emptyset$. For this $g$, the only constant stars of $g$ have roots in $K$ and have no children. Thus $K^\ast(g) = K$. Since $K \subseteq K^\ast \subseteq K^\ast(g)$, it follows that $K^\ast = K$. Finally, since $x_K \in \Sigma(\pi)$, 
$$ \Xi_K^\pi \odot x_K \leq (\Gamma_{KK} \vee \Xi_K^\pi) \odot x_K < x_K. $$
So in particular, for each edge $j \to i \in \pi$, 
$$ \Xi_K^{ij}\odot x_K < x_K. $$
Since $K = K^\ast(X_K = x_K)$, \eqref{eqn:bij.kl.x} holds. 
 Lemma \ref{lem:e+.alg} then implies $j \to i \in E^+(X_K = x_K)$. 
\end{proof}

\begin{lemma}\label{lem:in.cxk}
Let $\pi$ be a $\ast$-connecting path in $\D^\ast_K(C)$ with $\Sigma(\pi) \neq \emptyset$. Let $x_K\in \Sigma(\pi)$ that satisfies the conclusion of Proposition \ref{lem:L.emptyset}. Then $\pi \subseteq \C(X_K = x_K)$. 
\end{lemma}
\begin{proof}
Fix an edge $j \to i$ of $\pi$ and $x_K$ as above. Since $\pi$ is $\ast$-connecting, $j \notin K$. 
There are two cases.\\
\textbf{Case 1. } $i \notin K$. We shall show that there exists some $g \in \mathfrak{G}(X_K = x_K)$ that contains the edge $j \to i$, and that this edge is not part of a constant star of $g$. To do this, we construct a region $\E$ in the manner similar to the proof of Theorem~\ref{thm:starse.faith}, case b(I). Apply \eqref{eqn:z.bound} to $j$ and let $\beta_{j}$ be the constant on the right-hand side of this inequality. That is,
$$ \beta_j = \min_{k \in K: k \in \ch^\ast_K(j)} \frac{x_k}{c^\ast_{kj}}. $$
Let
$$ \gamma_j = \frac{1}{c^\ast_{ij}}\max_{\ell \in \pa^\ast_K(i)} x_\ell c^\ast_{i\ell}. $$
Since $j \to i \in E^+(X_K = x_K)$ and $i \notin K$, by \eqref{eqn:bij.kl.x} we have
$$ (\xi^{ij}_K)_{k\ell}x_\ell < x_k $$
for all $k \in \ch_{\D^\ast}(j)$ and $\ell \in \pa_{\D^\ast}(i)$. Rearranging gives
$$ \frac{x_k}{c^\ast_{kj}} > \frac{x_\ell c^\ast_{i\ell}}{c^\ast_{ij}} \mbox{ for all } k \in \ch^\ast_K(j), \ell \in \pa^\ast_K(i),$$
or equivalently, $\beta_j > \gamma_j$.
For a sufficiently small constant $\epsilon>0$, consider the region $\E$ defined by
\begin{itemize}
  \item $\gamma_j < Z_j < \beta_j$  
  \item for each $h \in H$, set $Z_h = x_h$ 
  \item $Z_r < \epsilon$ for all other nodes. 
\end{itemize}
In the above, the only nodes that were mentioned but did not get set to be less than $\epsilon$ are $Z_j$ and $Z_h$ for $h \in H$. By Proposition \ref{cor:dag.partition} and \eqref{eqn:no.shared.ell}, these nodes are all distinct. Since $\beta_j > \gamma_j$, $Z_j$ is well-defined. Since $L = \emptyset$, $\E$ is a non-empty polyhedron in $\Rplus^V$, $\E \subseteq \{X_K = x_K\}$, and $\E$ is full-dimensional relative to the region $\{X_K = x_K\}$. Therefore, there exists at least one $g \in \mathfrak{G}(X_K = x_K)$ with $\E(g) \cap \E \neq \emptyset$. 
Now suppose $ Z \in \E$. Then $Z_j > \gamma_j$ implies
$$ c^\ast_{ij}Z_j > c^\ast_{i\ell}x_\ell $$
for all $\ell \in K$. In addition, $Z_j \gg \epsilon > Z_r, Z_i$ implies
$$  c^\ast_{ij}Z_j > c^\ast_{ir}x_r $$ 
for all $r \neq j, r \notin K$ such that $c^\ast_{ir} > 0$. Thus $R_g(i) = j$, so in particular, $j \to i \in g$. \\
Since $Z_j < \beta_j$, it follows that
$$ c^\ast_{ki}Z_j < x_k $$
for all $k \in K \cap \ch^\ast_K(j)$. Since $K^\ast = K$, $j \notin K^\ast(g)$, so $j \to i \notin E^-(X_K = x_K)$. Thus $j \to i \in \mathcal{C}(X_K = x_K)$. We are done. \\
\textbf{Case 2. } $i \in K$. Since $j \to i \in E^+(X_K = x_K)$, we can rearrange \eqref{eqn:bij.kl.x} to obtain
\begin{equation}\label{eqn:z.j}
\min_{k \in \ch_{G}(j), k \neq i} \frac{x_k}{c^\ast_{kj}} \geq \frac{x_i}{c^\ast_{ij}}. 
\end{equation}
As $L = \emptyset$, by definition of $L$, $x_K$ satisfies the stronger inequality
\begin{equation}\label{eqn:z.j.strict}
\min_{k \in \ch_{G}(j), k \neq i} \frac{x_k}{c^\ast_{kj}} > \frac{x_i}{c^\ast_{ij}}. 
\end{equation}
Since $i \in K$, it is sufficient to show that $j \to i \in \mathcal{I}(X_K = x_K)$. That is, we need to construct $g \in \mathfrak{G}(X_K = x_K)$ such that $j \to i \in g$. For a very small constant $\epsilon>0$, consider the region $\E_1$ defined by
\begin{itemize}
  \item for each $h \in H, h \neq i$, set $Z_h = x_h$
  \item $Z_j = {x_j}/{c^\ast_{ij}}$
  \item $Z_r < \epsilon$ for all other nodes. 
\end{itemize}
Since $L = \emptyset$, $\E$ is well-defined and is full-dimensional relative to the region $\{X_K = x_K\}$. By \eqref{eqn:z.j.strict}, $Z_j$ satisfies \eqref{eqn:z.bound}, so $\E \subset \{X_K = x_K\}$. Therefore, there exists some $g \in \mathfrak{G}(X_K = x_K)$ such that $\E(g) \cap \E_1 \neq \emptyset$. On $\E_1$, by construction, $j \to i \in g$ and the proof is complete.
\end{proof}

For given matrix $C$ we now have the following result. This implies that $\critsep$ is not fully complete w.r.t.\ conditional independence for a given matrix $C$, as not  all $*$-connecting critical paths may be effective, as seen in Example~\ref{ex:counter.1.3}.
\begin{theorem}[Context-free completeness]\label{thm:faith.C}
Let $X$ be a max-linear Bayesian network over a directed acyclic graph $\D = (V,E)$ with fixed coefficient matrix $C$.  It then holds that 
\[X_I\notci X_J\cd X_K.\]
 if and only if there is an effective $*$-connecting path in the critical DAG $\D^*_K(C)$.
\end{theorem}
\begin{proof}By Theorem~\ref{thm:starse.C} and Theorem~\ref{thm:starse.faith}, $X_I \notci X_J \cd X_K$ if and only if there exists some $i \in I, j \in J$, some possible $x_K$, and some $\ast$-connecting path $\pi$ between $i$ and $j$ such that $\pi \subseteq \C(X_K = x_K)$. Thus the  statement in the proposition is equivalent to the claim that \eqref{eqn:lambda.bpi} holds if and only if there exists $x_K$ such that $\pi \subseteq \C(X_K = x_K)$.

So suppose \eqref{eqn:lambda.bpi} holds. By Proposition~\ref{prop:lambda.1}(c), $\Sigma(\pi) \neq \emptyset$. 
By Proposition~\ref{lem:L.emptyset}, we can pick a special $x_K$. 
Applying Lemma \ref{lem:in.cxk} to this special $x_K$, we conclude that $\pi \subseteq \C(X_K = x_K)$.

For the converse, suppose $x_K$ is such that $\pi \subseteq \C(X_K = x_K)$. By Corollary \ref{cor:e.subset}, each edge of $\pi$ is in $E^+(X_K = x_K)$. By Proposition \ref{prop:lambda.less.than.1}, this implies that $\pi_K$ is effective. 
\end{proof}
\begin{remark}
We note that we now could define yet another form of separation which we could name \emph{effective $*$-separation} and by Theorem~\ref{thm:faith.C} this would now be strongly complete for conditional independence for a given $C$ or, in other words, \emph{any} max-linear Bayesian network would be faithful to this criterion. However, whereas checking critical separation is a simple task, checking effective $*$-separation would involve calculating $\lambda(\Gamma_{KK} \vee \Xi_K^\pi)$ for all $*$-connecting paths $\pi$, so we have preferred not to introduce this variant of separation in this article. 
\end{remark}

Finally we are able to establish full completeness of $\tildese$-separation for an unspecified coefficient matrix $C$.
\begin{theorem}[Completeness of $\tildese$-separation]\label{thm:genfaith}
Let $X$ be a max-linear Bayesian network over a directed acyclic graph $\D = (V,E)$ and assume there is a $*$-connecting path in $\D^\ast_K$ between $I$ and $J$. Then there is a coefficient matrix $C$ with support included in $\D$ such that the corresponding max-linear Bayesian network satisfies 
$$X_I \notci X_J \cd X_K.$$
\end{theorem}

\begin{proof}
Let $\pi$ be a $\ast$-connecting path in $\D^\ast_K$ between $I$ and $J$. For each of the five types, our goal is to construct a $C$ such that $\pi \subset \D^\ast_K(C)$ and that \eqref{eqn:lambda.bpi} holds, i.e.\ the path $\pi$ is effective.

For each edge $v \to u \in \pi$, let $\pi_{uv} \subset \D$ be a path in $\D$ from $v$ to $u$ that does not factor through $K$. Define $C=C(\pi)$ as follows. 
\begin{itemize}
    \item If $a \to b \in \bigcup_{v\to u \in \pi}\pi_{uv}$, set $c_{ba} = 1$
    \item Otherwise, set $c_{ba}$ to be some constant such that $c_{ba} < 1$. 
\end{itemize}
First we claim that for this choice of $C$, $\pi \subset \D^\ast_K(C)$. That is, for each edge $a \to b \in \pi$, no critical paths from $a$ to $b$ on $C$ factor through $K$. Indeed, fix such an edge $a \to b \in \pi$. Let $\pi'_{ba}$ be another path in $\D$. Then either $\pi'_{ba}$ contains an edge not in $\bigcup_{u \to v \in \pi}\pi_{vu}$, in which case
$$ c(\pi'_{ba}) < c(\pi_{ba}) = 1, $$
or that it only uses edges in $\bigcup_{v \to u \in \pi}\pi_{uv}$ and 
$$ c(\pi'_{ba}) = c(\pi_{ba}) = 1. $$
But in this case, since none of the paths $\pi_{vu}$ factor through $K$, $\pi'_{ba}$ does not factor through $K$. This establishes our first claim. 

We now prove \eqref{eqn:lambda.bpi}. %
Note that all relevant substitution matrices are formed by combining substitution matrices for single edges $\Xi_K^{ij}$ for $j\notin K$ and we now
claim that each entry of such a matrix is strictly less than 1. 

As shown  above, we must have  $c^\ast_{ij} = 1$. Let $k \in K \cap \ch_{\D^\ast}(j), \; \ell \in K \cap (\pa_{\D^\ast}(i)\cup \{i\}), k\neq \ell$ so we again have $c^\ast_{i\ell}=1$. 
Since $k \notin \pi$, any path in $\D$ from $j$ to $k$ must utilize an edge of $C$ whose weight is strictly less than 1 with the choice of $C$ made above. Thus $c^\ast_{kj} < 1$.
By \eqref{eqn:bij.kl}, 
$$ (\Xi_K^{ij})_{k\ell} = \frac{c^\ast_{kj}c^\ast_{i\ell}}{c^\ast_{ij}}= c^\ast_{kj}c^\ast_{i\ell}=c^\ast_{i\ell}< 1.  $$
So each entry of $\Xi_K^{ij}$ is strictly less than 1, as claimed, and hence this also holds for $\Xi_K^\pi$. Since $c_{uv} \leq 1$ for all edges $v \to u \in \D$, we have $\gamma_{uv}=c^\ast_{uv} \leq 1$ for all edges $v \to u \in \D^\ast$. Thus
$$ \lambda(\Gamma_{KK} \vee \Xi_K^\pi) \leq 1. $$
Suppose now for contradiction that $\lambda(\Gamma_{KK} \vee \Xi_K^\pi) = 1$. Since all entries of $\Xi_K^\pi$ are strictly less than 1, a critical cycle of $\Gamma_{KK} \vee \Xi_K^\pi$ must only involve edges in $\Gamma_{KK}$. But $\D$ is a DAG, so $\D(\Gamma_{KK})$ is cycle-free, yielding a contradiction. This concludes the proof. 
\end{proof}

\section{Outlook}\label{sec:outlook}
\subsection{Properties of max-linear independence}\label{sec:graphoid}
In the previous section we have defined two abstract independence models $\tildese$ and $\critsep$ in the sense of \citep{studeny:05} and showed that they are sound and the former of them is also complete, whereas the latter needs additional conditions for completeness. 

One can show without too much effort that these are both \emph{compositional graphoids} (we refrain from giving the details) as also holds for most other graphical separation criteria  (see e.g.\ \cite{lauritzen:sadeghi:18}).  
However, we should emphasize that $\tildese$ is \emph{not strongly complete} as the Diamond example shows: in this example the classical $d$-separation $\dse$ and $\tildese$ coincide and there is no single coefficient matrix $C$ such that  the corresponding max-linear Bayesian network is faithful to $\tildese$, i.e.\ in that case $\tildese$ is strictly weaker than critical separation $\critsep$, and the same will happen for DAGs with more than a single directed path between any two points.  But even in this case, the context-specific analysis typically yields further valid conditional independence statements. 

Generally, the study of properties of conditional independence for max-linear models opens up several new avenues: concerning e.g.\ Markov equivalence as in \cite{verma1991,frydenberg:90}, or the algebraic properties of \emph{maxoids} as an analogue of gaussoids; see, for example \cite{gaussoids}.

\subsection{Extensions and special cases}\label{sec:extensions}  We have so far in this article not discussed identifiability, estimation, or any other statistical issues associated with these models. These have been briefly considered in \cite{GKL}; see also \cite{nadinethesis}.
This work was extended to a recursive max-linear model with propagating noise in \cite{BK}, but we are not considering models with noise in this article.

Extreme value models often rely on regular variation and several publications have combined Bayesian networks with such heavy-tailed innovations. In \cite{GKO} and \cite{KK}, algorithms have been proposed for statistically learning the model based on the estimated tail dependence matrix and on a scaling method, respectively. In \cite{engelke:hitz:18} for undirected graphs the authors apply a peaks-over-threshold approach giving a multivariate generalized Pareto distribution for exceedances such that a density exists. For a decomposable graph, this density factorizes into lower dimensional marginal densities, whereas \cite{Davis3} deals with conditional densities.

Natural extensions in the framework of recursive max-linear models are based on making dependent innovations $(Z_1,\dots,Z_d)$, thus defining the analogue of classical path analysis (\citep{wright:21,wright:34}), or \emph{recursive causal models}; see 
\cite{kiiveri:speed:carlin:84}. The models introduced by \cite{engelke:hitz:18}  could be interesting candidates for this.

An alternative for an appropriate model may originate from multivariate max-stable Fr\'echet distributions with distribution function (see e.g. \cite{DHF}, Section~6.1.4 and in particular Remark~6.1.16 with parametrization given in Theorem~1.1.3)
$$F(z) = \exp\Big\{-\int_{\mathcal S^{d-1}} \bigvee_{1\le i\le d} \frac{\omega_i}{z_i} \, \Theta(d\omega)\Big\},\quad z=(z_1,\dots,z_d),$$
where $\omega=(\omega_1,\dots,\omega_d)\in\mathcal S^{d-1}$, the unit sphere in $\R_+^d$ (with respect to any norm), and $\Theta$ is a finite measure on $S^{d-1}$, called the spectral measure. Then the innovation vector $(Z_1,\dots,Z_d)$ has Fr\'echet margins with algebraically decreasing tails.
If the spectral measure has a Lebesgue density, then the above integral becomes a Lebesgue integral. 
Then a large jump can happen in every direction with some probability.
The Bayesian network introduces additional dependence into the model, which directs the large jumps in special directions.
\subsection{Some open problems}\label{sec:open}

Proposition~\ref{cor:dag.partition}
gives some necessary conditions for a graph to be the source DAG for some context $\{X_K = x_K\}$. It would be of interest to know whether these are also sufficient. Formally this is stated as Problem~\ref{prob:source} below:
\begin{problem}\label{prob:source}
Fix $\D$. Find a characterization for all possible source DAGs. 
\end{problem}
Further, even though we have a full characterization of situations with conditional independence, there is still an issue about how to verify conditional independence from a computational point of view. Formally, we state this as
\begin{problem}\label{prob:compute.C}
Give an efficient algorithm to compute the source DAG $\mathcal{C}(X_K = x_K)$ and analyze its complexity.
\end{problem}

Critical directed paths in a graph can be computed with tropical matrix multiplication \cite[\S 3]{But2010}, and thus $\D^\ast_K$ and $\D^\ast_K(C)$ can both easily be computed in time at most $O(d^4)$. However, computing the source DAG is harder. A straight-forward algorithm using the characterization of the impact graphs $\mathfrak{G}(X_K = x_K)$ in Lemma \ref{lem:y.xk} goes as follows.
\begin{enumerate}
  \item Enumerate all elements in $\mathfrak{G}(X_K = x_K)$ using the system of equations and inequalities given in Theorem \ref{thm:impact} and Lemma \ref{lem:y.xk}, with $K^\ast(g)$ characterized by Lemma \ref{lem:k.ast.y}.
  \item Compute $K^\ast = K^\ast(X_K = x_K)$ from $\mathfrak{G}(X_K = x_K)$ via Definition \ref{defn:fixed}. 
  \item Compute the source DAG via Theorem~\ref{thm:source.rep}. 
\end{enumerate}
Of these steps, step 1 is the most computationally intensive. The set $\mathfrak{G}(X_K = x_K)$ represents all possible hitting scenarios in \cite{Wang2011}. For general $C$ (not necessarily supported on a DAG), \cite{Wang2011} noted that enumerating $\mathfrak{G}(X_K = x_K)$ is related to the NP-hard set covering problem. For our case, $C$ is a DAG, so we were able to characterize $\mathfrak{G}(X_K = x_K)$ in much greater detail than \cite{Wang2011}. 
However, it is unclear what is the complexity of enumerating this set.  The difficulty is that the inequalities corresponding to \eqref{eqn:Lambda}, \eqref{eqn:extra1} and \eqref{eqn:extra2} depend on $g$. So while it is easy to check whether $g \in \mathfrak{G}(X_K = x_K)$ for a given $g$, there are exponentially many impact graphs $g$ one needs to consider. 

We remark that Problem \ref{prob:compute.C} can be seen as finding the tropical analogue of Gaussian elimination. While there has been work on the tropical Fourier-Motzkin elimination \cite{allamigeon2014tropical}, we are not aware of algorithms to solve tropical Gaussian elimination. The geometric relative of this problem is to find minimal external representations of tropical polyhedra, to which algorithms and characterizations in terms of hypergraphs have been developed e.g.\ in \cite{allamigeon2013computing, allamigeon2011tropical, allamigeon2013minimal}. It would be interesting to deepen these connections between extreme value theory and tropical convex geometry.
A related problem is the following:
\begin{problem}\label{prop:efficient}
Give an efficient algorithm to simulate from the conditional distribution of $X$, given a context $\{X_K = x_K\}$.
\end{problem}
This problem was also considered by \cite{Wang2011} and has particular interest for Bayesian inference about the unknown parameters of a max-linear Bayesian network. Most Markov chain Monte Carlo (MCMC) algorithms will have such a simulation step built in at some point. In addition, this could be of interest if an unknown source for an observed extreme event should be identified, potentially of interest in environmental science.

\section*{Acknowledgements}Carlos Am\'endola was partially supported by the Deutsche For\-schungs\-ge\-mein\-schaft (DFG) in the context of the Emmy Noether junior research group KR 4512/1-1. Steffen Lauritzen has benefited from financial support from the Alexander von Humboldt Stiftung. Ngoc Tran would like to thank the Hausdorff Center for Mathematics for  making her summer research visits to Germany possible. 



\end{document}